\documentclass[11pt]{article}
\usepackage{geometry}
\usepackage{amssymb, amsmath, amsthm, color, tikz,enumerate,tikz-3dplot,mathrsfs}
\usepackage{pdfpages}
\usepackage{hyperref}
\usepackage{textcomp}
\usepackage{color}
\usepackage{xcolor}
\usepackage{listings}
\usepackage{todonotes}
\usepackage{mathabx}

\usepackage[backend=bibtex,style=alphabetic,maxbibnames=99,minalphanames=3,doi=false,isbn=false,url=false,eprint=false,sorting=nyt]{biblatex}
\renewcommand*{\intitlepunct}{\space}
\usepackage{biblatex}
\renewbibmacro{in:}{%
  \ifentrytype{article}{}{\printtext{\bibstring{in}\intitlepunct}}}

\bibliography{TS2_revision}
 
\usepackage{fullpage}

\newtheorem{sublemma}{Lemma}[subsection]
\newtheorem{subprop}[sublemma]{Proposition}
\newtheorem{subthm}[sublemma]{Theorem}
\newtheorem{subcor}[sublemma]{Corollary}

\newtheorem{corollary}{Corollary}[section]

\newtheorem*{conditionA}{Condition (A)}

\theoremstyle{definition}
\newtheorem{subdef}[sublemma]{Definition}
\theoremstyle{remark}
\newtheorem{subrmk}[sublemma]{Remark}
\newtheorem{subex}[sublemma]{Example}

\makeatletter
\newcommand{\setword}[2]{%
  \phantomsection
  #1\def\@currentlabel{\unexpanded{#1}}\label{#2}%
}
\makeatother

\numberwithin{equation}{section}

\newcommand{\Ccc}{\mathbb{C}}

\newcommand{\Rrr}{\mathbb{R}}

\newcommand{\Qqq}{\mathbb{Q}}
\newcommand{\Zzz}{\mathbb{Z}}
\newcommand{\R}{\Rrr}
\newcommand{\Q}{\Qqq}
\newcommand{\Z}{\Zzz}

\DeclareMathOperator{\sign}{sign}
\DeclareMathOperator{\Span}{Span}

\DeclareMathOperator{\rk}{rank}
\DeclareMathOperator{\Gal}{Gal}

\newcommand{\coveredby}{\prec}

\newcommand{\interior}[1]{#1^{\circ}}

\newcommand{\id}{\operatorname{id}}

\newcommand{\cupdots}{\cup \cdots \cup}

\newcommand{\Qn}{\Pi^{\text{ord}}_{n}}
\newcommand{\Qns}{\Pi^{\text{ord},B}_{n}}

\newcommand{\vanish}[1]{}

\newcommand{\fl}{\longrightarrow}

\newcommand{\iso}{\stackrel{\sim}{\fl}}
\newcommand{\la}{\langle}
\newcommand{\ra}{\rangle}
\newcommand{\bl}[2]{\langle #1, #2 \rangle}
\newcommand\til[1]{\widetilde{#1}}
\newcommand{\GL}{\mathbf{GL}}

\newcommand\Cf{\mathcal{C}}
\newcommand\F{\mathcal{F}}
\newcommand\Hf{\mathcal{H}}
\newcommand\Kf{\mathscr{K}}
\newcommand\Lf{\mathscr{L}}
\newcommand\Mf{\mathscr{M}}
\newcommand\Of{\mathscr{O}}
\newcommand\Pf{\mathscr{P}}
\newcommand\Tc{\mathscr{T}}
\newcommand\Tt{\mathcal{T}}

\newcommand\spam[1]{\langle#1\rangle}

\begin{document}

\title{
A generalization of combinatorial identities\\ for stable discrete series
constants}

\author{{\sc Richard EHRENBORG}\footnote{\noindent
{\sc University of Kentucky,
Department of Mathematics,
Lexington, KY 40506, USA.} \hfill\break
{\tt richard.ehrenborg@uky.edu}.},
        {\sc Sophie MOREL}\footnote{
\noindent
{\sc 
UMPA UMR 5669 CNRS,
ENS de Lyon site Monod,
46, all\'ee d'Italie,
69364 Lyon Cedex 07, France.}\hfill\break
{\tt sophie.morel@ens-lyon.fr}.
}
and        
        {\sc Margaret READDY}\footnote{
\noindent
{\sc University of Kentucky,
Department of Mathematics,
Lexington, KY 40506, USA.}  \hfill\break
{\tt margaret.readdy@uky.edu}.
}
}

\date{\today.}

\maketitle

\begin{abstract}
This article is concerned with the constants that appear
in Harish-Chandra's character formula for
stable discrete series of real reductive groups, 
although it does not require any knowledge about real reductive groups
or discrete series.
In Harish-Chandra's work the only information we have about
these constants is that they are uniquely determined by an inductive property.
Later Goresky--Kottwitz--MacPherson and Herb gave different formulas for these constants;
see~\cite[Theorem~3.1]{GKM} and~\cite[Theorem~4.2]{Herb-2S}.
In this article we generalize these formulas to the case of arbitrary finite Coxeter groups
(in this setting, discrete series no longer make sense), and
give a direct proof that the two formulas agree. We actually prove a slightly more
general identity that also implies the combinatorial identity underlying the
discrete series character identities of
Morel~\cite[Proposition~3.3.1]{Morel}.
We deduce this identity from a general abstract theorem
giving a way to calculate the alternating sum of the values of a
valuation on the chambers of a Coxeter arrangement.
We also introduce a ring structure on the set of valuations on
polyhedral cones in Euclidean space with values in a fixed ring.
This gives a theoretical framework
for the valuation appearing in~\cite[Appendix~A]{GKM}.
In Appendix~\ref{appendix_2_structures}
we extend the notion of $2$-structures (due to Herb) 
to pseudo-root systems.
\end{abstract}

\tableofcontents

\section{Introduction}

Although this paper deals exclusively with the combinatorics of real
hyperplane arrangements and Coxeter complexes, it has its origin
in the representation theory of real reductive groups and its connections
with the cohomology of locally symmetric spaces, and in particular, of
Shimura varieties. We start by explaining some of this background.
This explanation can be safely skipped by the reader not interested in
Shimura varieties.

\vspace{.5cm}

Let $G$ be an algebraic group over $\Q$. To simplify the exposition, we assume
that $G$ is connected and
semisimple. Let $K_\infty$ be a maximal compact subgroup of
$G(\R)$ and $K$ be an open compact subgroup of $G(\mathbb{A}^\infty)$,
where $\mathbb{A}^\infty=\widehat{\Z}\otimes_\Z\Q$ is the ring of finite
ad\`eles of $\Q$. We consider the double quotient $X_K=G(\Q)\backslash(G(\R)\times
G(\mathbb{A}^\infty))/(K_\infty\times K)$. This is a real analytic
variety for $K$ small enough, and the projective system $(X_K)_{K\subset
G(\mathbb{A}^\infty)}$ has an action of $G(\mathbb{A}^\infty)$ by
Hecke correspondences that induces an action of the Hecke algebra at level
$K$ on the cohomology of $X_K$ for any reasonable cohomology theory.

We restrict our attention
further to the case where the real Lie group $G(\R)$ has
a discrete series. This is the so-called ``equal rank case''
because it occurs if and only if the groups $G(\R)$ and $K_\infty$ have
the same rank. Then the $L^2$-cohomology $H^*_{(2)}(X_K)$ is finite-dimensional,
and Matsushima's formula,
proved in this generality by
Borel and Casselman~\cite{Borel_Casselman},
gives a description of this cohomology and
of its Hecke algebra action in terms of
discrete automorphic representations of~$G$ whose infinite component is
a cohomological representation of~$G(\R)$, and in particular, either a discrete
series or a special type of non-tempered representation.

Another cohomology
of interest in this case is the intersection cohomology $IH^*(\overline{X}_K)$
of the minimal
Satake compactification $\overline{X}_K$ of $X_K$. 
In order to study this cohomology,
Goresky, Harder and MacPherson introduced in~\cite{GHM} a family of cohomology
theories called ``weighted cohomologies'' and showed that the two middle
weighted cohomologies agree with $IH^*(\overline{X}_K)$
if~$X_K$ has the structure of a complex algebraic variety.
This result was later generalized by Saper in~\cite{Saper}. 

All the cohomology theories
that we discussed have actions of the Hecke algebra, and the isomorphism
of the previous paragraph
is equivariant for this action. Zucker
conjectured that there should be a Hecke-equivariant isomorphism between
$H^*_{(2)}(X_K)$ and $IH(\overline{X}_K)$. This conjecture was
proved by Looijenga~\cite{Loo}, Looijenga--Rapoport~\cite{LooRap} and
Saper--Stern~\cite{SaSt} if $X_K$ has the structure of a complex
algebraic variety and by Saper~\cite{Saper} in general. In particular,
by comparing the formulas for the action of a Hecke operator
on weighted cohomology (this was calculated by Goresky and MacPherson
using topological methods
in~\cite{GM}) and on $L^2$-cohomology (this was calculated by Arthur using
the Arthur--Selberg trace formula in~\cite{A-L2}), one can obtain a formula for
averaged discrete series characters of the group $G(\R)$. One of the
goals of the paper~\cite{GKM} of Goresky--Kottwitz--MacPherson was to
prove this identity directly. 

If moreover the space $X_K$ is the set of complex points of
a Shimura variety, then it descends to an algebraic variety over an explicit
number field $E$ known as the reflex field, as does the minimal
Satake compactification, and so the intersection cohomology has a
natural action of the absolute Galois group $\Gal(\overline{E}/E)$.
We can further complicate the calculation by trying to calculate the
trace on $IH^*(\overline{X}_K)$ of 
Hecke operators twisted
by elements of the group $\Gal(\overline{E}/E)$,
for example, powers of Frobenius maps. 
In the case where $X_K$ is a Siegel modular variety, this was done
by the second author
in~\cite{Morel}.
It requires a slightly different character identity for averaged
discrete series characters of~$G(\R)$, 
also involving discrete series characters of the
endoscopic groups of $G$, and whose relationship with the
Goresky--Kottwitz--MacPherson identity was not clear.

\vspace{.5cm}

For the specialists,
we give a more detailed explanation of the relevance of our main results to
cohomology calculations in Appendix~\ref{appendix_more_detail}.
Let us return here to a discussion of the current article.

\vspace{.5cm}

In a previous article of the authors~\cite{Ehrenborg_Morel_Readdy}, we investigate
the character identity of Morel~\cite{Morel}. In particular we relate it
to the geometry of the Coxeter complex of the symmetric group and give
a simpler and more natural proof than the brute force calculation in
the appendix of~\cite{Morel}.
The goal of the present article is to generalize the approach
of~\cite{Ehrenborg_Morel_Readdy} and to prove a combinatorial
identity (Theorem~\ref{thm_second_main}) that implies the character formulas
of
Goresky--Kottwitz--MacPherson~\cite{GKM} and of Morel~\cite{Morel} 
(see Subsections~\ref{section_Coxeter_case} and~\ref{second_application}).
To obtain the character formula of~\cite{GKM} from our results, we need
to use Herb's formula for
averaged discrete series characters
(see for example~\cite{Herb} and~\cite{Herb-2S}).
We also generalize, in Corollary~\ref{cor_shelling_order} and
Lemma~\ref{lemma_Coxeter_A},
the geometric result of~\cite{Ehrenborg_Morel_Readdy}
(see Theorem~4.3 of that article). 
In fact, we prove an identity
that holds
not just for root systems that are
generated by strongly orthogonal roots, but 
for all Coxeter systems with finite Coxeter group.
The representation-theoretic
interpretation of our identity in the general case is still unclear.

\vspace{.5cm}

We now describe in more detail the different sections of the article.

\vspace{.5cm}

In Section~\ref{section_hyperplane_arrangements} we review some
background material about real hyperplane arrangements and 
Coxeter arrangements.

In Section~\ref{section_first_main_theorem} we prove our first
main theorem (Theorem~\ref{theorem_2_structures_and_chambers})
that concerns the calculation over the chambers $T$ of a Coxeter
arrangement $\Hf$ of the alternating sum of quantities $f(T)$, where
$f$ is a valuation defined on closed convex polyhedral cones.
More precisely, Theorem~\ref{theorem_2_structures_and_chambers} reduces
this calculation to a similar calculation for simpler subarrangements
of $\Hf$ and it is the main ingredient in the proof of our second
main theorem (Theorem~\ref{thm_second_main}). The original proof
of Theorem~\ref{thm_second_main} used an induction similar to the
ones used in the proofs of the character identities of~\cite[Theorem~3.1]{GKM}
and~\cite[Theorem~4.2]{Herb-2S}, 
but we later realized that Theorem~\ref{thm_second_main}
was a particular case of the more general identity of
Theorem~\ref{theorem_2_structures_and_chambers}.

In Section~\ref{section_definition_sum}
we state and prove our second main theorem (Theorem~\ref{thm_second_main}).
We first introduce in Subsection~\ref{section_weighted_complex}
our main geometric construction, which we call 
the \emph{weighted complex},
that allows us to define the \emph{weighted sum};
see Remark~\ref{rmk_name_weighted_sum} for an explanation of these names.
The weighted complex is the set of all the faces of a fixed hyperplane
arrangement that are on the nonnegative side of an auxiliary hyperplane
$H_\lambda$. It contains what is known as the \emph{bounded complex} in the
theory of affine oriented matroids, and coincides with it if $H_\lambda$ is
in general position. 
We state Theorem~\ref{thm_second_main} in 
Subsection~\ref{section_second_main_thm}
and prove it in Subsection~\ref{section_proof_second_main_theorem}.
The proof is straightforward:
Using Corollary~\ref{cor_second_convolution}, which
generalizes~\cite[Proposition~A.4]{GKM}, to reinterpret the
weighted sum as an alternating sum on the chambers of the arrangement
of the value of a particular valuation,
we are able to show
that Theorem~\ref{thm_second_main} is a particular case of
Theorem~\ref{theorem_2_structures_and_chambers}.
In Subsections~\ref{section_Coxeter_case}
and~\ref{second_application}
we explain how Theorem~\ref{thm_second_main}
implies the identities of~\cite[Theorem~3.1]{GKM}
and of~\cite[Theorem~6.4]{Ehrenborg_Morel_Readdy}.

In Section~\ref{section_weighted_complex_shellable}, 
we study the geometric properties of the weighted complex.
We prove in particular that, under a
hypothesis about the dihedral
angles between the hyperplanes of the arrangement
(Condition~\ref{(A)}
in Subsection~\ref{section_condition_A},
which always holds in the Coxeter case), the
weighted complex is shellable; see Corollary~\ref{cor_shelling_order}, which
generalizes Theorem~4.3 of~\cite{Ehrenborg_Morel_Readdy}. 
We consider the case of Coxeter arrangements in
Subsection~\ref{section_weighted_complex_Coxeter}.
These geometric results were originally needed in the proof
of Theorem~\ref{thm_second_main}, but the new proof
via Theorem~\ref{theorem_2_structures_and_chambers} allows us
to circumvent them. We nevertheless decided to keep them in the article
because we thought that they could be of independent interest.

In Section~\ref{section_concluding_remarks}
we include concluding remarks.

We finish with three appendices. Each of the
first two appendices can be read independently
from the rest of the article (except that a proof
in Appendix~\ref{appendix_valuations}
uses Lemma~\ref{lemma_star}).
The goal of our Appendix~\ref{appendix_valuations}
is to generalize~\cite[Proposition~A.4]{GKM},
which is a key part in the proof of our main
theorem. In Appendix~A of their article~\cite{GKM}, 
Goresky--Kottwitz--MacPherson show that a
certain function, which they call $\psi_C(x,\lambda)$,
is a \emph{valuation} (see Definition~\ref{def_valuation})
on closed convex polyhedral cones, although they do not phrase it in these
terms. We show that their function is a special case of a general
construction that takes two valuations and produces a third one,
and that this operation makes the set of valuations on closed
convex polyhedral cones into a ring.
See Theorem~\ref{thm_coalgebra} and its corollaries for the
precise definition of this operation.

In Appendix~\ref{appendix_2_structures} we review the theory
of $2$-structures, due to Herb; see for example Herb's review
article~\cite{Herb-2S}.
We believe that this will be useful to the reader for a number of
reasons.
The proofs of the fundamental
results of this theory are somewhat scattered in the literature and
sometimes left as exercises. Furthermore, we needed to slightly
adapt a number of results so that they continue to hold 
for Coxeter systems that do not necessarily arise from a 
(crystallographic) root system.

Finally, Appendix~\ref{appendix_more_detail} is a continuation of
the first part of the introduction, and is intended to give specialists
more information about the way the weighted sum of
Definition~\ref{def_psi} and Theorem~\ref{thm_second_main}
appear in the calculation of the cohomology of locally symmetric
varieties.

\section{Hyperplane arrangements}
\label{section_hyperplane_arrangements}

\subsection{Background material}
\label{background}

We fix a finite-dimensional $\R$-vector space $V$ with an inner product
$(\cdot,\cdot)$. 
If $\alpha\in V$, we write
\begin{align*}
H_\alpha & = \{x\in V : (\alpha,x)=0\}, &
H^+_\alpha & = \{x\in V : (\alpha,x)> 0\}, &
H^-_\alpha & = \{x\in V : (\alpha,x)< 0\}.
\end{align*}
We also denote
by $s_\alpha$ the (orthogonal) reflection across the hyperplane $H_\alpha$.

Let $(\alpha_e)_{e \in E}$ be a finite family of nonzero vectors in $V$. 
The corresponding 
\emph{(central) hyperplane arrangement} is the family of hyperplanes
$\Hf = (H_{\alpha_e})_{e \in E}$.
Let $V_{0}$ be the intersection of all the hyperplanes,
that is, $V_0=\bigcap_{e \in E} H_{\alpha_e}$.
We say that the arrangement $\Hf$ is \emph{essential}
if $V_0 = \{0\}$, which means that the family
$(\alpha_e)_{e \in E}$ spans $V$.

Consider the map $s:V \longrightarrow \{+,-,0\}^E$ sending $x\in V$ to the family
$(\sign((\alpha_e,x)))_{e \in E}$, where $\sign:\R \longrightarrow \{+,-,0\}$ is the
map sending positive numbers to $+$, negative numbers to $-$ and
zero to $0$. 

\begin{subrmk}
The image of the map $s:V \longrightarrow \{+,-,0\}^E$
is the set of covectors of an oriented matroid
(see for example~\cite[Definition~4.1.1]{OM}). This is the oriented matroid
corresponding to the hyperplane arrangement. 
In fact, some of our results extend to general oriented matroids.
In this article we have chosen to concentrate
on hyperplane arrangements to keep the exposition more concrete.
In particular, we do not assume that the reader knows what an oriented
matroid is.
\end{subrmk}

We denote by $\Lf(\Hf)$ or just $\Lf$
the set of nonempty subsets of $V$ of the form
$C=s^{-1}(X)$, for a sign vector $X\in\{+,-,0\}^E$.
The elements of $\Lf$ are called \emph{faces} of the arrangement.
The set $\Lf$ has a natural partial order
given by $C \leq D$ if and only if $C \subseteq \overline{D}$.
The relation $C \leq D$ is equivalent to the fact that
for every $e\in E$ we have $s(C)_e=0$ or $s(C)_e=s(D)_e$.
The set $\Lf$ with this partial order
is called the \emph{face poset} of the arrangement.
Note that $V_{0}$ is the minimal element of~$\Lf$.
When we adjoin a maximal element~$\widehat{1}$ to the poset~$\Lf$,
we obtain a lattice $\Lf \cup \{\widehat{1}\}$
known as the \emph{face lattice}.
Note that under our convention
faces other than $V_0$ are not closed subsets of $V$:
for every $C\in\Lf$, the closure $\overline{C}$ is a closed convex polyhedral
cone in~$V$, and it is an intersection of closed half-spaces~$\overline{H^{\pm}_{\alpha_e}}$.
The poset $\Lf$
is graded with the rank of a face $C\in\Lf$ given by $\rho(C) = \dim(C)-\dim(V_0)$,
where we write $\dim(C)$ for $\dim(\Span(C))$.

We denote by $\Tc(\Hf)$ or just $\Tc$ the set of maximal faces of $\Lf$.
These elements are often called \emph{chambers},
\emph{regions} or \emph{topes},
and are the connected components of $V-\bigcup_{e\in E}H_{\alpha_e}$.
If $T \in \Tc$ then $T$ is an open
subset of~$V$,
and its closure is a closed convex polyhedral cone of dimension~$\dim(V)$.

If $X,Y\in\{+,-,0\}^E$, their \emph{composition} $X\circ Y$
is the sign vector defined by
\[(X\circ Y)_e = \begin{cases}X_e & \text{ if }X_e \neq 0,\\
Y_e & \text{ otherwise.}\end{cases} \]
If $C,D\in\Lf$ then $s(C)\circ s(D)$ is also the image of a face of
$\Lf$, and we denote this face by $C\circ D$.
This is the unique
face of $\Lf$ that contains all vectors of $V$ of the form
$x+\varepsilon y$, with $x\in C$, $y\in D$ and
$\varepsilon>0$ sufficiently small (relative to $x$ and $y$).
Define the
\emph{separation set} of $C$ and $D$ to be the set
\[S(C,D)=\{e \in E :  s(C)_e=-s(D)_e \neq 0\}.\]
This is the set of $e \in E$ such that $C$ and $D$ are on different sides
of the hyperplane $H_{\alpha_e}$.

Fix a chamber $B\in\Tc$. We can then define a partial order $\preceq_B$ on
$\Tc$ by declaring that $T\preceq_B T'$ if and only if $S(B,T) \subseteq S(B,T')$.
The resulting poset is called the \emph{chamber poset with base chamber~$B$}.
We will denote it by $\Tc_B$. 
It is a poset with minimal element~$B$ and maximal element~$-B$.
When all the hyperplanes are distinct,
this poset is also graded
with the rank function $\rho(T) = |S(B,T)|$;
see~\cite[Proposition~4.2.10]{OM}.

If the choice of the base chamber $B$ is understood, we write, for
every face $C$ of the arrangement,
\[(-1)^C=(-1)^{|S(B,C\circ B)|}.\]

We also consider the graph with vertex set $\Tc$, where two
chambers $T,T'\in\Tc$ are connected by an edge if and only if
$\overline{T}\cap\overline{T}'$ spans a hyperplane (necessarily one of
the hyperplanes $H_{\alpha_e}$).
In this situation, we say that this hyperplane is a \emph{wall} of
the chambers $T$ and $T'$.
This graph is called the \emph{chamber graph}.
In the case
when all the hyperplanes of the arrangement $\Hf$ are distinct,
the distance between two chambers $T$ and $T'$ in this graph is $|S(T,T')|$;
see~\cite[Proposition~4.2.3]{OM}.

Consider the sphere $\mathbb{S}$ of center $0$ and radius $1$ in $V/V_0$.
The intersections $\overline{C}\cap\mathbb{S}$, for $C\in\Lf$, form a regular
cell decomposition $\Sigma(\Lf)$ of $\mathbb{S}$, and we will identify $\Lf$
with the face poset of this regular cell decomposition.

Finally, we recall the definition of the star of a face in $\Lf$.

\begin{subdef}
Let $C\in\Lf$. The \emph{star} of $C$ in $\Lf$ is 
$\{D\in\Lf :  C\leq D\}$.
Geometrically it is the set of faces
of $\Lf$ whose closure contains $C$. We will denote it by $\Lf_{\geq C}$.
\end{subdef}

\begin{sublemma}
\label{lemma_star}
Let $C\in\Lf$ and let $E(C)=\{e\in E :  C\subset H_{\alpha_e}\}$. Consider the
hyperplane arrangement $\Hf(C)=(H_{\alpha_e})_{e\in E(C)}$ and let $\Lf_{\Hf(C)}$
be its face poset.
Then the following four statements hold:
\begin{itemize}
\item[(i)]
Each face $D$ of $\Lf$ is contained in a unique
face $D'$ of $\Lf_{\Hf(C)}$, and the map $D \longmapsto D'$ induces an isomorphism of
posets
$\iota_C: \Lf_{\geq C} \stackrel{\sim}{\longrightarrow} \Lf_{\Hf(C)}$.
In particular, it sends the chambers of
$\Tc\cap\Lf_{\geq C}$ to the chambers of $\Lf_{\Hf(C)}$.

\item[(ii)]
The isomorphism $\iota_C$ of (i) sends a face $D\geq C$ of
$\Hf$ to the relative interior of the closed convex
polyhedral cone $\overline{D}+\Span(C)$.
Let $\Cf_C=\bigcap_{e\in E-E(C)}H_e^{\epsilon_e}$, where
$\epsilon_e=s(C)_e$.
The inverse of the isomorphism $\iota_C$ sends a face $D'$ of $\Hf(C)$ to
the intersection $D'\cap\Cf_C$.

\item[(iii)]
If $D_1,D_2\in\Lf_{\geq C}$ then the inclusion $S(D_1,D_2)\subset E(C)$ holds.
In particular, we have the equality $S(D_1,D_2)=S(\iota_C(D_1),\iota_C(D_2))$,
where the isomorphism $\iota_C$ is as in (i). 

\item[(iv)] The isomorphism $\iota_C:\Lf_{\geq C} \stackrel{\sim}{\longrightarrow} 
\Lf_{\Hf(C)}$ preserves composition and dimension, that is, for all
$D,D'\in\Lf_{\geq C}$,
the identities
$\iota_C(D \circ D') = \iota_C(D)\circ\iota_C(D')$ 
and
$\dim(\iota_C(D)) = \dim(D)$ 
hold.
\end{itemize}

\end{sublemma}

In particular, $\Lf_{\geq C}$ is also isomorphic to the face poset
of a regular cell decomposition of the unit sphere in
$V/\bigcap_{e \in E(C)} H_{\alpha_e}$
that we denote by $\Sigma(\Lf_{\geq C})$.

\begin{proof}[Proof of Lemma~\ref{lemma_star}.]
Statement~(i) is clear.

We prove statement~(ii).
Let $D\in\Lf_{\geq C}$ and let $D'=\iota_C(D)$. As
$C\leq D$, we have $s(D)_e=s(C)_e$ for every $e\in E-E(C)$, and
so $D\subset\Cf_C$. As $D\subset D'$, we deduce that $D'\cap\Cf_C\supset D$.
As $D'$ is a face of $\Hf(C)$, hence an intersection
$\bigcap_{e\in E(C)}H_{\alpha_e}^{s_e}$ with $s_e\in\{0,+,-\}$,
the intersection $D'\cap\Cf_C$ is either empty or a face of
$\Hf$. We have just proved that this intersection contains $D$,
so it is not empty and hence is equal to the face $D$ of $\Hf$.
It remains to prove that $D'$ is the relative interior of $\overline{D}+
\Span(C)$. We write $D=
\bigcap_{e\in E_0}H_{\alpha_0}\cap\bigcap_{e\in E_+}
H^+_{\alpha_e}\cap\bigcap_{e\in E_-}H^-_{\alpha_e}$, 
with $E=E_0\sqcup E_+\sqcup
E_-$. We then have $E_0\subset E(C)$, and $D'$ is equal to
$\bigcap_{e\in E_0}H_{\alpha_0}\cap\bigcap_{e\in E_+\cap E(C)}
H^+_{\alpha_e}\cap\bigcap_{e\in E_-\cap E(C)}H^-_{\alpha_e}$. So it suffices
to show that $\overline{D}+\Span(C)=K$, where
$K=\bigcap_{e\in E_0}H_{\alpha_0}\cap\bigcap_{e\in E_+\cap E(C)}
\overline{H}^+_{\alpha_e}\cap\bigcap_{e\in E_-\cap E(C)}\overline{C}H^-_{\alpha_e}$.
We clearly have $\overline{D}\subset K$ and $\Span(C)\subset K$,
so $\overline{D}+\Span(C)\subset K$. Conversely, let $x\in K$ and
let $y\in C$. Then $(\alpha_e,y) \neq 0$ for every $e\in E-E(C)$,
so there exists $\lambda>0$ such that $(\alpha_e,\lambda y)+(\alpha_e,x)$
has the same sign as $(\alpha_e,y)$ for every $e\in E-E(C)$.
We then have $\lambda x+y\in\overline{D}$, and so $x\in\overline{D}+
\Span(C)$.

We prove~(iii). Let $D_1,D_2\in\Lf_{\geq C}$, and let $e\in S(D_1,D_2)$.
Suppose for example that $s(D_1)_e=+$ and $s(D_2)_e=-$. (The other case
is similar.) Then $\overline{D}_1\subset\overline{H_{\alpha_e}^+}$
and $\overline{D}_2\subset\overline{H_{\alpha_e}^-}$, so
$C \subset \overline{D}_1\cap\overline{D}_2
\subset \overline{H_{\alpha_e}^+} \cap \overline{H_{\alpha_e}^-}=
H_{\alpha_e}$, which implies that $e\in E(C)$.

The first statement of (iv)
follows easily from the definitions: the composition
$D\circ D'$ is defined on the sign vectors of $D$ and $D'$, and the
isomorphism $\iota_{C}$ just forgets the coordinates outside of~$E(C)$
in these sign vectors.

We prove the second statement of (iv). 
Let $D\in\Lf_{\geq C}$, and let $D'$ be the unique
face of $\Lf_{\Hf(C)}$ containing~$D$. We clearly have $\dim(D)\leq\dim(D')$.
If $\dim(D')>\dim(D)$ then there exists $e\in E$ such that
$D\subset H_{\alpha_e}$ and $D'\not\subset H_{\alpha_e}$. But
$C\subset\overline{D}$, so this implies that $e\in E(C)$.
As $D'$ is not included in~$H_{\alpha_e}$, it must be contained in
one of the open half-spaces $H^{\pm}_{\alpha_e}$, contradicting the fact
that $D'$ contains~$D$.
\end{proof}

\begin{subrmk}
\label{rmk_star_convex}
Let $C\in\Lf$ and let $F'=\{e\in E :  C\not\subset H_{\alpha_e}\}$.
Then the set $\Tc\cap\Lf_{\geq C}$ is equal to
$\{T\in\Tc : \forall e\in F'\ s(T)_e=s(C)_e\}$, so it is a $T$-convex subset of $\Tc$ in the
sense of~\cite[Definition~4.2.5]{OM}; see~\cite[Proposition~4.2.6]{OM}.
In other words, it contains every shortest path in the
chamber graph between any two of its elements, so it is a lower order
ideal in $\Tc_B$ for every choice of base chamber $B\in\Tc\cap\Lf_{\geq C}$.
\end{subrmk}

\subsection{Coxeter arrangements}
\label{section_Coxeter_arrangements}

Let $(W,S)$ be a Coxeter system,
that is, $W$ is the group generated by the set $S$
and the relations between the generators are of the form
$(st)^{m_{s,t}} = 1$ where $m_{s,s} = 1$ and $m_{s,t} \geq 2$ for $s \neq t$;
see~\cite[Section~1.1]{BB}.
The corresponding
Coxeter graph has vertex set $S$, and two generators $s$ and $t$
are connected with an edge if $m_{s,t} \geq 3$.
If $m_{s,t} \geq 4$ it is customary to label the edge by the integer~$m_{s,t}$.

There are three natural partial orders on the elements of the 
Coxeter group $W$.
First the {\em strong Bruhat order} is defined by the following
cover relation: $z \coveredby w$ if there exists $s\in S$ and $u\in W$
such that $(usu^{-1}) z = w$ and
$\ell(z)+1 = \ell(w)$ where $\ell$ is the length function on $W$;
see for example~\cite[Definition~2.1.1]{BB}.
Next, we have 
the {\em right} (respectively {\em left}) {\em weak Bruhat order},
where the cover relation is
$z \coveredby w$ if there exists $s\in S$
such that
$z \cdot s= w$
(respectively $s \cdot z = w$)
and
$\ell(z) + 1 =\ell(w)$. 
The strong Bruhat order refines both the left and right weak Bruhat orders.

Let $V=\bigoplus_{s\in S}\R e_s$, with the symmetric
bilinear form $(\cdot,\cdot)$ defined by
\[(e_s,e_t)=-\cos\left({\pi}/{m_{s,t}}\right).\]
In particular, $(e_s,e_s) = 1$.
The \emph{canonical representation} of $(W,S)$
is the representation of~$W$ on~$V$ given by
\begin{align}
\label{equation_canonical_representation}
s(v) & = v - 2 \cdot (e_s,v) \cdot e_s,
\end{align}
for every $s\in S$ and every $v\in V$. Note that 
this formula defines an orthogonal isomorphism of~$V$ for the symmetric
bilinear form $(\cdot,\cdot)$.
We refer the reader
to~\cite[Chapitre~V, \S~4, \textnumero~8, Th\'eor\`eme~2 p.~98]{Bourbaki}
for the next result.
\begin{subthm}\label{thm_W_finite}
Equation~\eqref{equation_canonical_representation}
defines a faithful representation of~$W$ on~$V$,
and the form $(\cdot,\cdot)$ is positive definite if and only if $W$ is finite.
\end{subthm}

From now on, we assume that $W$ is finite, and
we write $\Phi=\{w(e_s) : w\in W,\ s\in S\}$
and $\Phi^+=\Phi\cap\sum_{s\in S}\R_{\geq 0} e_s$.
The set $\Phi$ is a {\em pseudo-root system},
its subset $\Phi^+$ is a set of {\em positive pseudo-roots},
and
the set $\Phi^- = -\Phi^+ = \Phi - \Phi^+$ is the corresponding
set of {\em negative pseudo-roots};
see Definitions~\ref{def_pseudo_root_system}
and~\ref{def_positive_pseudo_roots}.
Then $\Hf=(H_\alpha)_{\alpha\in\Phi^+}$ is an essential
hyperplane arrangement on~$V$. The set of chambers~$\Tc$
of this arrangement is in
canonical bijection with $W$: the unit element
$1\in W$ corresponds to the chamber
$B=\bigcap_{\alpha\in\Phi^+}H_\alpha^+=\bigcap_{s\in S}H_{e_s}^+$, 
and an arbitrary element
$w$ of~$W$ corresponds to the chamber $w(B)$. 

More generally, a \emph{parabolic subgroup} of $W$
is a subgroup $W_I$ generated by a subset~$I$ of $S$, and the left
cosets of parabolic subgroups of $W$ are called \emph{standard cosets}.
The \emph{Coxeter complex}~$\Sigma(W)$
of~$W$ is the set of standard cosets of $W$
ordered by reverse inclusion. It is a simplicial complex, and
we have an isomorphism of posets from $\Sigma(W)$ to the
face poset~$\Lf$ of $\Hf$ 
sending a standard coset~$wW_I$ to the cone
$\{x\in V : \forall s\in I\ (x,w(e_s))=0\text{ and }
\forall s\in S-I\ (x,w(e_s))>0\}$. 
The fact that this is an isomorphism is proved 
in~\cite[Chapitre~V \S~4 \textnumero~6 pp.\ 96--97]{Bourbaki}, 
since the representation of
$W$ on $V^\vee$ is isomorphic to its canonical representation on~$V$ by
Theorem~\ref{thm_W_finite}. The fact that $\Sigma(W)$ is a simplicial
complex then follows 
from~\cite[Chapitre~V \S~3 \textnumero~3 Proposition~7 p.~85]{Bourbaki}.

The definitions of $B$ and of the isomorphism $\Tc\simeq W$ imply that,
if $w,w'\in W$ and $T_w,T_{w'}\in\Tc$ are the corresponding chambers, then
\begin{align*}
S(T_w,T_{w'})
& =
\{\alpha\in\Phi^+ :
w^{-1}(\alpha) \in \Phi^+ \text{ and } {w'}^{-1}(\alpha) \in \Phi^- \}  \\
& \quad\cup 
\{\alpha\in\Phi^+ : 
w^{-1}(\alpha) \in \Phi^- \text{ and } {w'}^{-1}(\alpha) \in \Phi^+\} ,
\end{align*}
and in particular
\[S(B,T_w)=\{\alpha\in\Phi^+ : w^{-1}(\alpha) \in \Phi^-\},\]
hence, by~\cite[Proposition~4.4.4]{BB},
\[(-1)^{T_w}=(-1)^{|S(B,T_w)|}=\det(w).\]
By~\cite[Propositions~3.1.3 and~4.4.6]{BB} this also implies that
the isomorphism $\Tc\simeq W$ sends the partial order
$\preceq_B$ to the right weak Bruhat order on $W$. 

\begin{subdef}
Let $\Hf=(H_{\alpha_e})_{e\in E}$
be a finite hyperplane arrangement on a finite-dimensional real
inner product space $V$, with
inner product denoted by $(\cdot,\cdot)$.
We say that $\Hf$ is a \emph{Coxeter arrangement} if
$\alpha_e\not\in\R\alpha_f$ for distinct $e,f\in E$ and if
for every $e \in E$
the family of hyperplanes~$\Hf$ is stable by the (orthogonal)
reflection $s_{\alpha_e}$ across~$H_{\alpha_e}$.
\end{subdef}

\begin{subthm}\label{thm_Coxeter_arrangements}
The hyperplane arrangement associated to a Coxeter system with finite
Coxeter group is a Coxeter arrangement.
Conversely, suppose that $\Hf$ is a Coxeter arrangement on 
an inner product space $V$, and 
that there exists
a chamber $B$ of $\Hf$ that is on the positive side of each
hyperplane in $\Hf$.
Let $W$ be
the subgroup of $\GL(V)$ generated by the
set $\{s_{\alpha_e} : e \in E\}$, let
$F$ be the set of $e \in E$
such that $\overline{B}\cap H_{\alpha_e}$ is a facet of $\overline{B}$
and let $S=\{s_{\alpha_f} : f \in F\}$. Then $(W,S)$ is a Coxeter system,
the group $W$ is finite,
and the hyperplane
arrangement induced by $\Hf$ on $V/\bigcap_{e\in E}H_{\alpha_e}$ 
is isomorphic to
the arrangement
associated to the Coxeter system~$(W,S)$.
\end{subthm}
\begin{proof}
The first statement is an immediate consequence of the definition of the
arrangement associated to a Coxeter system.
The second and fourth statements
follow from~\cite[Chapitre~V \S~3 \textnumero~2
Th\'eor\`eme~1 p.~74]{Bourbaki}. The statement that $W$ is finite follows
from~\cite[Chapitre~V \S~3 \textnumero~7 Proposition~4
p.~80]{Bourbaki} and from the fact that the arrangement $\Hf$ is central.
\end{proof}

\section{The abstract pizza quantity}
\label{section_first_main_theorem}

\subsection{$2$-structures and signs}
\label{section_2_structures}

Let $\Phi\subset V$ be a pseudo-root system 
(see Definition~\ref{def_pseudo_root_system})
with Coxeter group $W$ (see
Proposition~\ref{prop_pseudo_root_system_vs_Coxeter_system})
and $\Phi^+\subset\Phi$ be a
system of positive pseudo-roots
(see Definition~\ref{def_positive_pseudo_roots}).
Recall the definition of $2$-structures from
Subsection~\ref{subsection_2-structures}:
A \emph{$2$-structure} for $\Phi$ is
a subset $\varphi \subseteq \Phi$ such that:
\begin{itemize}
\item[(a)] $\varphi$ is a pseudo-root
system whose irreducible components are all of type~$A_1$, $B_2$
or $I_2(2^k)$ with $k\geq 3$;
\item[(b)] for every $w\in W$ such that $w(\varphi\cap\Phi^+)=
\varphi\cap\Phi^+$, we have $\det(w)=1$.
\end{itemize}

Recall that $\Tt(\Phi)$ is the set of $2$-structures for $\Phi$.
By Proposition~\ref{prop_W_acts_transitively_on_T_Phi},
the group $W$ acts transitively on $\Tt(\Phi)$.
In Definition~\ref{def_sign_2_structure}
we define the sign $\epsilon(\varphi)=\epsilon(\varphi,\Phi^+)$ of
any $2$-structure $\varphi\in\Tt(\Phi)$.
If $\varphi\in\Tt(\Phi)$, we write
$\varphi^+=\varphi\cap\Phi^+$.

We have the following proposition that extends~\cite[Theorem~5.3]{Herb-DSC}
to the case of Coxeter systems.
Note that our proof is a simple adaptation of Herb's proof.

\begin{subprop}
The sum of the signs of all $2$-structures of a pseudo-root system 
is equal to~$1$, that is,
\[\sum_{\varphi\in\Tt(\Phi)}\epsilon(\varphi)=1.\]
\label{prop_sum_of_signs}
\end{subprop}
\begin{proof}
We prove the result by induction on
$|\Phi|$. It is clear if $\Phi=\varnothing$, because then
$\Tt(\Phi)=\{\varnothing\}$ and the sign of $\varnothing$ is $1$.
Suppose that $|\Phi|\geq 1$ and that we know the result for all
pseudo-root systems of smaller cardinality.
Let $\alpha\in\Phi$, and set
$\Phi_\alpha=\alpha^\perp\cap\Phi$;
this is a pseudo-root system with positive system
$\alpha^\perp\cap\Phi^+$.

Let $\Tt''=\{\varphi\in\Tt(\Phi) : s_\alpha(\varphi)=\varphi\}$.
By statement~(0) of Lemma~\ref{lemma_T_Phi_induction_step}, we have
$\Tt''=\{\varphi\in\Tt(\Phi) : \alpha\in\varphi\}$.
If $\varphi\not\in\Tt''$, then $\varphi^+\subset\Phi^+ - \{\alpha\}$,
so $s_\alpha(\varphi^+)\subset\Phi^+$ by
Lemma~4.4.3 of~\cite{BB}, 
hence $\epsilon(s_\alpha(\varphi))=-\epsilon(\varphi)$ by
Lemma~\ref{lemma_sign_of_w(varphi)}.
This implies that $\sum_{\varphi\in\Tt(\Phi) - \Tt''}\epsilon(\varphi)=0$.

We define subsets $\Tt''_1$ and $\Tt''_2$ of $\Tt''$ by
\begin{align*}
\Tt''_1
& =
\{\varphi\in\Tt(\Phi) : \varphi\cap\Phi_\alpha\in \Tt(\Phi_\alpha)\} , \\
\Tt''_2
& =
\Tt'' - \Tt''_1.
\end{align*}
By (3) of Lemma~\ref{lemma_T_Phi_induction_step}, there exists an involution
$\iota$ of $\Tt''_2$ such that, for every
$\varphi\in\Tt''_2$, 
we have that $\iota(\varphi)\cap\Phi_\alpha=
\varphi\cap\Phi_\alpha$ and $\epsilon(\iota(\varphi))=-\epsilon(\varphi)$.
This implies that
\[\sum_{\varphi\in\Tt''_2}
\epsilon(\varphi)=0,\]
and so
\[
\sum_{\varphi\in\Tt(\Phi)} \epsilon(\varphi) 
=\sum_{\varphi\in\Tt''_1}
\epsilon(\varphi).\]

Finally, by (1) and (2)
of Lemma~\ref{lemma_T_Phi_induction_step},
the map $\varphi\longmapsto\varphi\cap\Phi_\alpha$
induces a bijection from $\Tt''_1$ to~$\Tt(\Phi_\alpha)$, and we have
$\epsilon(\varphi)=\epsilon(\varphi\cap\Phi_\alpha)$
for every $\varphi\in\Tt''_1$. Hence we obtain
\[\sum_{\varphi\in\Tt''_1}\epsilon(\varphi)=\sum_{\varphi_0\in\Tt(\Phi_\alpha)}
\epsilon(\varphi_0),\]
and this last sum is equal to $1$ by the induction hypothesis.
\end{proof}

\begin{subrmk}
As we are using the definition of the sign of a $2$-structure from
\cite{Herb-Pl}, our formula looks a bit different from the one
of~\cite[Theorem~5.3]{Herb-DSC}. This is explained
in~\cite[Remark~5.1]{Herb-DSC}, and we generalize the comparison
between the two definitions of the sign in
Corollary~\ref{cor_comparison_signs} below.
\end{subrmk}

\begin{subcor}
Let $\varphi\in\Tt(\Phi)$, 
$W(\varphi,\Phi^+)=\{w\in W:w(\varphi^+)\subset\Phi^+\}$ and
$W_1(\varphi,\Phi^+)=\{w\in W:w(\varphi^+)\subset\varphi^+\}$.
Then the sign $\epsilon(\varphi,\Phi^+)$ is given by
\[\epsilon(\varphi,\Phi^+)
=\frac{1}{|W_1(\varphi,\Phi^+)|}
\sum_{w\in W(\varphi,\Phi^+)}\det(w).\]
\label{cor_comparison_signs}
\end{subcor}
\begin{proof}
By Corollary~\ref{cor_T_Phi_as_quotient} we have a bijection
$W(\varphi,\Phi^+)/W_1(\varphi,\Phi^+) \longrightarrow \Tt(\Phi)$,
$w \longmapsto w(\varphi)$. By Proposition~\ref{prop_sum_of_signs}
and Lemma~\ref{lemma_sign_of_w(varphi)}, we obtain
\begin{align*}
1
& =
\frac{1}{|W_1(\varphi,\Phi^+)|}\sum_{w\in W(\varphi,\Phi^+)}
\epsilon(w(\varphi),\Phi^+)
=
\epsilon(\varphi,\Phi^+)
\frac{1}{|W_1(\varphi,\Phi^+)|}\sum_{w\in W(\varphi,\Phi^+)}
\det(w).
\qedhere
\end{align*}
\end{proof}

We consider the hyperplane arrangement $\Hf=(H_\alpha)_{\alpha\in\Phi^+}$
corresponding to $\Phi$, with base chamber $B=\bigcap_{\alpha\in\Phi^+}
H_\alpha^+$. 
For every $2$-structure $\varphi\in\Tt(\Phi)$, we denote
by $\Hf_\varphi$ the hyperplane arrangement $(H_\alpha)_{\alpha\in\varphi^+}$,
with base chamber $B_\varphi=\bigcap_{\alpha\in\varphi^+}H_\alpha^+$.
If $T$ is a chamber of $\Hf$, we denote by $Z_\varphi(T)$ the unique
chamber of $\Hf_\varphi$ containing $T$;
as $\varphi^+\subset\Phi^+$, we have $Z_\varphi(B)=B_\varphi$.

\begin{subcor}
For every chamber $T$ of $\Hf$, we have
\[(-1)^T=\sum_{\varphi\in\Tt(\Phi)}
(-1)^{Z_\varphi(T)}\epsilon(\varphi).\]
\label{cor_signs}
\end{subcor}

Recall that $(-1)^T=(-1)^{|S(B,T)|}$ for every $T\in\Tc(\Hf)$, and
similarly for $T\in\Tc(\Hf_\varphi)$.

\begin{proof}[Proof of Corollary~\ref{cor_signs}.]
For every $\varphi\in\Tt(\Phi)$, we denote the Coxeter group of
$\varphi$ by $W(\varphi)$. We also use the notation
of Lemma~\ref{lemma_sign_of_w(varphi)}.
Let $w$ be the unique element of $W$ such that $T=w^{-1}(B)$.
Let $\varphi\in\Tt(\Phi)$. Then $\varphi\cap w(\Phi^+)$ is a system
of positive pseudo-roots
in $\varphi$, so there exists a unique
$v\in W(\varphi)$ such that $v(\varphi^+)=\varphi\cap w(\Phi^+)$;
we write $v=v_\varphi(w)$. As $T=\{x\in V : \forall\alpha\in w(\Phi^+) \: 
(x,\alpha)> 0\}$, we have
\begin{align*}
Z_\varphi(T)
& =
\{x\in V : \forall\alpha\in w(\Phi^+)\cap\varphi \:\: (x,\alpha)> 0\} \\
& =
\{x\in V : \forall\alpha\in v_\varphi(w)(\varphi) \:\: (x,\alpha)> 0\},
\end{align*}
and so $v_\varphi(w)$ is the element of $W(\varphi)$
corresponding to $Z_\varphi(T)$ by the bijection from
$W(\varphi)$ to the set of chambers of $\Hf_\varphi$
sending $v$ to $v^{-1}(Z_\varphi(B))$.

For a $2$-structure $\varphi\in\Tt(\Phi)$ we have that
$w^{-1}v_\varphi(w)(\varphi^+)=w^{-1}(\varphi\cap w(\Phi^+))\subset\Phi^+$,
so by Lemma~\ref{lemma_sign_of_w(varphi)}
(and the fact that $v_\varphi(w)(\varphi)=\varphi$),
we obtain that
\[\epsilon(w^{-1}(\varphi))=\det(w^{-1}v_\varphi(w))\epsilon(\varphi).\]
Hence we have
\begin{align*}
\sum_{\varphi\in\Tt(\Phi)} (-1)^{Z_\varphi(T)}\epsilon(\varphi)
& =
\sum_{\varphi\in\Tt(\Phi)}\det(v_\varphi(w))\epsilon(\varphi) \\
& =
\det(w) \cdot \sum_{\varphi\in\Tt(\Phi)} \epsilon(w^{-1}(\varphi)) \\
& =
\det(w) \cdot \sum_{\varphi\in\Tt(\Phi)} \epsilon(\varphi) ,
\end{align*}
where in the last step
we used that the map $\varphi \longmapsto w^{-1}(\varphi)$
on the set $\Tt(\Phi)$ is bijective.
The result now follows by Proposition~\ref{prop_sum_of_signs}.
\end{proof}

\subsection{Calculating the abstract pizza quantity with 2-structures}
\label{thm_first_main}

We use the notation of Appendix~\ref{appendix_valuations}.
In particular, 
if $K$ is a closed convex polyhedral cone in $V$, we denote the set of
its closed faces by $\F(K)$
(we include $K$ itself in the set of its faces).
The
\emph{dimension} $\dim K$ of $K$ is by definition the dimension of
its span $\Span(K)$, and the \emph{relative interior}~$\mathring{K}$ of
$K$ is the interior of $K$ in $\Span(K)$. We say that
$K$ is \emph{degenerate} if $\Span(K)$ is strictly included in $V$,
equivalently, if $K$ has empty interior.

Let $\Hf$ be a central hyperplane arrangement on $V$ with fixed base
chamber $B$.
Let $\Cf_\Hf(V)$ be the set of closed convex polyhedral cones in $V$
that are intersections of closed half-spaces bounded by hyperplanes $H$ where $H\in\Hf$.
Denote the free abelian group on $\Cf_\Hf(V)$ 
by $\bigoplus_{K\in \Cf_\Hf(V)}\Z[K]$
and
let $K_\Hf(V)$ be its quotient 
by the relations 
$[K]+[K']=[K\cup K']+[K\cap K']$ for all
$K,K'\in \Cf_\Hf(V)$ such that $K\cup K'\in \Cf_\Hf(V)$.
For $K\in \Cf_\Hf(V)$, we still denote
the image of $K$ in $K_\Hf(V)$ by $[K]$. 
For the relative interior $\mathring{K}$
we also define
a class $[\mathring{K}]\in K_\Hf(V)$ by
\[[\mathring{K}]=(-1)^{\dim K}\sum_{F\in\F(K)}(-1)^{\dim F}[F].\]
We then have
\[[K]=\sum_{F\in\F(K)}[\mathring{F}]\]
by formula~(A.4) on page~543 of~\cite{GKM}.

Recall that $\Lf(\Hf)$ and $\Tc(\Hf)$ are the set of faces and chambers
of the arrangement $\Hf$ as in Subsection~\ref{background}.
Each $C\in\Lf(\Hf)$ is the relative interior of its closure and,
if $T\in\Tc(\Hf)$,
then $\F(\overline{T})=\{\overline{C}\,:\, C\in\Lf(\Hf),\ C\leq T\}$.
We have
$V=\coprod_{C\in\Lf(\Hf)}C$, and the family
$([C])_{C\in\Lf(\Hf)}$ is a $\Z$-basis of $K_\Hf(V)$.
As in Section~\ref{background},
the \emph{sign} of a face $C\in\F(\Hf)$ is defined by $(-1)^C=
(-1)^{|S(B,C\circ B)|}$.

We consider the following quantity:
\[\Pi(\Hf)=\sum_{C\in\Lf(\Hf)}(-1)^{C}[C]
\in K_\Hf(V).\]

Let $A$ be an abelian group. We say that a function
$\nu:\Cf_\Hf(V) \longrightarrow A$ is a \emph{valuation} on $\Cf_\Hf(V)$ if, for
all $K,K'\in \Cf_\Hf(V)$ such that $K\cup K'\in \Cf_\Hf(V)$ we have
$\nu(K\cup K')+\nu(K\cap K')=\nu(K)+\nu(K')$. 
Such a valuation $\nu$ defines a morphism of abelian groups
$K_\Hf(V) \longrightarrow A$ sending $[K]$ to $\nu(K)$
for every $K\in \Cf_\Hf(V)$, and we still denote this morphism
by $\nu:K_\Hf(V) \longrightarrow A$. We set
\[\Pi(\Hf,\nu)=\nu(\Pi(\Hf))\in A.\]
If $\nu $ vanishes on degenerate cones then we have
\begin{equation}
\Pi(\Hf,\nu)=\sum_{T\in\Tc(\Hf)}(-1)^T\nu(T)=\sum_{T\in\Tc(\Hf)}(-1)^T
\nu(\overline{T})\in A.
\label{equation_P_H_nu}
\end{equation}

The first main theorem of this article is the following.
For Coxeter arrangements
we can express the quantity $\Pi(\Hf)$ in terms
of the quantities $\Pi(\Hf_\varphi)$ for the arrangements $\Hf_\varphi$
associated to the $2$-structures of the arrangement.
\begin{subthm}
Let $\Phi\subset V$ be a pseudo-root
system.  Choose a system of positive pseudo-roots $\Phi^+\subset\Phi$
and let $\Hf$  be the hyperplane
arrangement $(H_\alpha)_{\alpha\in\Phi^+}$ on $V$. 
\begin{enumerate}
\item[(i)]
We have the identity
\[\Pi(\Hf)=\sum_{\varphi\in\Tt(\Phi)}\epsilon(\varphi)\Pi(\Hf_\varphi)\]
in the quotient $K_\Hf(V)$, where $\Hf_\varphi$ is as before the
arrangement $(H_\alpha)_{\alpha\in\varphi\cap\Phi^+}$ for every
$\varphi\in\Tt(\Phi)$.

\item[(ii)]
If $\nu:\Cf_\Hf(V)\longrightarrow A$ is a valuation, we have
\[\Pi(\Hf,\nu)=
\sum_{\varphi\in\Tt(\Phi)}\epsilon(\varphi)\Pi(\Hf_\varphi,\nu).\]
\end{enumerate}
\label{theorem_2_structures_and_chambers}
\end{subthm}

If $\varphi\in\Tt(\Phi)$ then the faces of $\Hf_\varphi$ 
are relative interiors of elements of $\Cf_\Hf(V)$, so $\Pi(\Hf_\varphi)$ makes
sense as an element of $K_\Hf(V)$.

\begin{subrmk}
This theorem is useful in the following situation. Suppose
that we have a function~$f$ on closed convex polyhedral cones and that
we wish to calculate the alternating sum over the chambers $T$ of a
hyperplane arrangement~$\Hf$ of the values $f(\overline{T})$. If $\Hf$
is a Coxeter arrangement and the function $f$ is a valuation that vanishes
on cones contained in hyperplanes of~$\Hf$, then the theorem says that
we can reduce the problem to a similar calculation for very simple
subarrangements of~$\Hf$ that are products of rank $1$ and rank $2$
Coxeter arrangements.

Here are two situations when we wish to calculate
alternating sums of $f(\overline{T})$ for such a valuation~$f$:
\begin{itemize}
\item[(a)] The weighed sums of Section~\ref{section_definition_sum}.
These sums appear in the calculation of weighted cohomology of locally
symmetric spaces and Shimura varieties 
(see Appendix~\ref{appendix_more_detail} for additional details and
references). We want to relate them to stable discrete series
constants to get a spectral description of that cohomology.

\item[(b)] The pizza problem (see for example
the paper~\cite{EMR_pizza}). In this setting, we fix a measurable
subset $K$ of $V$ with finite volume, and the function
$f$ sends a cone $C$ to the volume of $C\cap K$.
We are interested in ``the pizza quantity'',
that is,
the alternating sum of the volumes $f(K\cap\overline{T})$.
In particular we would like to know when this alternating
sum vanishes,
which is to say that the ``pizza'' $K$ has been evenly divided
among the two participants, $+$ and~$-$.
This problem is approached by analytic methods
in~\cite{EMR_pizza}.
Theorems~1.1 and~1.2 in~\cite{EMR_pizza} give
general sufficient conditions to guarantee that the pizza quantity vanishes.
Using Theorem~\ref{theorem_2_structures_and_chambers}
we can give a dissection proof; see~\cite{EMR_pizza2}.
\end{itemize}

When $f$ is a valuation that does not vanish on cones contained
in hyperplanes of $\Hf$, we have to decide how to count the contributions
of lower-dimensional faces of $\Hf$. One possibility is given in
Theorem~\ref{theorem_2_structures_and_chambers}, and another in
Corollary~\ref{corollary_no_walls}. In both cases, if $\Hf$ is a Coxeter arrangement,
then we can again reduce
the calculation to the case of simpler subarrangements of $\Hf$.
This is not needed in situation~(a), but in situation (b) it allows
us to obtain versions of the pizza theorem that hold for all the intrinsic volumes.
See~\cite{EMR_pizza2} for this.
\end{subrmk}

We will provide a proof of Theorem~\ref{theorem_2_structures_and_chambers} 
in Subsection~\ref{proof_main_theorem}.
First we state and prove a corollary.
For $\Hf$ a central hyperplane arrangement on $V$ with a fixed base chamber,
we define
\begin{align*}
P(\Hf)
& =
\sum_{T\in\Tc(\Hf)}(-1)^{T}[\overline{T}]\in K_\Hf(V) , \\
P_0(\Hf)
& =
\sum_{T\in\Tc(\Hf)}(-1)^{T}[T]\in K_\Hf(V).
\end{align*}
Analogous to the pizza quantity defined in Section~2
of~\cite{EMR_pizza}, we call $P(\Hf)$ the \emph{abstract pizza quantity}
of the arrangement $\Hf$.

\begin{sublemma}
For $\Hf$ a central hyperplane arrangement on $V$,
we have
\[P_0(\Hf)=P(\Hf).\]
\label{lemma_P_vs_P0}
\end{sublemma}
\begin{proof}
If $T\in\Tc(\Hf)$ then we have
\[[\overline{T}]
=
[T] + \sum_{\substack{F\in\Lf(\Hf) \\ F<T}}[F].\]
Summing over all chambers $T$ of $\Hf$ yields
\begin{align*}
\sum_{T\in\Tc(\Hf)}(-1)^T[\overline{T}]
& =
\sum_{T\in\Tc(\Hf)}(-1)^T
\left([T] + \sum_{\substack{F\in\Lf(\Hf) \\ F<T}}[F]\right)\\
&=
\sum_{T\in\Tc(\Hf)}(-1)^T[T]+
\sum_{F\in\Lf(\Hf) - \Tc(\Hf)}[F]
\sum_{\substack{T\in\Tc(\Hf) \\ T > F}} (-1)^T.
\end{align*}
The last inner sum is equal to zero,
yielding the result.
\end{proof}

\begin{subcor}
If $\Phi$ and $\Hf$ are as in
Theorem~\ref{theorem_2_structures_and_chambers}, we have
\[P(\Hf)=\sum_{\varphi\in\Tt(\Phi)}\epsilon(\varphi)
P(\Hf_\varphi).\]
\label{corollary_no_walls}
\end{subcor}
\begin{proof}
By Lemma~\ref{lemma_P_vs_P0}, it suffices to prove that
\[P_0(\Hf)=\sum_{\varphi\in\Tt(\Phi)}\epsilon(\varphi)
P_0(\Hf_\varphi),\]
where $P_0(\Hf)=\sum_{T\in\Tc(\Hf)}(-1)^T[T]$ and $P_0(\Hf_\varphi)$
is defined similarly.
For every $x\in V$ and every $\varphi\in\Tt(\Phi)$, let
$a_\varphi(x)$ equal $(-1)^Z$ if there exists a chamber
$Z$ of $\varphi$ such that $x\in Z$, and $0$ otherwise.
We need to show that
$\sum_{\varphi\in\Tt(\Phi)}\epsilon(\varphi)a_\varphi(x)$ is equal to
$(-1)^T$ if there exists a chamber $T$ of $\Hf$ such that $x\in T$,
and otherwise is zero.

Consider the valuation $\nu:\Cf_\Hf(V)\rightarrow K_\Hf(V)$ sending a cone 
$C\in \Cf_\Hf(V)$ to
$\sum_{T\in\Tc(\Hf),\ T\subset C}[T]\in K_\Hf(V)$. 
This valuation corresponds to the endomorphism of
$K_\Hf(V)$ sending the class of $T$ to itself if
$T\in\Tc(\Hf)$, and the class of $F$ to $0$ if
$F\in\F(\Hf) - \Tc(\Hf)$.
The valuation $\nu$
vanishes on degenerate cones, so
by statement~(ii) of Theorem~\ref{theorem_2_structures_and_chambers}
and equation~\eqref{equation_P_H_nu}, we have that
if $T$ is a chamber of~$\Hf$ and $x\in T$ then
$\sum_{\varphi\in\Tt(\Phi)}\epsilon(\varphi)a_\varphi(x)=(-1)^T$.

Now let $x\in V - \bigcup_{T\in\Tc(\Hf)}T$, and let
$F$ be the unique face of $\Hf$ such that $x\in F$.
For each $\varphi\in\Tt(\Phi)$, there is at most one chamber of
$\varphi$ that contains $x$. We denote by $X$ the set of
pairs~$(\varphi,Z)$, where $\varphi\in\Tt(\Phi)$ and $Z$ is a chamber
of $\varphi$ such that $x\in Z$.
As $F$ is not a chamber, there exists $e\in E$ such that $F\subset H_e$.
We denote by $s$ the orthogonal reflection in the hyperplane~$H_e$.
As $s(x)=x$, we can
make $s$ act on $X$ by sending $(\varphi,Z)$ to $(s(\varphi),s(Z))$.
This is a fixed-point free involution. Indeed, if
$\varphi$ is a $2$-structure such that $s(\varphi)=\varphi$, then
$e\in\varphi$ by statement~(0) of Lemma~\ref{lemma_T_Phi_induction_step},
so $H_e$ is a hyperplane of $\Hf_\varphi$,
which is impossible because $x$ is both in $H_e$ and in
a chamber of~$\varphi$. To prove that
$\sum_{\varphi\in\Tt(\Phi)}\epsilon(\varphi)a_\varphi(x)=\sum_{(\varphi,Z)\in X}
\epsilon(\varphi)(-1)^Z$ is equal to $0$, it suffices to show that
for every $(\varphi,Z)\in X$ we have
$\epsilon(\varphi)(-1)^Z=-\epsilon(s(\varphi))(-1)^{s(Z)}$. After applying
an element of $W$ to the whole situation, we may assume that
$x$ is in the base chamber of $\Hf$. Then~$Z$, respectively $s(Z)$, is the
base chamber of $\varphi$, respectively $s(\varphi)$, so
$(-1)^Z=(-1)^{s(Z)}=1$. Also, as the reflection~$s$ sends the base chamber of $\varphi$
to that of $s(\varphi)$, we have $s(\varphi^+)\subset\Phi^+$, and
so $\epsilon(s(\varphi))=-\epsilon(\varphi)$ by 
Lemma~\ref{lemma_sign_of_w(varphi)}.
\end{proof}

\subsection{Proof of Theorem~\ref{theorem_2_structures_and_chambers}}
\label{proof_main_theorem}

In this subsection we prove
Theorem~\ref{theorem_2_structures_and_chambers}.
We begin by stating and proving a lemma.

\begin{sublemma}
Let $\Hf$ be a central hyperplane arrangement on $V$, let $C$ be a face of
$\Hf$ and let $x\in V$. We denote by $C_0$ the unique face of $\Hf$
containing $x$. Then for $D\leq C$ a face of $\Hf$ the following three
conditions are equivalent:
\begin{itemize}
\item[(a)] $x\in\overline{C}+\Span(D)=\overline{C}+\Span(\overline{D})$;
\item[(b)] $\psi_x(\overline{D}^{\perp,\overline{C}})=1$, where
$\overline{D}^{\perp,\overline{C}}=\overline{D}^\perp\cap\overline{C}^*$
as in Subsection~\ref{subsection_convolution};
\item[(c)] $D\circ C_0\leq C$.
\end{itemize}
Moreover, if $C_0$ is a chamber then these conditions can only hold if
$C$ is also a chamber, and they are equivalent to the following
condition:
\begin{itemize}
\item[(d)] $D\circ C_0=C$.
\end{itemize}
\label{lemma_psi_x_vs_circ}
\end{sublemma}
\begin{proof}
We first note that $\Span(\overline{D})=\Span(D)$ because
$\Span(D)$ is a finite-dimensional subspace of~$V$, hence it is closed
and so contains $\overline{D}$. This explains the equality in condition~(a).

To prove that conditions (a) and (b) are equivalent, we note that
by the definition of the valuation $\psi_x$ in
Lemma~\ref{lemma_GKM} we have $\psi_x(\overline{D}^{\perp,\overline{C}})=1$
if and only if $x\in(\overline{D}^{\perp,\overline{C}})^*$. As
$\overline{D}^{\perp,\overline{C}}=\overline{D}^\perp\cap\overline{C}^*$ by
definition, we have $(\overline{D}^{\perp,\overline{C}})^*=\Span(\overline{D})+
\overline{C}$, so condition (b) is equivalent to the fact that
$x\in\Span(\overline{D})+\overline{C}$, which is condition (a).

We prove that (c) implies (a). Let $y\in D$. If $\varepsilon>0$ then
$y+\varepsilon x\in D\circ C_0$, so $y+\varepsilon x\in\overline{C}$ by~(c).
Thus $x=\frac{1}{\varepsilon}(y+\varepsilon x-y)\in\overline{C}+
\Span(D)$, which is condition~(a).

We prove that (a) implies~(c). By condition~(a) we can write
$x=x_1+x_2$, with $x_1\in\overline{C}$ and $x_2\in\Span(D)$.
Let $y\in D$. If $\varepsilon>0$ is small enough then
$y+\varepsilon x_2\in D$, so $y+\varepsilon x=(y+\varepsilon x_2)+
\varepsilon x_1\in\overline{C}$. As $y+\varepsilon x\in D\circ C_0$
for $\varepsilon>0$ small enough, this shows that
$D\circ C_0\subset\overline{C}$, that is, $D\circ C_0\leq C$, which
is condition~(a).

Finally, suppose that $C_0$ is a chamber. Then $D\circ C_0$ is a chamber,
so condition~(c) can only hold if $C$ is also a chamber, and it is equivalent
to condition~(d) because chambers are maximal faces.
\end{proof}

\begin{proof}[Proof of Theorem~\ref{theorem_2_structures_and_chambers}.]
We first prove statement~(ii) of~Theorem~\ref{theorem_2_structures_and_chambers}
for a valuation $\nu:\Cf_\Hf(V)\longrightarrow A$ that vanishes on degenerate
cones.
For every $\varphi\in\Tt(\Phi)$, we have by
equation~\eqref{equation_P_H_nu}:
\begin{align*}
\Pi(\Hf_\varphi,\nu)=
\sum_{Z\in\Tc(\Hf_\varphi)}(-1)^Z\nu(Z)=\sum_{Z\in\Tc(\Hf_\varphi)}(-1)^Z
\sum_{\substack{T\in\Tc(\Hf) \\ T\subset Z} } \nu(T).
\end{align*}
Hence
\begin{align*}
\sum_{\varphi\in\Tt(\Phi)}\epsilon(\varphi)\Pi(\Hf_\varphi,\nu)
& =
\sum_{\varphi\in\Tt(\Phi)}\epsilon(\varphi)\sum_{Z\in\Tc(\Hf_\varphi)}
(-1)^Z \sum_{\substack{T\in\Tc(\Hf) \\ T\subset Z}} \nu(T) \\
&=
\sum_{T\in\Tc(\Hf)}\nu(T)
\sum_{\varphi\in\Tt(\Phi)}\epsilon(\varphi)
(-1)^{Z_\varphi(T)}.
\end{align*}
As $\sum_{\varphi\in\Tt(\Phi)}\epsilon(\varphi)
(-1)^{Z_\varphi(T)}=(-1)^T$ for every $T\in\Tc(\Hf)$ by
Corollary~\ref{cor_signs}, the statement follows.

We now prove statement~(i) of~Theorem~\ref{theorem_2_structures_and_chambers}.
Fix a point $x$ in the base chamber~$B$ of~$\Hf$. For every
closed convex polyhedral cone $K\subset V$, let
$\F_x(K)$ be the set of closed faces $F$ of $K$ such that $x\in K+\Span(F)$.
Consider
the function $\psi:\Cf_\Hf(V)\longrightarrow K_\Hf(V)$ 
defined by 
\[\psi(K)=\sum_{F\in\F_x(K)}(-1)^{\dim F}[F].\]
This is the $\star$-product in the sense of
Corollary~\ref{cor_second_convolution}
(see also Remark~\ref{rmk_K_Hf})
of the valuations
$\Cf_\Hf(V) \longrightarrow K_\Hf(V)$, $K \longmapsto [K]$ and
$\psi_x:\Cf(V^\vee) \longrightarrow \Z$, where $V^\vee$ is the dual of $V$
and $\psi_x$ is the valuation of Lemma~\ref{lemma_GKM}.
More explicitly, for $K\subset V^\vee$ a nonempty closed convex
polyhedral cone, we have $\psi_x(K)=1$
if and only if $x\in K^*$. Indeed, with the notation of
that definition, we have $(F^{\perp,K})^*=K+\Span(F)$ for
every $K\in \Cf_\Hf(V)$ and every closed face $F$ of $K$.
By Corollary~\ref{cor_second_convolution},
the function~$\psi$ is a valuation, so
it induces a morphism $\psi:K_\Hf(V) \longrightarrow K_\Hf(V)$.
Moreover, 
the valuation~$\psi$ vanishes on degenerate cones in $\Cf_\Hf(V)$. Indeed,
if $K\in \Cf_\Hf(V)$ is contained in a hyperplane $H$ of~$\Hf$, then
$K+\Span(F)\subset H$ for every $F\in\F(K)$, so
$\F_x(K)=\varnothing$ because $x$ is not on any hyperplane of~$\Hf$.
Hence we can apply statement~(i)
that we just proved to~$\psi$.

Let $\Hf'$ be a subarrangement of $\Hf$, and let $B'$ be the unique
chamber of $\Hf'$ containing $B$. If $T\in\Tc(\Hf')$, we
write $\F_0(T)=\{C\in\Lf(\Hf') : C\leq T\ \mathrm{and}\ C\circ B'=T\}$. 
As $x\in B\subset B'$, we have 
$\F_x(\overline{T})=\{\overline{C} : C\in\F_0(T)\}$ by
Lemma~\ref{lemma_psi_x_vs_circ}.
We deduce that
\begin{align*}
\Pi(\Hf',\psi)
&=\sum_{T\in\Tc(\Hf')}(-1)^T\psi(\overline{T}) \\
& =
\sum_{T\in\Tc(\Hf')}(-1)^T\sum_{F\in\F_x(T)}(-1)^{\dim F}[F] \\
&=\sum_{T\in\Tc(\Hf')}(-1)^T\sum_{C\in\Lf(\Hf'),\ C\circ B'=T}
(-1)^{\dim C}[\overline{C}] \\
&=\sum_{C\in\Lf(\Hf')}(-1)^C(-1)^{\dim C}[\overline{C}].
\end{align*}
Using statement~(i) for the valuation $\psi$, we get that
\begin{align}
\sum_{C\in\Lf(\Hf)}(-1)^C(-1)^{\dim C}[\overline{C}]
& =
\sum_{\varphi\in\Tt(\Phi)}\epsilon(\varphi)
\sum_{C\in\Lf(\Hf_\varphi)}(-1)^C(-1)^{\dim C}[\overline{C}].
\label{equation_almost_there}
\end{align}
By the top of page 544 of~\cite{GKM}, there exists an endomorphism
of $K_\Hf(V)$ sending $[K]$ to $(-1)^{\dim K}[\mathring{K}]$, for every
$K\in \Cf_\Hf(V)$. Applying this endomorphism to the identity~\eqref{equation_almost_there}
yields statement~(i).

Finally, the general case of statement~(ii) immediately follows
from applying the morphism $\nu:K_\Hf(V)\longrightarrow A$
to both sides of the identity of statement~(i).
\end{proof}

\section{The weighted sum}
\label{section_definition_sum}

\subsection{The weighted complex and the weighted sum}
\label{section_weighted_complex}

We return to the situation of Subsection~\ref{background}.
In particular, we fix a finite-dimensional real inner product space
$V$ and a central hyperplane arrangement $\Hf=(H_{\alpha_e})_{e\in E}$
on $V$, and we denote by $\Lf$ and $\Tc$ the sets of faces and
chambers of $\Hf$.

\begin{subdef}
Let $\lambda\in V$. We consider the following subset of the face poset $\Lf$:
\[\Lf_{\lambda}=\{C\in\Lf :  C \subseteq \overline{H_\lambda^+}\}.\]
In other words, the set $\Lf_\lambda$ is the collection of faces on
the nonnegative side of the hyperplane $H_\lambda$.
More generally, if $C_0$ is a fixed face of $\Lf$, we also consider the
intersection
$$ \Lf_{\lambda,\geq C_0}=\Lf_\lambda\cap\Lf_{\geq C_0} . $$
\label{def_weighted_complex} 
\end{subdef}

\begin{subrmk}
(See~\cite[Section~4.5]{OM} for definitions.)
If $\lambda \neq 0$ then the hyperplane arrangement
$\{H_\lambda\}\cup\{H_{\alpha_e} : e \in E\}$ defines an affine oriented
matroid with distinguished hyperplane~$H_\lambda$.
If~$H_\lambda$ is in general position relative to the $H_{\alpha_e}$, that is,
if $\lambda$ is not in the span of any family $(\alpha_e)_{e \in F}$ for
$|F| \leq \dim(V)-1$, then $\Lf_{\lambda}$ coincides with the bounded complex
of this affine oriented matroid. In general,
$\Lf_{\lambda}$ is larger.
\label{rmk_weighted_vs_bounded}
\end{subrmk}

The basic properties of the subsets $\Lf_{\lambda}$ and $\Lf_{\lambda,\geq C_0}$
are given in the following proposition.

\begin{subprop}
The following two statements hold:
\begin{itemize}
\item[(i)]
For a fixed face $C_0$ of $\Lf$ the set $\Lf_{\lambda,\geq C_0}$
is a lower order ideal in $\Lf_{\geq C_0}$.

\item[(ii)] Let $C\in\Lf_{\lambda}$. Then there exists
$T\in\Tc\cap\Lf_{\lambda}$ such that $C\leq T$.
\end{itemize}
\label{prop_weighted_complex_basic}
\end{subprop}

\begin{proof}
It suffices to prove (i) when $C_0$ is the minimal face of $\Lf$.
Let $C,D\in\Lf$ such that $C \leq D$ and $D\in\Lf_{\lambda}$.
The hypothesis implies that $C\subset\overline{D}$ and
$D\subset\overline{H_\lambda^+}$. As $\overline{H_\lambda^+}$ is closed,
this immediately gives $C\subset\overline{H_\lambda^+}$,
hence $C\in\Lf_{\lambda}$.

To show (ii) 
let $D_{1},D_{2}, \ldots, D_{p}$ be the chambers of $\Tc$ that are
larger than $C$ with respect to the partial order $\leq$. If one of them is
contained in $\overline{H_\lambda^+}$, then we are done. Otherwise, for every
$1 \leq i \leq p$, we can find a point $x_i\in D_i$ such that
$(\lambda,x_i)<0$. Let 
$\Hf^{\prime}$ be the subarrangement of~$\Hf$
where we remove all the hyperplanes of $\Hf$
that contain the cone $C$. In the arrangement~$\Hf^{\prime}$
all the points~$x_{i}$ are contained in the same chamber
$C^{\prime}$. 
In particular, the convex hull $P$ of the points $x_{1}, x_{2}, \ldots, x_{p}$ 
is contained
in $C^{\prime}$. The convex hull $P$ intersects the linear span of
the cone~$C$ in a point $x$. Since all the points $x_{i}$ are in
the open half-space $H_\lambda^-$, so is the point~$x$,
that is, $(\lambda,x) < 0$.
By inserting the hyperplanes of $\Hf$ that contain
the cone~$C$ back in the arrangement $\Hf^{\prime}$,
we subdivide the region $C^{\prime}$ into regions $C_{1}, C_{2}, \ldots, C_{p}$.
But the point $x$ belongs to the closure of each region $C_{i}$, thus
the point $x$ belongs to the cone~$C$. This is a contradiction since
$C$ is contained in the half-space $\overline{H_{\lambda}^+}$,
so $(\lambda,x) \geq 0$.
\end{proof}

\begin{subdef}
The subcomplex 
of the cell decomposition
$\Sigma(\Lf)$, respectively $\Sigma(\Lf_{\geq C_0})$,
whose face poset is the lower
order ideal $\Lf_\lambda$, respectively $\Lf_{\lambda,\geq C_0}$,
is called the \emph{weighted complex}
and denoted by $\Sigma(\Lf_\lambda)$,
respectively $\Sigma(\Lf_{\lambda,\geq C_0})$.
By
Proposition~\ref{prop_weighted_complex_basic}
it is pure of the same dimension as 
$\Sigma(\Lf_\lambda)$,
respectively $\Sigma(\Lf_{\lambda,\geq C_0})$.
\end{subdef}

We are interested in the following quantity.
\begin{subdef}
\label{def_psi}
Let $\lambda$ be a vector in~$V$ and $B$ a chamber of $\Hf$,
that is, $B \in\Tc$. 
The {\em weighted sum} is defined to be
\begin{align}
\psi_\Hf(B,\lambda)
& =
\sum_{D\in\Lf_\lambda} 
(-1)^{\dim(D)} \cdot (-1)^{|S(B,D\circ B)|}.
\label{equation_psi_H}
\end{align}
More generally, if $C$ is a face of the arrangement $\Hf$, that is,
$C\in\Lf$, and if $B$ is a chamber whose closure contains the face $C$, that is,
$B\in\Tc\cap\Lf_{\geq C}$, we define the weighted sum to be
\begin{align}
\psi_{\Hf/C}(B,\lambda)
& =
\sum_{D\in\Lf_{\lambda,\geq C}} (-1)^{\dim(D)} \cdot (-1)^{|S(B,D\circ B)|}.
\label{equation_psi_H_C}
\end{align}
\end{subdef}

\begin{subrmk}
This definition seems very arbitrary and mysterious. We try to give
some context for it in Appendix~\ref{appendix_more_detail}. Without
getting into too much detail here, the weighted
sum (in the situation of Example~\ref{ex_parabolic_arrangement}) plays
a role in the calculation of the trace of Hecke operators 
on the weighted cohomology of locally symmetric varieties that is very
similar to the role played by stable discrete constants
(see for example pages~493
and 498--500 of~\cite{GKM}) in the calculation of the trace of
Hecke operators on $L^2$ cohomology of these varieties.
It is because of this that we chose the names ``weighted complex''
and ``weighted sum''.
In fact, in that situation we only really need the ``absolute'' version
where $C$ is the minimal face of $\Hf$. The ``relative'' version,
where $C$ is not minimal anymore, appears when the locally
symmetric variety is a Shimura variety defined over a number field $E$
and we are considering the
trace of a Hecke operator multiplied by an element of the absolute
Galois group of $E$.
\label{rmk_name_weighted_sum}
\end{subrmk}

\begin{subrmk}
In our previous paper, we used the notation $S(\lambda)$
(see~\cite[Equation~(6.1)]{Ehrenborg_Morel_Readdy})
to denote what turns out to be
a particular case of the sum in equation~\eqref{equation_psi_H_C}
in the type~$B$ Coxeter case; 
see equation~\eqref{eq_S_vs_psi} in Subsection~\ref{second_application} for 
the precise relation between the two.
In this paper, we decided to follow the notation of~\cite{GKM}
in order to avoid overuse of the letter~$S$.
\end{subrmk}

We state the following lemma. It reduces the calculation of
$\psi_{\Hf/C}(B,\lambda)$ to the case of an essential arrangement.

\begin{sublemma}
Let $V_0$ be the intersection of all the hyperplanes of $\Hf$,
that is, $V_0=\bigcap_{e\in E}H_{\alpha_e}$.
Let $\pi$ denote the projection $V \longrightarrow V/V_0$.
Let $\Hf/V_0$ be the hyperplane arrangement $(H_{\alpha_e}/V_0)_{e\in E}$ on~$V/V_0$.
Note that $\pi$ induces an isomorphism between the face poset
of~$\Hf$ and~$\Hf/V_0$.
Let $C\in\Lf$, let $B$ be a chamber of $\Hf$ such that $B\in\Lf_{\geq C}$
and let $\lambda\in V$.
Then the following identity holds:
\[\psi_{\Hf/C}(B,\lambda)
=
\begin{cases}
(-1)^{\dim(V_0)} \cdot \psi_{(\Hf/V_0)/\pi(C)}(\pi(B),\pi(\lambda))
& \text{ if }\lambda\in V_0^\perp, \\
0 & \text{ if } \lambda\not\in V_0^\perp .
\end{cases} \]
\label{lemma_psi_inessential_arrangement}
\end{sublemma}
\begin{proof}
Note that $\lambda\in V_0^\perp$ if and only if $V_0\subset H_\lambda$.
If $\lambda\not\in V_0^\perp$ then the linear functional $(\lambda,\cdot)$
takes both positive and negative values on $V_0$. As $V_0\subset\overline{D}$
for every $D\in\Lf$, this linear functional also takes both positive and
negative values on $D$, so $D\not\in\Lf_\lambda$. This shows that
$\Lf_\lambda=\varnothing$ if $\lambda\not\in V_0^\perp$ and gives the second
case.
Now suppose that $\lambda\in V_0^\perp$. Then $V_0\subset H_\lambda$, and it
is easy to see that $D\in\Lf_\lambda$, respectively $D\in\Lf_{\geq C}$,
if and only if $\pi(D)\subset
\overline{H^+_{\pi(\lambda)}}$, respectively $\pi(D)\geq\pi(C)$,
and that $\dim(\pi(D))=\dim(D)-\dim(V_0)$.
This yields the first case.
\end{proof}

Suppose that $V=V_1\times\cdots\times V_r$ with the $V_i$ mutually
orthogonal subspaces of $V$
and that $\Hf$ also decomposes
as a product $\Hf_1\times\cdots\times\Hf_r$. By this, we mean that there
is a decomposition $E=E_1\sqcup\cdots\sqcup E_r$ such that,
for $1 \leq i \leq r$ and every $e\in E_i$, we have $\alpha_{e} \in V_i$.
The arrangement $\Hf_i=(V_i\cap H_{\alpha_e})_{e\in E_i}$
is a hyperplane arrangement on the subspace $V_i$,
and each hyperplane of $\Hf$ is of the form
$H \times\prod_{j \neq i}V_j$,
where $1 \leq i \leq r$ and $H$ is one of the hyperplanes of $\Hf_i$.

Let $\Lf_i$ be the face poset of $\Hf_i$ for $1 \leq i \leq r$.
Then the faces of $\Lf$ are exactly the products $C_1\times\cdots
\times C_r$ where $C_i\in\Lf_i$, and the order on $\Lf$ is the product
order. In particular, $C$ is a chamber in $\Lf$ if and only if all the $C_i$ are
chambers in $\Lf_i$.

\begin{sublemma}
\label{lemma_product}
Assume that the arrangement $\Hf$ factors as $\Hf_1\times\cdots\times\Hf_r$
as described in the two previous paragraphs.
Let $C=C_1\times\cdots\times C_r$ be a face in $\Lf$, and let
$B=B_1\times\cdots\times B_r$ be a chamber in $\Lf_{\geq C}$. Finally,
let $\lambda\in V$. Then
\[\psi_{\Hf/C}(B,\lambda)=\prod_{i=1}^r\psi_{\Hf_i/C_i}(B_i,\lambda_i),\]
where, for $1 \leq i \leq r$, $\lambda_i$ is the orthogonal
projection of $\lambda$ on $V_i$.
\end{sublemma}
\begin{proof}
The expression for
$\psi_{\Hf/C}(B,\lambda)$ follows from the
fact that $\Lf=\Lf_1\times\cdots\times\Lf_r$ as posets 
once we prove the following statement: Let $D=D_1\times\cdots\times D_r\in\Lf$,
with $D_i\in\Lf_i$. Then $D\in\Lf_\lambda$ if and only if $D_i\in
\Lf_{i,\lambda_i}$ for $1 \leq i \leq r$.

We prove this last fact. Note that $\lambda=(\lambda_1,\ldots,\lambda_r)$
in $V_1\times\cdots\times V_r=V$ because the $V_i$ are pairwise orthogonal.
If $D_i\in\Lf_{i,\lambda_i}$ for every
$1 \leq i \leq r$ then for every $x=(x_1,\ldots,x_r)\in V$ we have
$(\lambda,x)=\sum_{i=1}^r(\lambda_i,x_i)\geq 0$. Conversely, suppose that
$D\in\Lf_{\lambda}$. Let $x_j\in D_j$ for $1 \leq j \leq r$. 
As all the $D_j$ are cones, for every $\varepsilon>0$ the element
$\varepsilon x_j$ is in $D_j$. 
Fix $1 \leq i \leq r$ and consider the element
$x_\varepsilon=
(\varepsilon x_1,\ldots,\varepsilon x_{i-1},x_i,\varepsilon x_{i+1},\ldots,\varepsilon x_r)$.
Then $x_\varepsilon$ is in $D$.
Thus we have the inequality
$0 \leq (\lambda,x_\varepsilon)
=(\lambda_i,x_i)+\varepsilon\sum_{j \neq i} (\lambda_j,x_j)$.
Letting $\varepsilon$ tend to $0$, we obtain $(\lambda_i,x_i)\geq 0$
and hence $D_i\in \Lf_{i,\lambda_i}$.
\end{proof}

\subsection{Calculating the weighted sum for some arrangements
with many symmetries}
\label{section_second_main_thm}

We continue to use the notation of
Subsections~\ref{background} and~\ref{section_Coxeter_arrangements}.
Suppose that 
$E=E^{(1)}\sqcup E^{(2)}$, with $\Hf^{(1)}=(H_{\alpha_e})_
{e\in E^{(1)}}$ a Coxeter arrangement 
whose Coxeter group~$W$ stabilizes $\Hf^{(2)}=(H_{\alpha_e})_{e\in E^{(2)}}$,
and that 
$C=(\bigcap_{e\in E^{(1)}}H_{\alpha_e})\cap(\bigcap_
{e\in E^{(2)}}H^+_{\alpha_e})$,
so that $E(C) = \{e\in E : C\subset H_{\alpha_e}\} = E^{(1)}$.
To simplify some of the notation,
without loss of generality
we may assume that the vectors~$\alpha_e$, where $e\in E$, are all unit vectors.
We assume that there
exists a chamber~$B$ of~$\Hf$ that is on the positive side
of every hyperplane of~$\Hf$ (not just~$\Hf^{(1)}$); in particular,
we have $C\leq B$.
This also defines a chamber of the arrangement $\Hf^{(1)}$, and we denote
by~$(W,S)$ the associated Coxeter system, as in
Theorem~\ref{thm_Coxeter_arrangements}.

The set $\Phi=\{\pm\alpha_e : e\in E^{(1)}\}$ is
a normalized pseudo-root system (see
Definition~\ref{def_pseudo_root_system}), the subset $\Phi^+=\{\alpha_e : 
e\in E^{(1)}\}$ is a system of positive pseudo-roots in $\Phi$ 
(see Definition~\ref{def_positive_pseudo_roots}), and $(W,S)$ is the
corresponding Coxeter system.
See
Proposition~\ref{prop_pseudo_root_system_vs_Coxeter_system}.

Our main example of such arrangements is the following.

\begin{subex}
\label{ex_parabolic_arrangement}
{\rm
Let $(W,S)$ be a Coxeter system, let $V$ be the canonical representation
of~$W$, and let $\Hf=(H_\alpha)_{\alpha\in\Phi^+}$ be the associated
hyperplane arrangement on $V$ as in 
Section~\ref{section_Coxeter_arrangements}.
Let $I$ be a subset of $S$, set
$\Phi^{(1)} = \Phi^+ \cap\bigl(\sum_{\alpha\in I}\R\alpha\bigr)$
and $\Phi^{(2)}=\Phi^+-\Phi^{(1)}$. Then $\Hf^{(1)}=(H_\alpha)_{\alpha\in
\Phi^{(1)}}$ is a Coxeter arrangement with associated Coxeter system
$(W_I,I)$, where $W_I$ is the subgroup of~$W$ generated by~$I$,
and $W_I$ preserves the arrangement
$\Hf^{(2)}=(H_\alpha)_{\alpha\in\Phi^{(2)}}$.
If $C = (\bigcap_{\alpha\in \Phi^{(1)}} H_\alpha)
\cap (\bigcap_{\alpha\in \Phi^{(2)}}H^+_\alpha)$
as before,
then the chamber $B$ corresponding to $1\in W$ is in $\Lf_{\geq C}$.
}
\end{subex}

We use again the notion of $2$-structures for $\Phi$;
see Subsection~\ref{section_2_structures}.
If $\varphi\in\Tt(\Phi)$ we write
$\varphi^+=\varphi\cap\Phi^+$, and we denote by
$\Hf_\varphi$ the hyperplane arrangement 
$(H_{\alpha})_{\alpha\in\varphi^+\sqcup E^{(2)}}$
and by $B_\varphi$, respectively $C_\varphi$,
the unique chamber of $\Hf_\varphi$ containing~$B$, respectively~$C$.
By the choice of $B$, the chamber $B_\varphi$ is also the unique chamber
on the positive side of every hyperplane in $\Hf_\varphi$.

\begin{subthm}
Let $\Hf = \Hf^{(1)} \sqcup \Hf^{(2)}$ be an arrangement in~$V$ with base chamber~$B$.
Assume that the subarrangement $\Hf^{(1)}$ is a Coxeter arrangement with
pseudo root system~$\Phi$ and its Coxeter group stabilizes~$\Hf$.
Let~$C$ be the intersection of the base chamber of~$\Hf^{(2)}$ and the hyperplanes in~$\Hf^{(1)}$.
Then for every $\lambda\in V$ we have
\[\psi_{\Hf/C}(B,\lambda)=\sum_{\varphi\in\Tt(\Phi)}
\epsilon(\varphi) \cdot \psi_{\Hf_\varphi/C_\varphi}(B_\varphi,\lambda).\]
\label{thm_second_main}
\end{subthm}

This is the second main theorem of this article.
It will be proved in Subsection~\ref{section_proof_second_main_theorem}.
See Definition~\ref{def_psi} for the description of the
terms and Remark~\ref{rmk_name_weighted_sum} for an explanation
of their significance. Theorem~\ref{thm_second_main}
states that the weighted
sum for a Coxeter arrangement can be expressed as an alternating sum of weighted
sums for much simpler subarrangements (the $2$-structures) that are direct
products of rank one and rank two Coxeter arrangements. As
weighted sums for these subarrangements can be calculated directly
(see Corollary~\ref{cor_root_systems} and
Proposition~\ref{prop_psi_rank_1_2_system} for the case where $C$ is
the minimal face), this gives a way to calculate the weighted sum
for the original Coxeter arrangement, and thus, as explained in
Appendix~\ref{appendix_more_detail}, to relate weighted cohomology
of locally symmetric varieties to the spectral side of the Arthur-Selberg
trace formula.

We first give some applications of Theorem~\ref{thm_second_main}.

\subsection{First application: the case of Coxeter arrangements}
\label{section_Coxeter_case}

We specialize Theorem~\ref{thm_second_main} to the case where $\Hf=\Hf^{(1)}$
is a Coxeter arrangement. In particular, $C$ is the minimal face
of $\Lf$, so $\psi_{\Hf/C}(B,\lambda)=\psi_\Hf(B,\lambda)$ for
every $\lambda\in V$. 

Let $\varphi\in\Tt(\Phi)$, and let $\varphi=\varphi_1
\sqcup\varphi_2\sqcup\cdots\sqcup\varphi_r$ be the decomposition of
$\varphi$ into irreducible pseudo-root systems.
Let $V_{i,\varphi}=\Span(\varphi_i)$ for $1 \leq i \leq r$.
Then
$V=V_{0,\varphi}\times V_{1,\varphi}\times\cdots\times V_{r,\varphi}$, where
$V_{0,\varphi}=\varphi^\perp$. 
The dimension of
$V_{0,\varphi}$ is equal to $\dim(V)-\rk(\varphi)$,
so it is independent of~$\varphi$ by
Proposition~\ref{prop_W_acts_transitively_on_T_Phi}.
Let $\Hf_{i,\varphi}$ be the hyperplane arrangement given by
$\varphi_i\cap\Phi^+$ on $V_{i,\varphi}$
where $1 \leq i \leq r$.
For a fixed index $i$
let $B_{i,\varphi}$ be the chamber
of the arrangement $\Hf_{i,\varphi}$ that is on the positive side of every hyperplane, and
let $\lambda_{i,\varphi}$ be the orthogonal projection of $\lambda$
on $V_{i,\varphi}$.

Combining Theorem~\ref{thm_second_main} with
Lemmas~\ref{lemma_psi_inessential_arrangement} and~\ref{lemma_product},
we obtain:
\begin{subcor}
\label{cor_root_systems}
For every $\lambda\in V$, we have
\[\psi_\Hf(B,\lambda)
=
(-1)^{\dim(V)-R} \cdot
\sum_{\substack{\varphi\in\Tt(\Phi)\\ \lambda\in\Span(\varphi)}}
\epsilon(\varphi) \cdot
\prod_{i=1}^r\psi_{\Hf_{i,\varphi}}(B_{i,\varphi},
\lambda_{i,\varphi}),\]
where $R$ is the rank of any $\varphi\in\Tt(\Phi)$.
\end{subcor}

To finish the calculation of $\psi_\Hf(B,\lambda)$ in this case,
we use the following proposition, whose proof is a straightforward
calculation.

\begin{subprop}
In types $A_1$, $B_2 = I_2(4)$ and $I_2(2^k)$ for $k \geq 3$,
the function $\psi$ is given
by the following expressions:
\begin{itemize}
\item[(1)] Type~$A_1$: Suppose that $V=\R e_1$ and that $\Phi^+=\{e_1\}$.
Then $\psi$ is given by
\[\psi_\Hf(B,c e_1)=\begin{cases}0 & \text{ if } c>0, \\
1 & \text{ if } c=0, \\
2 & \text{ if } c<0.\end{cases}\]

\item[(2)] Type~$I_2(2^k)$, where $k\geq 2$:
Let $V=\R e_1\oplus\R e_2$ with the usual inner product. For every $v\in V-\{0\}$, let
$\theta(v)\in[0,2\pi)$ be the angle from $e_1$ to $v$.
Suppose
that $\Phi$ is the set of unit vectors that have an angle of $r\pi/2^k$ with $e_1$, where
$r\in\Z$, and that $B$ is the set of nonzero vectors
$v\in V$ such that $0<\theta(v)<\pi/2^k$.
Then $\psi$ is given by
\[\psi_\Hf(B,\lambda)
=
\begin{cases}
1 & \text{ if } \lambda=0, \\
2 & \text{ if }\lambda \neq 0\text{ and }\theta(\lambda)=r\pi/2^k\text{ with }2^{k-1}+1\leq r\leq
3\cdot 2^{k-1}, \\
4 & \text{ if }\lambda \neq 0\text{ and }r\pi/2^k<\theta(\lambda)<(r+1)\pi/2^k \\
    & \text{ with } r \text{ odd and }2^{k-1}+1\leq r\leq 3\cdot 2^{k-1}  - 1, \\
0 & \text{ otherwise.} 
\end{cases}\]
\end{itemize}
\label{prop_psi_rank_1_2_system}
\end{subprop}

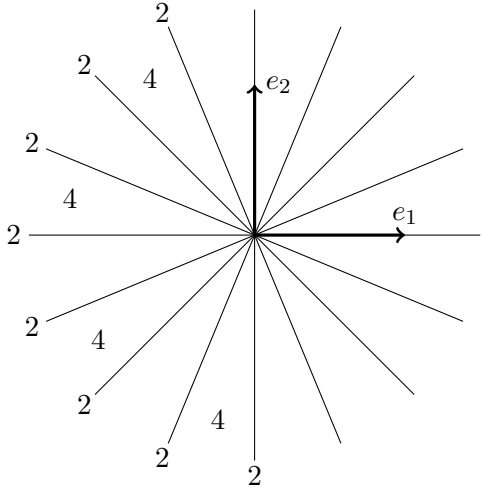
\begin{figure}[t]
\centering
\begin{tikzpicture}
\foreach \tttt in {0, 22.5, ..., 337.5}
{\draw[-] (0,0) -- ({3*cos(\tttt)},{3*sin(\tttt)});};
\foreach \tttt in {112.5, 135, ..., 270}
{\draw[-] ({3.2*cos(\tttt)},{3.2*sin(\tttt)}) node{$2$};};
\foreach \tttt in {123.75, 168.75, ..., 258.75}
{\draw[-] ({2.5*cos(\tttt)},{2.5*sin(\tttt)}) node{$4$};};
\draw[->, very thick] (0,0) -- (2,0) node[above]{$e_{1}$};
\draw[->, very thick] (0,0) -- (0,2) node[right]{$e_{2}$};
\end{tikzpicture}
\caption{The function $\psi_\Hf(B,\lambda)$ in
the dihedral pseudo-root system $I_{2}(8)$.
The origin is assigned the value $1$ and the unmarked
faces are assigned the value $0$.}
\label{figure_dihedral_psi}
\end{figure}

\begin{subrmk}
If $(W,S)$ arises from a root system $\Phi$ and $-1$ is an element of $W$
(or, equivalently, the root system is generated by strongly orthogonal
roots), then Goresky--Kottwitz--MacPherson~\cite[Theorem~3.1]{GKM} 
and Herb~\cite[Theorem~4.2]{Herb-2S}
give two different expressions for the coefficients appearing
in the formula for the averaged discrete series characters
of a real reductive group with root system $\Phi$.
Corollary~\ref{cor_root_systems} asserts the equality of these two formulas.
In general, although there no longer exist discrete series
in this setting, the formulas of Goresky--Kottwitz--MacPherson and Herb still
make sense, and Corollary~\ref{cor_root_systems} says that they are still
equal. Also, Corollary~\ref{cor_root_systems} implies that
$\psi_\Hf(B,\lambda)=0$ if $\lambda$ is not in the span of any
$2$-structure for $\Phi$, so it implies~\cite[Theorem~5.3]{GKM}.
It is not clear whether this is an easier proof than the one given in~\cite{GKM}.
\end{subrmk}

\subsection{Second application: the type {\em{A}} 
identity involving ordered set partitions}
\label{second_application}

We now show how to
deduce~\cite[Theorem~6.4]{Ehrenborg_Morel_Readdy}\footnote{This
is a reformulation of~\cite[Proposition~A.4]{Morel}.}
from~Theorem \ref{thm_second_main}.
We take $V=\R^n$ with the usual inner product, and
we denote by $(e_1,\ldots,e_n)$ the standard basis of $V$.
We consider
the hyperplane arrangement $\Hf$ of type~$B_n$ on $V$, that is,
$\Hf=(H_\alpha)_{\alpha\in\Phi_B^+}$, where $\Phi_B^+=
\{e_i\pm e_j:1\leq i<j\leq n\}\cup\{e_1,\ldots,e_n\}$. 
We write $\Phi_B^+=\Phi^{(1)}\sqcup\Phi^{(2)}$, where
$\Phi^{(1)}=\{e_i-e_j:1\leq i<j\leq n\}$, and we denote
by $\Hf=\Hf^{(1)}\sqcup\Hf^{(2)}$ the corresponding decomposition of
$\Hf$. 
The arrangement~$\Hf^{(1)}$ is a Coxeter arrangement of type~$A_{n-1}$,
and we denote by $\Phi=\Phi^{(1)}\cup(-\Phi^{(1)})$ the associated root system.
Let $C$ be the intersection
$(\bigcap_{\alpha\in\Phi^{(1)}}H_\alpha)\cap(\bigcap_{\alpha\in
\Phi^{(2)}}H^+_\alpha)$. Then $C$ is the open ray
$\R_{>0}\cdot(e_1+e_2+\cdots+e_n)$.

Recall that $\Lf$ is the face poset of $\Hf$. We will now give a
description of $\Lf$ in terms of signed ordered partitions.
See also~\cite[Section~5]{ER} for this description.
A \emph{signed block} is a nonempty subset~$\til{B}$ 
of $\{\pm 1,\ldots,\pm n\}$ such that, for every $i\in\{1,
\ldots,n\}$, at most one of $\pm i$ is in $\til{B}$. We then denote
by $B$ the subset of $\{1,\ldots,n\}$ defined by $B=\{|i| : i\in\til{B}\}$.
A \emph{signed ordered partition} of a subset~$I$ 
of $\{1,\ldots,n\}$ is a list $(\til{B}_1,\ldots,\til{B}_r)$
of signed blocks such that $(B_1,\ldots,B_r)$ is
an ordered partition of~$I$. 

We consider the poset $\Qns$
whose elements are pairs $\pi=(\til{\pi},Z)$, where $Z \subseteq \{1,\ldots,n\}$
and $\til{\pi}$ is a signed ordered partition of $\{1,\ldots,n\}-Z$,
and the cover relation is given by the following two rules:
\begin{align*}
((\til{B}_1,\ldots,\til{B}_r),Z) & \coveredby ((\til{B}_1,\ldots,\til{B}_{r-1}),B_r\cup Z) , \\
((\til{B}_1,\ldots,\til{B}_r),Z) & \coveredby ((\til{B}_1,\ldots,\til{B}_{i-1},
\til{B}_i\cup\til{B}_{i+1},\til{B}_{i+2},\ldots,\til{B}_{r}),Z).
\end{align*}
The set $Z$ is usually called the {\em zero block} of $\pi$.

Let $\pi=(\til{\pi},Z)$ be an element of $\Qns$, with
$\til{\pi}=(\til{B}_1,\ldots,\til{B}_r)$.
We define the cone $C_\pi$
to be the set of $(x_1,\ldots,x_n)\in V$ such that
(with the convention that $x_{-i}=-x_i$ for $1 \leq i \leq n$):
\begin{itemize}
\item[(i)]
if $Z=\{i_1,\ldots,i_m\}$
then the equalities $x_{i_1}= \cdots =x_{i_m}=0$ hold;
\item[(ii)]
for every block $\til{B}= \{i_1,\ldots,i_m\}$ in  $\til{\pi}$,
the equalities and inequality $x_{i_1}= \cdots =x_{i_m}> 0$ hold;
\item[(iii)]
for every two consecutive blocks $\til{B}_{s}$ and $\til{B}_{s+1}$
in $\til{\pi}$
with $i \in \til{B}_s$ and $j \in \til{B}_{s+1}$, the inequality $|x_i| > |x_j|$ holds.
\end{itemize}

It is easy to see that the map $\varphi:\Qns \longrightarrow \Lf$ sending $\pi$ to
$C_\pi$ is a bijection, and that it induces an 
order-reversing isomorphism between the
poset $\Qns$  and the face poset~$\Lf$.
The inverse image of the ray
$C = \R_{>0}\cdot(e_1+e_2+\cdots+e_n)$
by this bijection is the
element $\pi_0=((\{1,\ldots,n\}),\varnothing)$ of~$\Qns$, so
the elements of $\Lf_{\geq C}$ correspond exactly to the
(unsigned) ordered partitions of $\{1,\ldots,n\}$. In other words,
the bijection $\varphi$ induces an 
order-reversing isomorphism between the poset~$\Qn$
of ordered partitions of $\{1,\ldots,n\}$ defined
in~\cite[Section~2]{Ehrenborg_Morel_Readdy}
and the poset $\Lf_{\geq C}$.

Let $\lambda=(\lambda_1,\ldots,\lambda_n)\in\R^n$. 
For a signed block $\til{B}$, we set $\lambda_{\til{B}}=\sum_{i\in B}
\lambda_i$, with the convention that
$\lambda_{-i}=-\lambda_{i}$ for $1 \leq i \leq n$.
Define the subset $\Qns(\lambda)$ of $\Qns$ by
\[
\Qns(\lambda)
=
\left\{ ((\til{B}_{1}, \til{B}_{2}, \ldots, \til{B}_{r}),Z) \in \Qns \:\: : \:\:
\sum_{i=1}^{s} \lambda_{\til{B}_{i}} \geq 0 
\text{ for } 1 \leq s \leq r \right\} .\]
Then an element $\pi$ of $\Qns$ is in $\Qns(\lambda)$ if and only if
$C_\pi$ is in $\Lf_\lambda$. Moreover, the subset~$\Lf_{\lambda,\geq C}$
corresponds to the set $\Qn(\lambda)$
of ordered partitions $(B_1,\ldots,B_r)$ of
$\{1,\ldots,n\}$ such that, for every $1 \leq s \leq r$, we have
$\sum_{i=1}^s\lambda_{B_i}\geq 0$. This is almost the set
$\mathcal{P}(\lambda)$ of~\cite[Section~3]{Ehrenborg_Morel_Readdy};
the only difference is that the inequalities defining $\mathcal{P}(\lambda)$
are strict. We can give the following identity relating these two sets:
For every $\varepsilon\in\R$, let $\lambda_\varepsilon=
(\lambda_1-\varepsilon,\ldots,\lambda_n-\varepsilon)$.
Then if $\varepsilon>0$ is sufficiently small, we have
$\Qn(\lambda_\varepsilon)=\mathcal{P}(\lambda)$.

Let $B$ be the unique chamber of $\Lf$ that is on the positive side of
every hyperplane, that is, $B=\{x_1>x_2>\cdots>x_n>0\}$. As we already
observed, $\Lf_{\geq C}$ is isomorphic to the face poset of the
arrangement $\Hf^{(1)}$, which is a Coxeter arrangement of type~$A_{n-1}$.
The unique chamber of this arrangement containing $B$ corresponds to
the identity element in the symmetric group $\mathfrak{S}_n$. It then follows 
from
Proposition~\ref{prop_f_B} that the function
$f_B:\Lf_{\geq C} \longrightarrow 
\Tc\cap\Lf_{\geq C}$ sending $C'\in\Lf_{\geq C}$ to $C\circ B$
corresponds via
$\varphi:\Qn \stackrel{\sim}{\longrightarrow} \Lf_{\geq C}$ to the
function $f:\Qn \longrightarrow  \mathfrak{S}_n$
of~\cite[Section~4]{Ehrenborg_Morel_Readdy}.
We obtain the equality:
\[\psi_{\Hf/C}(B,\lambda)
=
\sum_{\pi\in\Qn(\lambda)} (-1)^{|\pi|} \cdot (-1)^{f(\pi)},\]
where $|\pi|$
denotes the number of blocks of the ordered partition
$\pi = (B_1, \ldots, B_r)$, in other words, $|\pi|=r$.
Let 
$\overline{\lambda}$
denote the reverse of $\lambda$, that is,
$\overline{\lambda}=(\lambda_n,\ldots,\lambda_1)$.
For $\varepsilon$ real we let $\overline{\lambda}_\varepsilon$ be
$(\lambda_n-\varepsilon, \ldots,\lambda_1-\varepsilon)$.
By~\cite[Lemma~7.1]{Ehrenborg_Morel_Readdy}, we have
\[\psi_{\Hf/C}(B,\overline{\lambda})
=
(-1)^{\binom{n}{2}} \cdot
\sum_{\pi\in\Qn(\lambda)}(-1)^{|\pi|} \cdot (-1)^{g(\pi)},\]
where $g:\Qn \longrightarrow \mathfrak{S}_n$ is the function defined at the beginning
of~\cite[Section~6]{Ehrenborg_Morel_Readdy}.\footnote{The
map $f : \Qn \longrightarrow \mathfrak{S}_n$
takes an ordered partition, orders the elements in each block in increasing order
and then maps them to the permutation formed by
reading the elements from left to right.
The map $g$ is similarly defined, except the elements in each block
are reordered in decreasing order.}
Finally, using the fact that
$\Qn(\lambda_\varepsilon)=\mathcal{P}(\lambda)$ for
sufficiently small $\varepsilon>0$,
then the sum $S(\lambda)$
of~\cite[Section~6]{Ehrenborg_Morel_Readdy} is given by the expression:
\begin{equation}\label{eq_S_vs_psi}
S(\lambda)
=
(-1)^{\binom{n}{2}} \cdot \psi_{\Hf/C}(B,\overline{\lambda}_\varepsilon)
\end{equation}
for any sufficiently small $\varepsilon>0$.

We now find an expression for the sum $T(\lambda)$
of~\cite[Section~6]{Ehrenborg_Morel_Readdy} in terms of $2$-structures.
As in~\cite{Ehrenborg_Morel_Readdy},
we denote by~$M_n$
the set of maximal matchings on $\{1,2,\ldots,n\}$. Then
we have a bijection $M_n \stackrel{\sim}{\longrightarrow} \Tt(\Phi)$ sending
a matching $p=\{p_1,\ldots,p_m\}$,
where $p_1=\{i_1<j_1\},\ldots,p_m=\{i_m<j_m\}$ are the edges of~$p$,
to the $2$-structure
$\varphi_p=\{\pm(e_{i_1}-e_{j_1}),\ldots,\pm(e_{i_m}-e_{j_m})\}$.
Moreover, we have $(-1)^p=\epsilon(\varphi_p)$.
We can calculate 
$\psi_{\Hf_{\varphi_p}/C_{\varphi_p}}(B_{\varphi_p},\lambda)$ using
Lemma~\ref{lemma_product} for the decomposition $V=V_0\times
V_1\times\cdots\times V_m$, where $V_k=\R e_{i_k}+\R e_{j_k}$ for
$1\leq k\leq m$, $V_0=\{0\}$ if $n$ is even, and $V_0=\R e_i$
if $n$ is odd and $i$ is the unique unmatched element of $\{1,\ldots,n\}$.
By Lemma~\ref{lemma_product} we have
\[\psi_{\Hf_{\varphi_p}/C_{\varphi_p}}(B_{\varphi_p},\lambda)=
\prod_{k=1}^{m} d_2(\lambda_{i_k},\lambda_{j_k}) 
\cdot
\begin{cases}
1 & \text{if $n$ is even,} \\
d_{1}(\lambda_{i}) & \text{if $n$ is odd,}
\end{cases}
\]
where:
\begin{itemize}
\item[(a)] 
The function $d_1:\R \longrightarrow \R$ is defined by $d_1(a)=\psi_{\Hf_1/C_1}
(B_1,a)$, where $\Hf_1$ is the hyperplane arrangement $(H_e)$ on
$\R e$ and $B_1=C_1=\R_{>0}e$. 
\item[(b)] The function $d_2:\R^2 \longrightarrow \R$
is defined by $d_2(a,b)=\psi_{\Hf_2/C_2}
(B_2,(a,b))$, where $\Hf_2$ is the hyperplane arrangement $(H_e,H_f,
H_{e-f},H_{e+f})$ on
$\R e\oplus\R f$, $C_2=\{\alpha e+\beta f : \alpha=\beta>0\}$ and
$B_2=\{\alpha e+\beta f : \alpha>\beta>0\}$.
\end{itemize}
In other words, the functions $d_1$ and $d_2$ are precisely the function
$\psi_{\Hf/C}(B,\lambda)$ that we are trying to determine in the cases
$n=1$ and $n=2$. A direct calculation yields:
$$
d_1(a)=
\begin{cases}
-1 & \text{ if } a\geq 0, \\
0 & \text{ if } a<0,
\end{cases}
\:\:\:\: \text{ and } \:\:\:\:
d_2(a,b)=\begin{cases}
-1 & \text{ if } a,b \geq 0, \\
-2 & \text{ if } b \geq -a>0, \\
0 & \text{ otherwise}.
\end{cases}
$$
Comparing this with the formula defining $c(p,\lambda)$
in~\cite[Section~6]{Ehrenborg_Morel_Readdy}, we see that, for all $a,b\in\R$,
if $\varepsilon>0$ is sufficiently small relative to $a$ and $b$, we have
$d_{1}(a-\varepsilon) = - c_{1}(a)$ and $d_{2}(a-\varepsilon,b-\varepsilon) = - c_{2}(b,a)$,
and hence
\begin{align*}
\psi_{\Hf_{\varphi_p}/C_{\varphi_p}}(B_{\varphi_p},\overline{\lambda}_\varepsilon)
& =
c(p,\lambda)\cdot
\begin{cases}(-1)^{n/2} & \text{ if } n\text{ is even,} \\
(-1)^{(n+1)/2} & \text{ if }n\text{ is odd}\end{cases}
\\
& = 
(-1)^{n} \cdot 
(-1)^{\binom{n}{2}} \cdot 
c(p,\lambda),
\end{align*}
if $\varepsilon>0$ is sufficiently small relative to the $\lambda_i$.
Combining all these calculations,
we see that if $\varepsilon>0$ is sufficiently small, then
\[\sum_{\varphi\in\Tt(\Phi)}
\epsilon(\varphi) \cdot
\psi_{\Hf_\varphi/C_\varphi}(B_\varphi,\overline{\lambda}_
\varepsilon)
=
(-1)^{n} \cdot 
(-1)^{\binom{n}{2}} \cdot 
\sum_{p\in M_n}(-1)^p \cdot c(p,\lambda)
=
(-1)^{n} \cdot 
(-1)^{\binom{n}{2}} \cdot 
T(\lambda).\]
The identity
$S(\lambda)=(-1)^n \cdot T(\lambda)$
in~\cite[Theorem~6.4]{Ehrenborg_Morel_Readdy}
now follows from Theorem~\ref{thm_second_main},
applied to~$\overline{\lambda}_\varepsilon$ for
$\varepsilon>0$ sufficiently small.

\subsection{Proof of Theorem~\ref{thm_second_main}}
\label{section_proof_second_main_theorem}

We assume for now that $\Hf=(H_{\alpha_e})_{e\in E}$ is an arbitrary
central hyperplane arrangement on $V$.
The following definition will be useful.

\begin{subdef}
Let $C\in\Lf$ and $\lambda\in V$. If $D,D'\in\Lf_{\geq C}$,
we define $\psi_{D/C}(D',\lambda)$ by the sum
\[\psi_{D/C}(D',\lambda)=\sum_{\substack{C'\in\Lf_{\lambda,\geq C} \\
C'\circ D'\leq D}}(-1)^{\dim(C')},\]
where $\Lf_{\geq\lambda,C}=\Lf_\lambda\cap\Lf_{\geq C}$.
\end{subdef}

\begin{subrmk}
Suppose that $D'$ is a chamber. Then $C'\circ D'$ is a chamber
for every $C'\in\Lf$, so $\psi_{D/C}(D',\lambda)=0$ unless $D$ is
also a chamber. If $D$ is a chamber, we have
\[\psi_{D/C}(D',\lambda)=\sum_{\substack{C'\in\Lf_{\lambda,\geq C} \\
C'\circ D'=D}}(-1)^{\dim(C')}.\]
\label{rmk_psi_for_chambers}
\end{subrmk}

The functions $\psi_{D/C}(D',\lambda)$ are related to
$\psi_{\Hf/C}(B,\lambda)$ by the following lemma.

\begin{sublemma}
Let $C\in\Lf$ and $B\in\Tc\cap\Lf_{\geq C}$.
Then for every $\lambda \in V$ the following identity holds:
\[\psi_{\Hf/C}(B,\lambda)
=
\sum_{T\in\Tc\cap\Lf_{\geq C}}(-1)^{|S(B,T)|} \cdot \psi_{T/C}(B,\lambda).\]
\label{lemma_psi_vs_psi}
\end{sublemma}
\begin{proof} Indeed, if $D\in\Lf_{\geq C}$ then the chamber
$D\circ B$ is also in $\Lf_{\geq C}$. Hence, using
equation~\eqref{equation_psi_H_C}
in Definition~\ref{def_psi}
and 
Remark~\ref{rmk_psi_for_chambers}, we obtain:
\begin{align*}
\psi_{\Hf/C}(B,\lambda)
& =
\sum_{T\in\Tc\cap\Lf_{\geq C}}(-1)^{|S(B,T)|} \cdot
\sum_{\substack{D\in\Lf_{\lambda,\geq C} \\ D\circ B=T}}(-1)^{\dim(D)}
=
\sum_{T\in\Tc\cap\Lf_{\geq C}} (-1)^{|S(B,T)|} \cdot \psi_{T/C}(B,\lambda).
\end{align*}
\end{proof}

Before Corollary~\ref{cor_psi_C_is_a_valuation} of
Appendix~\ref{appendix_valuations},
we define, for $K$ a closed convex polyhedral
cone in~$V$, a function 
$\psi_K:V\times V^\vee \longrightarrow \R$. For fixed $(x,\ell)\in V\times V^\vee$,
the function $K \longmapsto\psi_K(x,\ell)$ is a valuation on the set of
closed convex polyhedral
cones in $V$ (see Definition~\ref{def_valuation}).
This function is related to the functions
$\psi_{D/C}(D',\lambda)$ in the following way.
\begin{sublemma}
Let $C\in\Lf$, let $D\in\Lf_{\geq C}$ and let
$\lambda\in V$.
Denote by $\ell\in V^\vee$ the linear functional~$(\cdot,\lambda)$.
Then for every $D'\in\Lf_{\geq C}$ the following identity holds:
\[\psi_{D/C}(D',\lambda) = \psi_{\overline{D}}(x,\ell),\]
where $x$ is any point in $D'_1=(-C)\circ D'$.
\label{lemma_on_the_definition_of_psi}
\end{sublemma}
\begin{proof}
As before we write $E(C)=\{e\in E :  C\subset H_{\alpha_e}\}$. 
Note that $s(D'_1)_e=s(D')_e$ for $e\in E(C)$, and $s(D'_1)=-s(C)_e \neq 0$
for $e\in E-E(C)$. Also, by definition of $\Lf_{\geq C}$,
if $e$ is any index of $E  -  E(C)$ and
$C'\in\Lf_{\geq C}$ then $s(C)_e=s(C')_e \neq 0$.

We claim that for every $C'\in\Lf$ we have
$C'\circ D'_1\leq D$ if and only if $C'\in\Lf_{\geq C}$ and $C'\circ D'\leq D$.
Suppose first that $C'\in\Lf_{\geq C}$ and $C'\circ D'\leq D$.
Then for every $e\in E(C)$ we have $s(C'\circ D'_1)_e=s(C'\circ D')_e\leq s(D)_e$.
Moreover, if $e\in E-E(C)$ then $s(C')_e=s(C)_e=s(D)_e \neq 0$, so
$s(C'\circ D'_1)_e=s(C')_e=s(D)_e$. This shows that $C'\circ D'_1\leq D$.
Conversely, suppose that $C'$ is a face of $\Lf$ such that
$C'\circ D'_1\leq D$.
If $e\in E-E(C)$ then $0 \neq s(C)_e=s(D)_e=-s(D'_1)_e$,
thus $s(C')_e \neq 0$, and so $s(C')_e=s(C'\circ D'_1)_e=s(D)_e$.
This implies that $C'\in\Lf_{\geq C}$.
Moreover, if $e\in E(C)$ we have $s(D'_1)_e=s(D')_e$,
thus $s(C'\circ D')_e=s(C'\circ D'_1)_e\leq s(D)_e$.
Hence we conclude that $C'\circ D'\leq D$.

By the claim, we obtain
\[
\psi_{D/C}(D',\lambda)
=
\sum_{\substack{C'\in\Lf_{\lambda} \\ C'\circ D'_1\leq D}} (-1)^{\dim(C')}
=
\psi_D(D'_1,\lambda) .
\]
We wish to show that this is equal to $\psi_{\overline{D}}(x,\ell)$,
if $x\in D'_1$. As in Appendix~\ref{appendix_valuations},
we denote by $\mathcal{F}(\overline{D})$ the set of closed faces of the closed
convex polyhedral cone $\overline{D}$. We have
$\mathcal{F}(\overline{D})=\{\overline{C'} : C'\in\Lf,\ C'\leq D\}$, and
the set $\{C'\in\Lf : C'\circ D'_1\leq D\}$ is included in the set
$\{C'\in\Lf : C'\leq D\}$.
To prove the equality above, it suffices to show 
that the two following statements hold
for $C'\in\Lf$ such that $C'\leq D$ (see
Lemma \ref{lemma_GKM} for the definition of $\psi_x$ and $\psi_\ell$,
and the beginning of
Section~\ref{subsection_convolution} for the notation $\overline{C'}^{\perp,\overline{D}}$):
\begin{itemize}
\item[(a)] The face $C'$ belongs to $\Lf_\lambda$ if and only $\psi_\ell(\overline{C'})=1$.
\item[(b)] The inequality $C'\circ D'_1\leq D$ holds if and only if 
$\psi_x\left(\overline{C'}^{\perp,\overline{D}}\right)=1$.
\end{itemize}

Statement (a) is just a direct translation of the definition of $\Lf_\lambda$,
and statement (b) is proved in Lemma~\ref{lemma_psi_x_vs_circ}.

\end{proof}

We are now ready to prove Theorem~\ref{thm_second_main}, so we assume that
we are in the situation of that theorem.

Let $\ell$ be the linear functional $(\cdot,\lambda)$ on $V$,
let $x\in (-C)\circ B$, and
consider the valuation $\nu$ on the set of closed convex polyhedral cones
in $V$ sending such a cone $K$ to $\psi_{K\cap\overline{\Cf}}(x,\ell)$,
where $\Cf=\Cf_C=\bigcap_{e\in E^{(2)}}H_{\alpha_e}^+$.
The function $K\longmapsto\psi_K(x,\ell)$ is a priori defined only
on the set of closed convex polyhedral cones. However, as it is a valuation,
we can extend it to the set of all finite intersections of
closed and open half-spaces in $V$. See
Remark~\ref{rmk_extension_valuation}.
As $x$ is not on any hyperplane of $\Hf$, 
the valuation $\psi_x$ vanishes on any cone contained in
a hyperplane of~$\Hf$.
It follows from the definition of the function
$K\longmapsto\psi_K(x,\ell)$ in the discussion before
Corollary~\ref{cor_psi_C_is_a_valuation}
that we have $\psi_K(x,\ell)=0$ if $K$ is contained in a hyperplane of $\Hf$.
In particular, for $D$ a face of~$\Hf$, we have
$\nu(D)=0$ unless $D$ is a chamber, and if $D$ is a chamber
then $\nu(D)=\nu(\overline{D})$.

Let $\varphi\subset\Phi$ be a pseudo-root system (we do not
assume that $\varphi$ is a $2$-structure), let $\varphi^+=
\varphi\cap\Phi^+$ and $\Hf_\varphi=(H_{\alpha_e})_{e\in\varphi^+\sqcup E^{(2)}}$, 
and denote by $B_\varphi$ and $C_\varphi$ the unique
faces of $\Hf_\varphi$ containing~$B$ and~$C$. We have
$C_\varphi=\bigcap_{e\in\varphi^+}H_{\alpha_e}\cap\bigcap_{e\in E^{(2)}}H_{\alpha_e}^+$.
We also set $\Hf^{(1)}_\varphi=(\Hf_{\alpha_e})_{e\in\varphi^+}$,
$\Lf_\varphi=\Lf(\Hf_\varphi)$ and $\Tc_\varphi=\Tc(\Hf_\varphi)$.

As in statement~(i) of Lemma~\ref{lemma_star}, we denote
by $\iota:\Lf_{\varphi,\geq C_\varphi}\fl\Lf(\Hf^{(1)}_\varphi)$ the map sending a face
$D\geq C_\varphi$ of $\Hf_\varphi$ to the unique face of 
$\Hf^{(1)}_\varphi$ that contains it.
By the lemma we just cited,
we know that this is an order-preserving bijection, and that its
inverse sends a face $D^{(1)}$ of $\Hf^{(1)}_\varphi$ to
$D^{(1)}\cap\Cf$, where $\Cf=\bigcap_{e\in E^{(2)}}H_{\alpha_e}^+$ as before.
We claim that, if $D^{(1)}$ is a face of $\Hf^{(1)}_\varphi$, then
$\overline{D^{(1)}\cap\Cf}=\overline{D^{(1)}}\cap\overline{\Cf}$.
Indeed, let $s\in\{+,-,0\}^{\varphi^+}$ be the sign vector of $D^{(1)}$.
Then the sign vector $t\in\{+,-,0\}^E$ of $D^{(1)}\cap\Cf$ is given
by $t_e=s_e$ if $e\in\varphi^+$, and $t_e=+$ if $e\in E^{(2)}$.
We set $\Rrr_+=\Rrr_{\geq 0}$, $\Rrr_-=\Rrr_{\leq 0}$ and $\Rrr_0=\{0\}$.
Let $y\in V$. Then $y\in\overline{D^{(1)}}$
if and only if $(\alpha_e,y)\in\Rrr_{s_e}$ for every
$e\in \varphi^+$,
while $y\in\overline{D^{(1)}\cap\Cf}$ if and only
$(\alpha_e,y)\in\Rrr_{t_e}$ 
for every $e\in E$
and $y\in\overline{\Cf}$ if and only if $(\alpha_e,y)\geq 0$ for
every $e\in E^{(2)}$. This immediately implies the claim.

Let $D^{(1)}$ be a face of $\Hf^{(1)}_\varphi$.
By Lemma~\ref{lemma_on_the_definition_of_psi} we have
\[\psi_{\iota^{-1}(D^{(1)})/C_\varphi}(B_\varphi,\lambda)=
\psi_{\overline{D^{(1)}\cap\Cf}}(x,\ell)=
\psi_{\overline{D^{(1)}}\cap\overline{\Cf}}(x,\ell)=
\nu(\overline{D^{(1)}}),\]
because $x\in(-C)\circ B\subset(-C_\varphi)\circ B_\varphi$.
Moreover, by Lemma~\ref{lemma_psi_vs_psi}, we have
\[\psi_{\Hf_\varphi/C_\varphi}(B_\varphi,\lambda)
=
\sum_{T\in\Tc_\varphi\cap\Lf_{\varphi,\geq C_\varphi}}(-1)^{|S(B_\varphi,T)|} \cdot 
\psi_{T/C_\varphi}(B_\varphi,\lambda).\]
So, using statements~(i) and~(iii) of Lemma~\ref{lemma_star}, we get that
\begin{align*}\psi_{\Hf_\varphi/C_\varphi}(B_\varphi,\lambda)&=
\sum_{T^{(1)}\in\Tc(\Hf^{(1)}_\varphi)}
(-1)^{|S(\iota(B_\varphi),T^{(1)})|}\psi_{\iota^{-1}(T^{(1)})/C_\varphi}(B_\varphi,\lambda)\\
&=
\sum_{T^{(1)}\in\Tc(\Hf^{(1)}_\varphi)}
(-1)^{|S(\iota(B_\varphi),T^{(1)})|}\nu(\overline{T^{(1)}})\\
&=
\sum_{C^{(1)}\in\Tc(\Hf^{(1)}_\varphi)}
(-1)^{|S(\iota(B_\varphi),C^{(1)}\circ\iota(B_\varphi))|}\nu(\overline{C^{(1)}}).
\end{align*}
In other words, using the notation of Subsection~\ref{thm_first_main},
we obtain
\begin{equation}
\psi_{\Hf_\varphi/C_\varphi}(B_\varphi,\lambda)=\Pi(\Hf_\varphi,\nu).
\label{eq_psi_vs_pizza}
\end{equation}
Applying this identity to $\varphi=\Phi$, we have
$\psi_{\Hf/C}(B,\lambda)=\Pi(\Hf,\nu)$.
By Theorem~\ref{theorem_2_structures_and_chambers}, we deduce that
\[\psi_{\Hf/C}(B,\lambda)=\sum_{\varphi\in\Tt(\Phi)}\epsilon(\varphi)
\Pi(\Hf_\varphi,\nu).\]
The conclusion of Theorem~\ref{thm_second_main} follows from this
equality and from equation~\eqref{eq_psi_vs_pizza} for all
$\varphi\in\Tt(\Phi)$.

\section{Properties of the weighted complex}
\label{section_weighted_complex_shellable}

This section is independent of Sections~\ref{section_first_main_theorem}
and~\ref{section_definition_sum}, except for
Definition~\ref{def_weighted_complex} 
and Remark~\ref{rmk_weighted_vs_bounded}.
We prove that the weighted complex is shellable for
Coxeter arrangements, or more generally for arrangements satisfying
some condition on the dihedral angles between their hyperplanes
(Condition~\ref{(A)}). This implies that the weighted complex
is a PL ball for arrangements satisfying Condition~\ref{(A)}.

\subsection{Shellable polytopal complexes}

We introduce the following definition.
For instance, see~\cite[Definition~4.7.14]{OM}.
\begin{subdef}
A pure $n$-dimensional polytopal complex $\Delta$ is {\em shellable} if 
it is $0$-dimensional (and hence a collection of a finite number of points),
or 
if there is a linear order of the facets $F_{1}, F_{2}, \ldots, F_{k}$
of $\Delta$, called a \emph{shelling order}, such that:
\begin{itemize}
\item[(i)]
The boundary complex of $F_{1}$ is shellable.
\item[(ii)]
For $1 < j \leq k$ the intersection of $\overline{F_{j}}$
with the union of the closures
of the previous facets is nonempty and is the beginning
of a shelling of the $(n-1)$-dimensional boundary complex
of~$F_{j}$, that is,
$$ \overline{F_{j}} \cap 
(\overline{F_{1}} \cup \overline{F_{2}} \cup \cdots \cup \overline{F_{j-1})}
= 
\overline{G_{1}} \cup \overline{G_{2}} \cup \cdots \cup \overline{G_{r}} , $$
where
$G_{1}, G_{2}, \ldots, G_{r}, \ldots, G_{t}$ 
is a shelling order of $\partial F_{j}$ and $r\geq 1$.
\end{itemize}
\end{subdef}

We then have the following result.

\begin{subthm} 
(\cite[Theorem~4.3.3]{OM}.)
Let $\Hf$ be a central hyperplane arrangement on $V$, let
$\Lf=\Lf(\Hf)$ and $\Tc=\Tc(\Hf)$, and
let $B$ be a chamber in $\Tc$.
Then any linear extension of the chamber poset with base chamber $B$
is a shelling order on the facets of $\Sigma(\Lf)$.
\label{thm_shelling_order}
\end{subthm}

Let $\Hf$ be a central hyperplane arrangement on $V$. We
write $\Lf=\Lf(\Hf)$ and $\Tc=\Tc(\Hf)$.
Let $B\in\Tc$, and let $B = T_1,T_2,\ldots,T_r = -B$ be a linear ordering of
$\Tc$ refining the partial order $\preceq_B$. 
By Theorem~\ref{thm_shelling_order}
the linear order
$T_1,T_2,\ldots,T_r$
is a shelling order of the chambers of~$\Sigma(\Lf)$.
In particular, the shelling order
defines a partition of the faces of~$\Sigma(\Lf)$:
\[\Lf
=
\coprod_{i=1}^r \{C\in\Lf :  C\leq T_i \text{ and }C\not\leq T_j \text{ for }1\leq j<i\} .\]
We will give a formula for the blocks of this partition
(see Proposition~\ref{prop_fibers_of_f_B}) which implies in
particular that the partition is independent of the linear
refinement of $\preceq_B$.

\begin{subdef}
Given a chamber $B\in\Tc$, we define a function $f_B$
from the face poset $\Lf$ to the set of chambers $\Tc$
by $f_B(C)=C\circ B$.
\end{subdef}

The next proposition gives some basic properties of the function $f_B$.

\begin{subprop}
\label{prop_f_B}
The following two statements hold:
\begin{itemize}
\item[(i)] Fix a face $C\in\Lf$ and consider the poset isomorphism
$\iota_C:\Lf_{\geq C} \stackrel{\sim}{\longrightarrow} \Lf_{\Hf(C)}$ of Lemma~\ref{lemma_star}. 
If $B\in\Tc\cap\Lf_{\geq C}$
then for every $D\in\Lf_{\geq C}$
we have $\iota_C(f_B(D))=f_{\iota_C(B)}(\iota_C(D))$.

\item[(ii)] Suppose that $\Hf$ is a Coxeter arrangement with a
chamber $B$ that is on the positive side of every hyperplane, and let
$(W,S)$ be the associated
Coxeter system. 
Identify $\Lf$ with the Coxeter complex $\Sigma(W)$ as in
Subsection~\ref{section_Coxeter_arrangements}. If the face $C\in\Lf$ corresponds
to a standard coset $c\subset W$, then the element $w\in W$ corresponding
to the chamber $f_B(C)$ is the shortest element of $c$ and also the
minimal element of the coset $c$ in the right weak Bruhat order.
\end{itemize}
\end{subprop}
In particular, part (ii) implies that, for the type~$A$ Coxeter
complex, the function $f_B$ defined here (for $B$ the chamber
corresponding to $1\in W$) is equal to the function $f$ defined at the
beginning of~\cite[Section~4]{Ehrenborg_Morel_Readdy}.
Note that the existence of a minimal element in every
standard coset is proved
in~\cite[Proposition~2.4.4]{BB}.

\begin{proof}[Proof of Proposition~\ref{prop_f_B}.]
Statement (i) follows immediately from Lemma~\ref{lemma_star}.

We now prove (ii).
By definition of the composition~$\circ$, the chamber $f_B(C)=C\circ B$
is the element of $\Tc\cap\Lf_{\geq C}$ closest to $B$ in the chamber
graph; in other words, it is the minimal element of $\Tc\cap\Lf_{\geq C}$
for the order $\preceq_B$; see Subsection~\ref{background}.
As we know that $\preceq_B$ corresponds 
to the right weak
Bruhat order on $W$ (see the discussion after Theorem~\ref{thm_W_finite}),
and as the elements of $W$ corresponding to the chambers of
$\Tc\cap\Lf_{\geq C}$ are the elements of the coset $c$,
the result follows.
\end{proof}

The link between the function $f_B$ and the shellings of
Theorem~\ref{thm_shelling_order} is established in the following
proposition. For the type~$A$ Coxeter complex, this result appeared
implicitly in the proof of~\cite[Proposition~4.1]{Ehrenborg_Morel_Readdy}.

\begin{subprop}
Let $B\in\Tc$, and let $B = T_1,T_2,\ldots,T_r = -B$ be a linear ordering of
$\Tc$ refining the partial order $\preceq_B$. Then for every
index $1 \leq i \leq r$ the fiber of $f_B$ over $T_i$ is given by
\begin{align}
f_B^{-1}(T_i)
& =
\{C\in\Lf : C\subset\overline{T}_i
-
(\overline{T}_1 \cup \overline{T}_2 \cupdots \overline{T}_{i-1})\}
\label{equation_fiber_of_f_B_1}
\\
& =
\{C\in\Lf :  C\leq T_i \text{ and }C\not\leq T_j \text{ for }1\leq j<i\}.
\label{equation_fiber_of_f_B_2}
\end{align}
\label{prop_fibers_of_f_B}
\end{subprop}

\begin{proof}
The equivalence between equalities~\eqref{equation_fiber_of_f_B_1}
and~\eqref{equation_fiber_of_f_B_2}
is an immediate consequence of the definition of
the order $\leq$ on $\Lf$. Let us prove
equality~\eqref{equation_fiber_of_f_B_2}.

Let $C\in f_B^{-1}(T_i)$, that is, $C\circ B=T_i$.
In particular, we have $C\leq C\circ B=T_i$.
Suppose that $T\in\Tc$ is another chamber such that $C\leq T$.
Then for every $e\in S(B,T_i)$ we have $s(T_i)_e \neq s(B)_e$, but
$s(C\circ B)_e=s(T_i)_e$, so $0 \neq s(C)_e=s(T_i)_e$. As $C\leq T$, this
implies that 
$s(T)_e = s(C)_e = s(T_i)_e \neq s(B)_e$,
hence that $e\in S(B,T)$.
So we have proved that $S(B,T_i)\subseteq S(B,T)$, which means that
$T_i\preceq_B T$. In particular, if $1 \leq j \leq i-1$ then $C\not\leq T_j$.

Conversely, let $C\in\Lf$ be such that
$C\leq T_i$ and $C\not\leq T_j$ for $1\leq j<i$, and let $T=C\circ B$.
If $e\in S(B,T)$ then $0 \neq s(C)_e=s(T)_e$.
As $C\leq T_i$, this implies that $s(C)_e = s(T_i)_e$, so
$s(T_i)_e = s(C)_e = s(T)_e  \neq  s(B)_e$,
that is, $e\in S(B,T_i)$.
So we have proved that $S(B,T) \subseteq S(B,T_i)$, which means
that $T\preceq_B T_i$. Hence there exists an index $1 \leq i' \leq i$ such
that $T=T_{i'}$. As $C\leq T$ and $C\not\leq T_j$ for $1\leq j<i$, we must
have $i'=i$, that is, $f_B(C) = C\circ B = T = T_i$.
\end{proof}

\subsection{A condition on hyperplane arrangements}
\label{section_condition_A}

We now introduce a geometric condition on the hyperplane arrangement $\Hf$ that
will imply the shellability of the weighted complex.

\begin{conditionA}
Denote by \setword{(A)}{(A)}
the following condition on the family
$(\alpha_e)_{e \in E}$ (or the corresponding arrangement):
For every $T\in\Tc$ and for every $e \in E$ such that $S = \overline{T}\cap
H_{\alpha_e}$ is of dimension $\dim(V)-1$, that is, $S$ is a facet of the
convex cone $\overline{T}$, the following inclusions hold:
\begin{align*}
T\subseteq\mathring{S}+\R_{> 0}\alpha_e \text{ if } T\subseteq H^+_{\alpha_e}, \\
T\subseteq\mathring{S}+\R_{< 0}\alpha_e \text{ if } T\subseteq H^-_{\alpha_e},
\end{align*}
where $\mathring{S}$ is the relative interior of the cone $S$, that is, the interior
of $S$ in $\Span(S)$.
\end{conditionA}

Geometrically, Condition~\ref{(A)} means that if $T\in\Tc$
then the dihedral angle between any two adjacent facets 
(facets whose intersection is a face of
dimension $\dim(V)-2$)
of the convex polyhedral cone~$\overline{T}$ is acute, 
that is, less than or equal to $\pi/2$.

\begin{subprop}
\label{prop_crossing_wall}
Suppose that the arrangement $\Hf$ satisfies Condition~\ref{(A)}.
Let $T,T'\in\Tc$ and $e \in E$ such that $S(T,T')=\{e\}$,
the inner product $(\alpha_e,\lambda)$ is nonnegative 
and the inclusion $T' \subset H^-_{\alpha_e}$ holds.
Then $T'\in\Lf_{\lambda}$ implies that $T\in\Lf_\lambda$.
\end{subprop}
\begin{proof}
The hypothesis implies that
$\overline{T}\cap\overline{T}'
=
\overline{T}\cap H_{\alpha_e}
=
\overline{T}'\cap H_{\alpha_e}$.
We denote this intersection by~$S$. It is a facet of both $\overline{T}$ and
$\overline{T}'$. By Condition~\ref{(A)}, we have
$T \subset \mathring{S}+\R_{> 0}\alpha_e$
and 
$T' \subset \mathring{S}+\R_{< 0}\alpha_e$.
In particular, if $x\in T$ then there exists $c>0$ such that
$x - c \cdot \alpha_e \in T'$. Then we have
$(x,\lambda)=(x- c \cdot \alpha_e,\lambda)+c \cdot (\alpha_e,\lambda)\geq 0$.
This implies that $T\subset\overline{H_\lambda^+}$,
that is, $T\in\Lf_{\lambda}$.
\end{proof}

\begin{subcor}
Suppose that the arrangement $\Hf$ satisfies Condition~\ref{(A)}.
If $(\lambda,\alpha_e)\geq 0$ for every $e \in E$ and if
there exists $B\in\Tc$ such that $B\subset H^+_{\alpha_e}$
for every $e \in E$, then $\Tc\cap \Lf_{\lambda}$ is a lower order
ideal in $\Tc_B$.
More generally, if $C\in\Lf$ and $E(C)=\{e\in E :  C\subset H_{\alpha_e}\}$,
if $(\lambda,\alpha_f)\geq 0$ for every
$f\in E(C)$ and if
there exists $B\in\Tc\cap\Lf_{\geq C}$ such that $B\subset H^+_{\alpha_f}$
for every $f\in E(C)$, then $\Tc\cap\Lf_{\lambda,\geq C}$ is a lower order
ideal in $\Tc_B$.
\label{cor_lower_order_ideal}
\end{subcor}
\begin{proof}
It suffices to prove the second statement.
Let $T,T'$ be
such that $S(B,T)\subset S(B,T')$ and $T' \in \Lf_{\lambda,\geq C}$.
We want to show that $T\in\Lf_{\lambda,\geq C}$. As $\Tc_B$
is a graded poset and the intersection $\Tc\cap\Lf_{\geq C}$ is a lower order ideal in
$\Tc_B$ (see Remark~\ref{rmk_star_convex}), 
we know that $T\in\Tc\cap\Lf_{\geq C}$, and
it suffices to treat the case where $S(T',B)-S(T,B)$
is a singleton. 
Let $f$ be the single index of $S(B,T')-S(B,T)$.
As $B,T'\in\Tc\cap\Lf_{\geq C}$, we have $f\in E(C)$ by
Lemma~\ref{lemma_star}(iii), so $B\subset H^+_{\alpha_f}$.
As $f \in S(B,T')-S(B,T)$, we have $T'\subset H^-_{\alpha_f}$
and $T\subset H^+_{\alpha_f}$. Also, as $f \in E(C)$, we have
$(\lambda,\alpha_f) \geq 0$.
So we may
apply Proposition~\ref{prop_crossing_wall}, and we obtain that $T\in\Lf_{\lambda}$.
\end{proof}

\begin{subcor}
Suppose that the arrangement $\Hf$ satisfies Condition~\ref{(A)}.
Then the complex~$\Sigma(\Lf_\lambda)$ is shellable.
Moreover, there exists a shelling order on its chambers
which is an initial shelling of~$\Sigma(\Lf)$.
In particular, if $\lambda \neq 0$ then $\Sigma(\Lf_\lambda)$
is a shellable $PL$ ball of dimension $\dim(V/V_0)-1$.
\label{cor_shelling_order}
\end{subcor}
\begin{proof}
If $\lambda=0$ then $\Lf_\lambda=\Lf$ and $\Sigma(\Lf_\lambda)=\Sigma(\Lf)$,
and the corollary is just Theorem~\ref{thm_shelling_order}.

We now assume that $\lambda \neq 0$.
By Theorem~\ref{thm_shelling_order} and
Corollary~\ref{cor_lower_order_ideal}, it suffices to find a family
of signs $(\varepsilon_e) \in \{\pm 1\}^E$ such that:
\begin{itemize}
\item[--] for every $e \in E$, we have $(\lambda,\varepsilon_e\alpha_e)\geq 0$;
\item[--] there exists a chamber $B\in\Lf_\lambda$ with $B\subset
H^+_{\varepsilon_e\alpha_e}$ for every $e \in E$.
\end{itemize}
Indeed, Corollary~\ref{cor_lower_order_ideal} will then imply that
$\Tc\cap\Lf_\lambda$ is a lower order ideal in $\Tc_B$, so it will be an
initial segment for at least one linear extension of $\preceq_B$.

Let $F = \{e\in E : (\lambda,\alpha_e) \neq 0\}$. For every $e\in F$,
we choose $\varepsilon_e \in \{\pm 1\}$ such that
$(\lambda,\varepsilon_e\alpha_e)>0$.
Let $x_0$ be a point in $V$ not on any hyperplane of $\Hf$,
that is, $x_0\in V - \bigcup_{e \in E} H_{\alpha_e}$.
Then for every $e \in F$, the inner product
$(x_0 + c \cdot \lambda,\varepsilon_e\alpha_e)
= (x_0,\varepsilon_e\alpha_e)+c \cdot (\lambda,\varepsilon_e\alpha_e)
$ tends to $+\infty$ as $c$ tends to $+\infty$, so it is positive
for $c$ large enough. Similarly, the inner product
$(x_0+c \cdot \lambda,\lambda)=(x_0,\lambda)+c \cdot (\lambda,\lambda)$ is positive
for $c$ large enough. On the other hand, if $e\in E -  F$, then
$(x_0+c \cdot \lambda,\alpha_e)=(x_0,\alpha_e) \neq 0$ for every $c\in\R$.
So, if $c\in\R$ is large enough, then $x=x_0+c \cdot \lambda\in V - 
\bigcup_{e \in E} H_{\alpha_e}$, and $x$ is in 
$H^+_{\varepsilon_e\alpha_e}$ for every $e\in F$ and in $H^+_\lambda$. 
In particular, there exists a chamber
$B\in\Tc$ such that $x\in B$, and $B$ is included in
$H^+_{\varepsilon_e\alpha_e}$ for every $e\in F$ and in $H^+_\lambda$. 
Now, if $e\in E -  F$, we choose $\varepsilon_e\in\{\pm 1\}$ such that
$B\subset H^+_{\varepsilon_e\alpha_e}$.
As $(\lambda,\alpha_e)=0$, we clearly
have $(\lambda,\varepsilon_e\alpha_e)\geq 0$.
\end{proof}

\subsection{The case of Coxeter arrangements}
\label{section_weighted_complex_Coxeter}

\begin{sublemma}\label{lemma_Coxeter_A}
Every Coxeter arrangement $\Hf$ satisfies Condition~\ref{(A)}.
\end{sublemma}
\begin{proof}
In a Coxeter arrangement, 
the dihedral angle between any two adjacent facets 
is $\pi/n$, with $n\geq 2$.
\end{proof}

In particular, Corollaries~\ref{cor_lower_order_ideal}
and~\ref{cor_shelling_order} apply to Coxeter arrangements.
But
we can actually prove a stronger result
in this case.

We fix a Coxeter arrangement $\Hf$ on an inner product space $V$,
and we use the notation introduced above.
We say that
a vector $\lambda\in V$ is \emph{dominant} if
$(\lambda,\alpha)\geq 0$ for every $\alpha\in\Phi^+$.
\begin{sublemma}
Suppose that $\lambda\in V$ is dominant.
Denote by $B$ the chamber of $\Hf$ corresponding to $1\in W$.
Let $z,w \in W$ such that
$z \leq w$ in the strong Bruhat order of~$W$.
Then for every $x\in B$ the following inequality holds:
\[
(z^{-1}(\lambda),x)  \geq  (w^{-1}(\lambda),x) .
\]
\label{lemma_Bruhat_order_implies}
\end{sublemma}
\begin{proof}
We may assume that $w$ covers $z$, so that there exists
$s\in S$ and $u\in W$ such that $w=(usu^{-1})z$. Let $\alpha$ be the
unique pseudo-root of $\Phi^+$ such that $u(e_s)$ is a multiple (positive
or negative) of~$\alpha$. If $s_\alpha$ is the reflection across
$H_\alpha$, we have $usu^{-1}=s_\alpha$, and so $w=s_\alpha z$
and $s_\alpha w = z$.
Since the elements of~$\Phi^+$ are unit vectors,
$s_\alpha$ is given by the following formula:
$s_{\alpha}(\mu)
=
\mu - 2\cdot (\mu,\alpha) \cdot \alpha$
for $\mu \in V$.
Hence
\begin{align*}
(s_{\alpha} w)^{-1}(\lambda)
& =
(w^{-1} s_{\alpha})(\lambda)
=
w^{-1}(\lambda)
- 2\cdot (\lambda,\alpha) \cdot w^{-1}(\alpha), \\
\intertext{and so, if $x \in B$,}
((s_\alpha w)^{-1}\lambda,x)
& =
(w^{-1}(\lambda),x)
-
2\cdot(\lambda,\alpha)
\cdot
(w^{-1}(\alpha),x) .
\end{align*}
As $\lambda$ is dominant, we have $(\lambda,\alpha) \geq 0$.
By \cite[Equation~(4.25)]{BB}, we have
the equivalence $w^{-1}(\alpha) \in \Phi^+
\Longleftrightarrow \ell(w^{-1}s_\alpha) > \ell(w^{-1})$,
and~\cite[Proposition~1.4.2(iv)]{BB}
states that $\ell(v^{-1}) = \ell(v)$ for every $v\in W$.
Using these two facts 
and the condition $\ell(s_{\alpha} w)<\ell(w)$, we see that
$w^{-1}(\alpha) \in \Phi^-$.
Thus $(w^{-1}(\alpha),x) < 0$ by definition of $B$. 
Hence the term $-2\cdot (\lambda,\alpha) \cdot (w^{-1}(\alpha),x)$
is nonnegative, that is, 
$(z^{-1}(\lambda),x)=((s_{\alpha} w)^{-1}(\lambda),x) \geq (w^{-1}(\lambda),x)$.
\end{proof}

\begin{subprop}
\label{prop_strong_Bruhat}
Let $(W,S)$ be a Coxeter system, and let $\Hf=(H_\alpha)_{\alpha\in\Phi^+}$
be the associated hyperplane arrangement on the space $V$ of the canonical
representation of $(W,S)$.
Let $\lambda\in V$ be a dominant vector.
Then the set $W_\lambda$ of $w\in W$ such that the corresponding
chamber of $\Hf$ is in $\Tc\cap\Lf_\lambda$ is a lower order ideal
with respect to
the \emph{strong} Bruhat order on $W$.
\end{subprop}
\begin{proof}
We denote by $B$ the chamber of $\Hf$ corresponding to $1\in W$. 
By definition of $W_\lambda$, an element $w$ of $W$ is in $W_\lambda$ if and
only if for every $x\in B$ we have $(\lambda,w(x))=(w^{-1}(\lambda),x)\geq 0$.
By Lemma~\ref{lemma_Bruhat_order_implies}, if $z,w \in W$
and $w$ is greater than $z$ in the strong Bruhat order then for
every $x\in B$, we have
$(z^{-1}(\lambda),x)  \geq  (w^{-1}(\lambda),x)$. If moreover
$w\in W_\lambda$, this immediately implies that $z \in W_\lambda$.
\end{proof}

\section{Concluding remarks}
\label{section_concluding_remarks}

As mentioned in the introduction, we are not aware
of whether 
there is a representation-theoretic interpretation
of the identity in
Theorem~\ref{thm_second_main} in general.
More precisely, what is the meaning
of the constants $\psi_{\Hf/C}(B,\lambda)$
for different values of $\lambda$?

The main results in this paper are in the setting of
Coxeter arrangements.
However, the sum~$\psi_\Hf$ makes sense for general hyperplane
arrangements, our original proof of Theorem~\ref{thm_second_main}
used an induction formula that linked the sum $\psi_\Hf$
to similar sums for subarrangements of the restricted arrangements
on the hyperplanes of $H$, and this induction formula is valid
for general hyperplane arrangements, and with some adaptions,
for oriented matroids.
In the paper \cite{EMR_pizza}, we use a similar type of induction
argument, this
time to calculate the alternating sum of another
valuation on the chambers of a hyperplane arrangement,
that is, the volume
of the intersection of the chamber with some set of finite volume.
Is there is some analogue
of $2$-structures for more general hyperplane arrangements?

It is natural to ask is there some analogue of
Theorem~\ref{theorem_2_structures_and_chambers} for Coxeter
systems with possible infinite Coxeter groups?

In Section~\ref{section_weighted_complex_shellable} we prove
that the weighted complexes of a hyperplane arrangement
are shellable under a geometric condition on the arrangement
that we call Condition~\ref{(A)}. This implies that the weighted complexes
are PL balls for arrangements satisfying Condition~\ref{(A)}.
Are the weighted complexes always PL balls?
By Remark~\ref{rmk_weighted_vs_bounded},
this extends a conjecture of Zaslavsky
(see~\cite[Chapter~I, Section~3C, p.\ 33]{Zaslavsky})
that the bounded complex of a simple hyperplane arrangement
is always a PL ball.
As a consequence to Corollary~\ref{cor_shelling_order},
we have the following result.
\begin{corollary}
Zaslavsky's conjecture holds
for affine arrangements obtained
by intersecting 
an affine hyperplane with
an arrangement satisfying Condition~\ref{(A)}.
\end{corollary}

In a paper of Dong (see~\cite{Dong})
he claims to have proven Zaslavsky's conjecture. However, we do not
understand the proof of the crucial Lemma~4.7 in that paper:
In the second paragraph of case~2, Dong chooses a linear
extension $\leq$
of $\mathcal{T}(\Lf/g,d_i)$ such that $[d_i,d_j]$ is
an initial segment. This linear extension is a shelling order,
and Dong deduces that there exists $d_k\in [d_i,d_j)$ such that
$d_k\wedge d_j\precdot d_j$. But the only thing that
we can deduce from the fact that we have a shelling order
is that $d_i\leq d_k< d_j$, which does not imply that
$d_k\in[d_i,d_j)$ for the order on $\mathcal{T}(\Lf/g,d_i)$.
If we do not know that $d_k\in[d_i,d_j)$, the rest
of the argument fails.

In Appendix~\ref{appendix_valuations} we
construct a ring structure on the set of valuations on convex closed
polyhedral cones with values in a fixed ring.
Are there other products of
valuations that also yield valuations?
For instance, can the ring structure of
Corollary~\ref{cor_convolution}
be extended to valuations on (not necessarily polyhedral) cones
in Euclidean space?

In Appendix~\ref{appendix_2_structures}
the proof of
Proposition~\ref{prop_W_acts_transitively_on_T_Phi}
that the group $W$ acts transitively on
the set of $2$-structures~$\Tt$
consists of verifying the result for all
irreducible pseudo-root systems.
Is there a general proof that does not use the classification
of irreducible pseudo-root systems?

\appendix

\section{Extending the construction of a valuation by Goresky, Kottwitz and MacPherson}
\label{appendix_valuations}

We introduce a ring structure on the set of valuations defined
on closed convex polyhedral cones in a finite-dimensional
real vector space with values in a fixed ring.
As a special case we obtain in Corollary~\ref{cor_psi_C_is_a_valuation}
a valuation due to Goresky, Kottwitz and MacPherson;
see~\cite[Proposition~A.4]{GKM}.

Subsection~\ref{subsection_convolution} of this appendix
contains definitions and statements of results.
The proofs are relegated to Subsection~\ref{subsection_convolution_proofs}.

\subsection{The ring of valuations}
\label{subsection_convolution}

Let $V$ be a finite-dimensional real vector space and $V^\vee$ its dual.
A \emph{closed convex polyhedral cone} in $V$ is a nonempty subset
of the form $\R_{\geq 0} v_{1} + \R_{\geq 0} v_{2} + \cdots + \R_{\geq 0} v_{k}$,
where $v_{1}, v_{2}, \ldots, v_{k} \in V$ and $k\geq 0$.

For a subset $X$ of the space~$V$, define
$X^{\perp} = 
\{ \alpha \in V^{\vee} : \forall x \in X \: \bl{\alpha}{x} = 0 \}$
and
$X^{*} = 
\{ \alpha \in V^{\vee} : \forall x \in X \: \bl{\alpha}{x} \geq 0 \}$.
Note that $X^{\perp}$ is a subspace of $V^{\vee}$
and depends only on the linear span of $X$,
and that $X^*$ is a convex cone in $V^{\vee}$ and depends only on
the closed convex polyhedral cone generated by $X$.

For $F$ a face\footnote{In this appendix, we take all faces to be closed faces,
unlike in the rest of the article.}
of a closed convex polyhedral cone $K$,
define $F^{\perp,K} = F^{\perp} \cap K^{*}$.
The map $F \longmapsto F^{\perp,K}$
is an order-reversing bijection from the set $\F(K)$
of faces of $K$ to the set of faces
of $K^*$.
This statement and other basic properties of closed convex polyhedral
cones are proved in~\cite[Section~1.2]{Fulton}.

\begin{subrmk}
For two  closed convex cones $X_1$ and $X_2$ such that $X_1\cup X_2$ is
convex then the set $X_1^*\cup X_2^*$ is also convex and we have the two identities
\[(X_1\cup X_2)^*=X_1^*\cap X_2^*
\quad \text{ and } \quad
(X_1\cap X_2)^* = X_1^* \cup X_2^*.\]
\label{rmk_dual_sum}
\end{subrmk}

\begin{subdef}
We denote by $\Cf(V)$ the set of closed
convex polyhedral cones in $V$.
Denote the free abelian group on $\Cf(V)$
by $\bigoplus_{K\in \Cf(V)}\Z[K]$
and
let $K(V)$ be its quotient by
the relations 
$[K\cup K']+[K\cap K']=[K]+[K']$ for all $K,K'\in \Cf(V)$
such that $K\cup K'$. For $K\in \Cf(V)$, we still denote its
image in $K(V)$ by $[K]$.
\label{def_K(V)}
\end{subdef}

For $\lambda\in V^\vee$, we define the hyperplane
$H_{\lambda}$ and the two open half-spaces
$H^{+}_\lambda$ and $H^{-}_\lambda$ by
\[
H_\lambda=\{x\in V : \bl{\lambda}{x}=0\},
\quad
H^{+}_\lambda=\{x\in V : \bl{\lambda}{x} > 0\}
\quad
\text{ and }
\quad
H^{-}_\lambda=\{x\in V : \bl{\lambda}{x} < 0\}.
\]
The closed half-spaces are given by
$\overline{H_\lambda^+} = \{x\in V : \bl{\lambda}{x} \geq 0\}$
and
$\overline{H_\lambda^-} = \{x\in V : \bl{\lambda}{x} \leq 0\}$.

\begin{subdef}
A \emph{valuation} on $\Cf(V)$ with values in an
abelian group $A$ is a function
$f:\Cf(V) \longrightarrow A$ such that $f(\varnothing)=0$
and that for any $K,K'\in\Cf(V)$ such that $K\cup K'\in\Cf(V)$, we have
\begin{align}
f(K\cup K')+f(K\cap K') & = f(K)+f(K').
\end{align}
\label{def_valuation}
\end{subdef}

By the definition of $K(V)$, saying that a function $f:\Cf(V)\fl A$
is a valuation is equivalent to saying that there exists a morphism
(necessarily unique) $K(V)\fl A$ sending $[K]$ to $f(K)$ for every
$K\in \Cf(V)$. We also denote this morphism $K(V)\fl A$ by $f$.

\begin{subex}
By Remark~\ref{rmk_dual_sum},
the function $\delta:\Cf(V)\fl K(V^\vee)$ sending $K\in\Cf(V)$ to
$[K^*]$ is a valuation. Thus it induces a morphism
$\delta:K(V)\fl K(V^*)$. 
\label{example_delta}
\end{subex}

We have the following criterion for recognizing valuations
on closed convex polyhedral cones.
This is known as {\em Groemer's first extension theorem}
and is proved in~\cite[Theorem~2]{Gr}.

\begin{subthm}[Groemer]
Let $A$ be an abelian group and
$f:\Cf(V) \longrightarrow A$ be a function such that $f(\varnothing)=0$.
Suppose that for every $K\in\Cf(V)$ and every $\mu \in V^\vee$ the
following holds:
\begin{align}
f(K)+f(K\cap H_\mu)
& =
f\bigl(K\cap \overline{H_\mu^{+}}\bigr)
+
f\bigl(K \cap \overline{H_\mu^{-}}\bigr).
\label{equation_Groemer}
\end{align}
Then the function $f$ is a valuation.
\label{theorem_Groemer}
\end{subthm}

The main result of this appendix is the following theorem
whose proof is in Subsection~\ref{subsection_convolution_proofs}.
To make the notation more compact in this appendix, we denote the linear span
of a subset $S$ of vector space by $\spam{S}$.

\begin{subthm}
\begin{enumerate}
\item[(i)] Consider the function $\Delta:\Cf(V)\fl K(V)\otimes_\Z K(V)$ defined
by
\[\Delta(K)=\sum_{F\in\F(K)}(-1)^{\dim F}[F]\otimes[\spam{F}+K],\]
for every $K\in \Cf(V)$. Then $\Delta$ is a valuation and it
induces a morphism $\Delta:K(V)\fl K(V)\otimes_\Z K(V)$.
Moreover this morphism $\Delta$ is coassociative,
that is, we have
\[(\Delta\otimes\id_{K(V)})\circ\Delta=(\id_{K(V)}\otimes\Delta)\circ\Delta.\]

\item[(ii)] Consider the function $\varepsilon:\Cf(V)\fl\Z$ defined by
$\varepsilon(K)=(-1)^{\dim K}$ if
$K$ is a vector subspace of~$V$ and $\varepsilon(K)=0$ otherwise. Then
$\varepsilon$ is a valuation and it
induces a morphism $\varepsilon:K(V)\fl\Z$.
This morphism is a counit of $\Delta$, in other words, we have
\[(\varepsilon\otimes\id_{K(V)})\circ\Delta
= \id_{K(V)} =
(\id_{K(V)}\otimes\varepsilon)\circ\Delta.\]
\end{enumerate}
\label{thm_coalgebra}
\end{subthm}

In short, Theorem~\ref{thm_coalgebra} says that the morphisms $\Delta:K(V)\fl
K(V)\otimes_\Z K(V)$ and $\varepsilon:K(V)\fl\Z$ are well-defined and make
$K(V)$ into a $\Z$-coalgebra.

\begin{subcor}
Let $A$ be a ring.
Let $f_1,f_2:\Cf(V)\fl A$ be two valuations. Then the function
$f_1*f_2:\Cf(V)\fl A$ defined by
\[(f_1*f_2)(K)=\sum_{F\in\F(K)}(-1)^{\dim F}f_1(F)f_2(\spam{F}+K)\]
is also a valuation.
This operation $*$ makes 
the group of valuations $\Cf(V)\fl A$ into a ring.
The unit element of this ring
is the composition of $\varepsilon:\Cf(V)\fl\Z$
and the canonical ring morphism $\Z\fl A$.
\label{cor_convolution}
\end{subcor}
\begin{proof}
The valuations $f_1$ and $f_2$ induce two morphisms
$f_1,f_2:K(V)\fl A$, hence a morphism
$f_1\otimes f_2:K(V)\otimes_\Z K(V)\fl A$, $x\otimes y\mapsto f_1(x)\otimes
f_2(y)$. As we have
\[(f_1*f_2)(K)=(f_1\otimes f_2)(\Delta([K]))\]
for every $K\in\Cf(V)$, this shows that $f_1*f_2$ descends to a morphism
$K(V)\fl A$, hence is a valuation.

The operation $*$ is clearly linear in each variable, and it is
associative by
the coassociativity of~$\Delta$.
The last statement follows immediately from the fact that
$\varepsilon$ is a counit of $\Delta$.
\end{proof}

\begin{subcor}
Let $A$ be a ring, and let $f_1:\Cf(V)\fl A$ and $g_2:
\Cf(V^\vee)\fl A$ be valuations. Then the function
$f_1\star g_2:\Cf(V)\fl A$ defined by
\[(f_1\star g_2)(K)=\sum_{F\in\F(K)}(-1)^{\dim F}f_1(F)g_2(F^{\perp,K})\]
is also a valuation.
\label{cor_second_convolution}
\end{subcor}
\begin{proof}
Consider the valuation $\delta:\Cf(V)\fl K(V^\vee)$ 
of Example~\ref{example_delta}.
Then the map $f_2=g_2\circ\delta:\Cf(V)\fl A$ is also a valuation.
As $F^{\perp,K}=(\spam{F}+K)^*$ for every $K\in\Cf(V)$ and every face $F$ of
$K$, we have $f_1\star g_2=f_1*f_2$, so the statement
follows from Corollary~\ref{cor_convolution}.
\end{proof}

\begin{subrmk}
Let $\Hf$ be a central hyperplane arrangement on $V$, 
let $\Cf_\Hf(V)$ be the set of closed convex polyhedral cones that are
intersections of closed half-spaces bounded by hyperplanes of
$\Hf$, and let $K_\Hf(V)$ be the quotient of the free abelian group 
$\bigoplus_{K\in \Cf_\Hf(V)}\Z[K]$
on $\Cf_\Hf(V)$ by the relations $[\varnothing]=0$ and
$[K]+[K']=[K\cup K']+[K\cap K']$ for all
$K,K'\in \Cf_\Hf(V)$ such that $K\cup K'\in \Cf_\Hf(V)$.
Then the formulas of Theorem~\ref{thm_coalgebra} also define
a coalgebra structure on $K_\Hf(V)$. Indeed, if $K\in\Cf_\Hf(K)$ and
$F\in\F(K)$, then $F$ and $\spam{F}+K$ are also in $\Cf_\Hf(V)$.

In particular, the products in Corollaries~\ref{cor_convolution}
and~\ref{cor_second_convolution} also make sense if the first
valuation is only defined on $\Cf_\Hf(V)$.

\label{rmk_K_Hf}
\end{subrmk}

We now explain how to use Corollary~\ref{cor_second_convolution}
to recover~\cite[Proposition~A.4]{GKM}.

\begin{sublemma}
Let $X$ be a subset of $V$ such that the complement $V - X$ is convex.
Then the function $\phi_X: \Cf(V) \longrightarrow \Z$ 
defined by
$$
\phi_X(K)
=
\begin{cases}
1 & \text{ if } \emptyset \subsetneq K\subseteq X, \\
0 & \text{ otherwise,}
\end{cases}
$$
is a valuation.
In particular, if $\lambda\in V^\vee$ then the function
$\psi_\lambda=\phi_{\overline{H_\lambda^+\!\!}}$
is a valuation.
\label{lemma_GKM}
\end{sublemma}
\begin{proof}
Let $K\in\Cf(V)$ be nonempty and let $\mu\in V^\vee$. Let $K_{0}=K\cap
H_\mu$, $K_{+}=K\cap \overline{H_\mu^{+}}$
and
$K_{-}=K\cap \overline{H_\mu^{-}}$.
We must check Criterion~\eqref{equation_Groemer}
in Theorem~\ref{theorem_Groemer},
that is,
$\phi_X(K) + \phi_X(K_{0})
=
\phi_X(K_{+}) + \phi_X(K_{-})$.

If $K\subseteq X$
then $K_{0}$, $K_{+}$ and $K_{-}$ are also included in $X$, and the
equality above is clear. If $K_{+}\subseteq X$ but $K_{-} \nsubseteq X$,
then $K_{0}\subseteq X$ and $K \nsubseteq X$,
so again the desired equality holds. The case where $K_{-}\subseteq X$
and $K_{+} \nsubseteq X$ is symmetric.
Finally, suppose
that $K_{+},K_{-} \nsubseteq X$. Then $K \nsubseteq X$, and
so we must show that $K_{0} \nsubseteq X$. Take $x\in K_{+}-X$
and $y\in K_{-}-X$. Then the segment $[x,y]$ is contained in
the convex set $V-X$. As this segment intersects
$K_{0}$, this shows that $K_{0} \nsubseteq X$.
\end{proof}

Given $x \in V$ and $\lambda \in V^{\vee}$, we have
two valuations $\psi_\lambda:\Cf(V)\fl\Z$ and
$\psi_x:\Cf(V^\vee)\fl\Z$ defined in
Lemma~\ref{lemma_GKM}.
Let $K\longmapsto\psi_{K}(x,\lambda)$
be the function defined by
\begin{align*}
\psi_{K}(x,\lambda) & = (\psi_{\lambda} \star \psi_{x})(K)
\end{align*}
for every $K\in\Cf(V)$.
This function is defined in~\cite[Appendix~A]{GKM} (at the top of page~540).

\begin{subcor}
For every $x\in V$ and
every $\lambda\in V^\vee$, the function 
$K \longmapsto\psi_K(x,\lambda)$
from~$\Cf(V)$ to~$\R$ is a valuation.
\label{cor_psi_C_is_a_valuation}
\end{subcor}

Since any valuation satisfies the
additivity property, we obtain the next corollary~\cite[Proposition~A.4]{GKM}.
\begin{subcor}
[Goresky--Kottwitz--MacPherson]
Let $K$ be a closed convex polyhedral cone.
Suppose that its relative interior $\interior{K}$ is the disjoint union of
the relative interiors 
$\interior{K_1}, \interior{K_2}, \ldots, \interior{K_r}$
of $r$ closed convex polyhedral cones
$K_1, K_2, \ldots, K_r$.
Then for every $x\in V$ and every $\lambda\in V^\vee$
\[
     \psi_K(x,\lambda) = \sum_{i=1}^r (-1)^{\dim(K) - \dim(K_i)} \cdot 
\psi_{K_i}(x,\lambda).
\]
\end{subcor}

\begin{subrmk}
Valuations
on $\Cf(V)$ can be extended to relatively open 
cones as well.
Let $G$ be a collection of sets
that is closed under finite intersections.
Define $B(G)$ to be the Boolean algebra
generated by $G$, that is, the smallest collection
of sets that contains $G$ and is closed
under finite unions, finite intersections and complements.
Groemer's Integral Theorem states that
a valuation on $G$ can be extended
to a valuation on the Boolean algebra $B(G)$; see~\cite{Gr}
and also~\cite[Chapter~2]{Klain_Rota}.
In the case where $G = \Cf(V)$,
that is, 
the collection of closed convex polyhedral cones in $E$,
the associated Boolean algebra $B(\Cf(V))$ contains all cones that
are obtained by intersecting closed and open half-spaces.
\label{rmk_extension_valuation}
\end{subrmk}

\begin{subrmk}
The results of this appendix extend to oriented matroids without
much change. Let~$E$ be a finite set and consider an oriented matroid
$\Mf$
on $E$ with set of covectors $\Lf \subseteq \{+,-,0\}^E$.
This set of covectors forms a graded poset
with the partial order given by componentwise 
comparing the entries by $0 < +$ and $0 < -$.
We denote its rank function by $\rho$.
For every $F\subseteq E$ and every $s\in\{+,-,0\}^{F}$, we write
$\Lf_{\leq s}=\{x\in\Lf : x|_{F} \leq s\}$.
Let $\Kf$ be the set of lower order ideals of~$\Lf$ of the form
$\Lf_{\leq s}$. We order $\Kf$ by inclusion.
If $\Lf$ is the
oriented matroid corresponding to a central hyperplane arrangement $\Hf$ on $V$,
then $\Kf$ is the set cones obtained by intersecting
closed half-spaces bounded by hyperplanes of $\Hf$.
In general, every element of $\Kf$ is of the form $\Lf_{\leq x|_{F}}$
for some $x\in\Lf$ and $F\subseteq E$.

We say that an element $a$ of $\Kf$ is a vector subspace if
$a \neq \varnothing$ and $a$ is of the form $\Lf_{\leq x|_{F}}$,
for some $x\in\Lf$ and some $F\subseteq E$ such that $x_e=0$ for
every $e\in F$. 

A \emph{valuation} on $\Kf$ with values in an abelian group $A$
is a function $f:\Kf\fl A$ such that $f(\varnothing)=0$ and, for all
$a,b\in\Kf$ such that $a\cup b\in\Kf$, 
we have $f(a\cup b)+f(a\cap b)=f(a)+f(b)$. 
Giving such a valuation is equivalent to giving a function
$w:\Lf\fl A$; the corresponding valuation then sends $a\in\Kf$ to
$\sum_{x\in a}w(x)$.

The analogue of $K(V)$ is the quotient of the free abelian group
$\bigoplus_{a\in\Kf}\Z [a]$ by the relations $[\varnothing]=0$ and
$[a\cup b]+[a\cap b]=[a]+[b]$ if $a\cup b\in\Kf$. 
We denote this group by $K(\Lf)$.
We have an isomorphism $K(\Lf)\iso\bigoplus_{x\in\Lf}\Z [x]$ sending
$[a]$ to $\sum_{x\in a}[x]$. 

Let $F\subseteq E$.
We denote the set of covectors of the
deletion $\Mf-(E-F)$
by $\Lf_F$ and the rank
function of $\Lf_F$ by $\rho_F$. Let $y\in\Lf_F$. 
If $F(y)=\{e\in F : y_e=0\}$,
then we have an order-preserving
bijection $\Lf_{F,\geq y}\iso\Lf_{F(y)}$ sending any $z\geq y$
in $\Lf_F$ to $z|_{F(y)}$. This is the analogue of Lemma~\ref{lemma_star}.

We define the comultiplication $\Delta:K(\Lf)\fl
K(\Lf)\otimes_\Z K(\Lf)$ by sending $[\Lf_{\leq x|_{F}}]$ to
\[
\Delta([\Lf_{\leq x|_{F}}])
=
\sum_{y\in\Lf_{F,\leq x|_{F}}}
(-1)^{\rho_F(y)}
[\Lf_{\leq y}] \otimes [\Lf_{\leq x|_{F(y)}}].
\]
Let $a\in\Lf$.
The counit $\varepsilon$ sends $[a]$ to $0$ if $a$ is not a vector
subspace. If $a=\Lf_{\leq x|_{F}}$, with $x\in\Lf$ and $F\subseteq E$ such
that $x_e=0$ for every $e\in F$, then we set $\varepsilon([a])=
(-1)^{\rho_F(x|_{F})}$.
\end{subrmk}

\begin{subrmk}
Let $\Kf=\bigoplus_{n\geq 0}K(\Rrr^n)$. We make $\Kf$ into a coalgebra
using the direct sum of the
morphisms $\Delta$ and $\epsilon$ of Theorem~\ref{thm_coalgebra}.
There is a product on
$\Kf$ defined by $[K]\cdot[L]=[K\times L]$ if $K\in\Cf(\Rrr^n)$
and $L\in\Cf(\Rrr^m)$, where we identity $\Rrr^n\times\Rrr^m$ and
$\Rrr^{n+m}$ in the usual way. This product is associative, and the
class of the cone $\{0\}\in
\Cf(\Rrr^0)$ is a unit. It is then straightforward to see that $\Kf$
is actually a bialgebra. However, it is not a Hopf algebra because
if $V$ is a vector subspace of~$\Rrr^n$ with $n\geq 1$, then
the element $(-1)^{\dim V}[V]$ of $\Kf$ is group-like but not invertible.

If $K\in\Cf(\Rrr^n)$ and $F$ is a face of $K$, then the poset of faces
of $\spam{F}+K$ is isomorphic to the interval $[F,K]$ in the poset of
faces of $K$. So the bialgebra $\Kf$ is related to the incidence Hopf
algebras defined by Joni and Rota in~\cite{Joni_Rota} and further studied
by Schmitt in~\cite{Schmitt}, although, unlike those Hopf algebras,
it has signs in the definition of its coproduct. Let us make this
relation more precise. For every $n\geq 0$, we denote by
$K_f(\Rrr^n)$ the free abelian group on the set of closed convex
polyhedral cones in $\Rrr^n$; if $K\in\Cf(\Rrr^n)$, we denote
its class in $K_f(\Rrr^n)$ by $[K]_f$. The formulas for $\Delta$ and
$\epsilon$ also define a coalgebra structure on $K_f(\Rrr^n)$ and, if
we set $\Kf_f=\bigoplus_{n\geq 0}K_f(\Rrr^n)$, then the product on
$\Kf_f$ defined by $[K]_f\cdot[L]_f=[K\times L]_f$ makes $\Kf_f$ into a
coalgebra. Let $\Pf$ be the set of isomorphism classes of finite
posets, and let $\Z[\Pf]$ be the free abelian group on $\Pf$
equipped with the Hopf algebra structure defined in Sections~3 and~4
of~\cite{Schmitt}. Then we have bialgebra morphisms
$\pi_1:\Kf_f\fl \Kf$ and $\pi_2:\Kf_f\fl\Z[\Pf]$ defined as follows:
if $K\in\Cf(\Rrr^n)$ then $\pi_1$ sends $[K]_f$ to $[K]$ and
$\pi_2$ sends $[K]_f$ to $(-1)^d$ times the class of the poset of
faces of $K$, where $d$ is the dimension of the largest vector subspace
contained in $K$.

\end{subrmk}

\subsection{Proofs}
\label{subsection_convolution_proofs}

Before proving Theorem~\ref{thm_coalgebra},
we state and prove the following lemma.
\begin{sublemma}
Let $K \subseteq V$ be a closed convex polyhedral cone, let
$F$ be a closed face of the cone~$K$ and let $\mu\in V^\vee$. 
We write $K_{0}=K\cap H_\mu$, 
$K_{+}=K\cap \overline{H^+_\mu}$ and
$K_{-}=K\cap \overline{H^-_\mu}$.
\begin{itemize}
\item[(a)]
Assume that $F\subseteq \overline{H_\mu^{+}}$ but $F\not\subseteq H_\mu$,
that is, $F$ is a face of $K_{+} $ but not of $K_{0}$.
Then the equality $\spam{F}+K=\spam{F}+K_+$ holds.

\item[(b)]
Assume that $F \cap H_\mu^+ \neq \varnothing$
and
$F \cap H_\mu^{-} \neq \varnothing$,
in other words, the hyperplane $H_\mu$ cuts the face $F$ in two.
Then the equality
$\spam{F}+K=\spam{F}+K_0$ holds.

\item[(c)] In the situation of (b), let $F_{0}=F\cap H_\mu$.
Then the equality
$\spam{F_0}+K=\spam{F}+K$ holds.

\item[(d)] Let $X$ be a subset of $V$.
Then $X+K=(X+K_+)\cup(X+K_-)$. If moreover
$X\subseteq H_\mu$, we also have $X+K_0=(X+K_+)\cap(X+K_-)$.
\end{itemize}
\label{lemma_cut_face}
\end{sublemma}

\begin{proof} 
We first prove (a).
The inclusion
$\spam{F}+K_+\subseteq\spam{F}+K$
clearly holds,
so we just need to show the reverse inclusion. 
Let $x\in\spam{F}+K$, and write $x=y+z$, with
$y\in\spam{F}$ and $z\in K$. As $F\not\subseteq H_\mu$,
there exists $y'\in F$ such that $\la\mu,y'\ra>0$.
Then for $a\in\R_{>0}$ large enough we have
$\la\mu,ay'+z\ra\geq 0$, hence $ay'+z\in K_+$. As
$x=(y-ay')+(ay'+z)$, this shows that $x\in\spam{F}+K_+$.

We now prove (b).
The inclusion $\spam{F}+K_0\subseteq\spam{F}+K$
clearly holds, 
so we just need to verify the reverse inclusion. 
Let $x\in\spam{F}+K$, and write $x=y+z$ with
$y\in\spam{F}$ and $z\in K$. 
By the assumption on $F$, the image of $F$ by $\mu$ is not contained
in $\R_{\geq 0}$ or in $\R_{\leq 0}$; as this image is a cone in~$\R$, we
conclude that it is equal to $\R$. In particular, we can find
$y'\in F$ such that
$\la\mu,y'\ra=-\la\mu,z\ra$.
Then $x=(y-y')+(y'+z)$ with $y-y'\in\spam{F}$, $y'+z\in K$ and
$\la\mu,y'+z\ra=0$, hence $x\in\spam{F}+K_0$. 

We prove (c).
The inclusion $\spam{F_0}+K\subseteq\spam{F}+K$ is clear,
so we need to show the reverse inclusion. 
By the proof of (b), the image of $F$ by $\mu$ is equal to
$\R$, so we can find $y'\in F$ such that $\la\mu,y'\ra=
\la\mu,y\ra$. 
Then $x=(y-y')+(y'+z)$ with $y-y'\in\spam{F}$, $y'+z\in K$ and
$\la\mu,y-y'\ra=0$, hence $x\in\spam{F_0}+K$. 

Finally, we prove (d).
The inclusion $(X+K_+)\cup(X+K_-)\subseteq X+K$ is clear. Conversely,
let $x\in X+K$, and write $x=y+z$ with $y\in X$ and $z\in K$. Then
either $z\in K_+$, in which case $x\in X+K_+$, or
$z\in K_-$, in which case $x\in X+K_-$. The inclusion
$X+K_0\subseteq (X+K_+)\cap(X+K_-)$ is also clear and holds without
any condition on $X$. Assume that
$X\subseteq H_\mu$, and let $x\in(X+K_+)\cap (X+K_-)$.
Write $x=y_1+z_1=y_2+z_2$, with $y_1,y_2\in X$, $z_1\in K_+$ and
$z_2\in K_-$. Then $\la\mu,y_1\ra=\la\mu,y_2\ra=0$, so
\[\la\mu,z_1\ra=\la\mu,x\ra-\la\mu,y_1\ra=
\la\mu,x\ra-\la\mu,y_2\ra=\la\mu,z_2\ra.\]
As $\la\mu,z_1\ra\geq 0$ and $\la\mu,z_2\ra\leq 0$, this
implies that $\la\mu,z_1\ra=\la\mu,z_2\ra=0$, hence
that $z_1,z_2\in K_0$, and so $x\in X+K_0$.
\end{proof}

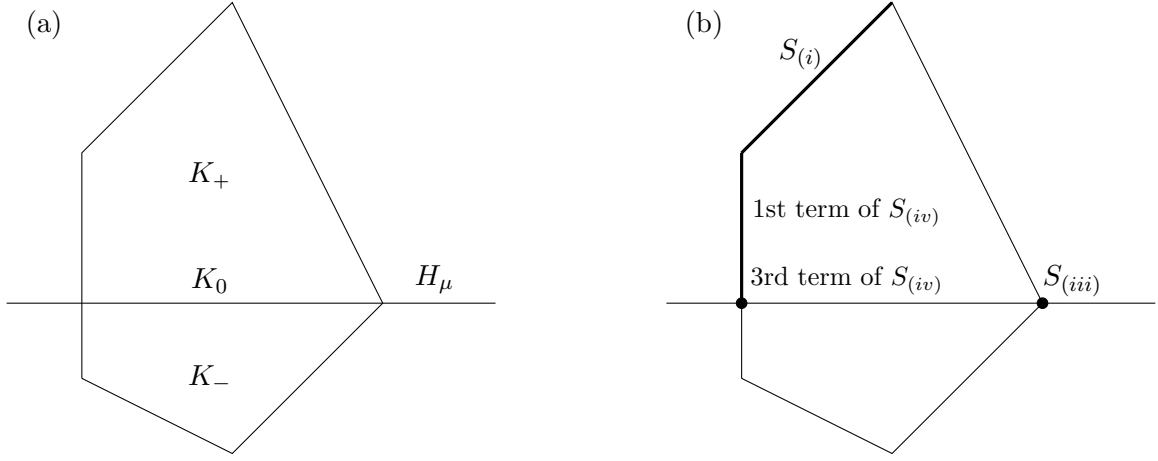
\begin{figure}[t]
\begin{tikzpicture}
\node at (-0.5,3.7) {(a)};
\draw[-] (-1,0) -- node[very near end,above] {$H_{\mu}$} (5.5,0);
\draw[-] (0,-1) -- (0,2) -- (2,4) -- (4,0) -- (2,-2) -- (0,-1);
\node at (1.7,1.7) {$K_{+}$};
\node at (1.7,0.3) {$K_{0}$};
\node at (1.7,-1) {$K_{-}$};
\end{tikzpicture}
\hspace*{20mm}
\begin{tikzpicture}
\node at (-0.5,3.7) {(b)};
\draw[-] (-1,0) -- (5.5,0);
\draw[-] (0,-1) -- (0,2) -- (2,4) -- (4,0) -- (2,-2) -- (0,-1);

\node at (0.79,3.31) {$S_{(i)}$};
\draw[-,very thick] (0,2) -- (2,4);

\node at (4.4,0.3) {$S_{(iii)}$};
\filldraw (4,0) circle (2pt);

\node at (1.4,1.2) {\small $1$st term of $S_{(iv)}$};
\draw[-,very thick] (0,0) -- (0,2);

\node at (1.4,0.3) {\small $3$rd term of $S_{(iv)}$};
\filldraw (0,0) circle (2pt);
\end{tikzpicture}
\caption{A two-dimensional representation of a five-sided three-dimensional cone $K$.
In (b) the four different contributions to $g(K_{+})$ are marked.}
\label{figure_one}
\end{figure}

\begin{proof}[Proof of Theorem~\ref{thm_coalgebra}]
We first show that $\Delta:\Cf(V)\fl K(V)\otimes_\Z K(V)$ and
$\varepsilon$ are valuations.
We check the criterion of Theorem~\ref{theorem_Groemer}. Let $K\in\Cf(V)$
and let $\mu \in V^\vee$.
We define as before three closed convex polyhedral
cones
$K_{+} = K \cap \overline{H_\mu^{+}}$,
$K_{-} = K \cap \overline{H_\mu^{-}}$,
$K_{0} = K \cap H_{\mu}$.
We show that
\begin{equation}
\label{equation_epsilon_valuation}
\varepsilon(K)+\varepsilon(K_0)=\varepsilon(K_+)+
\varepsilon(K_-).
\end{equation}
If $K_+=K_0$ then $K=K_-$ so equation~\eqref{equation_epsilon_valuation} is
clear. The case where $K_-=K_0$ is similar. Suppose that $K_+ \neq K_0$
and $K_- \neq K_0$. Then the image of $K_+$ by $\mu$ is $\R_{\geq 0}$, so
$K_+$ cannot be a vector subspace of $V$, and similarly $K_-$ cannot be a vector
subspace of $V$. This implies that $\varepsilon(K_+)=\varepsilon(K_-) = 0$,
so equation~\eqref{equation_epsilon_valuation} holds if and only
if $\varepsilon(K)=-\varepsilon(K_0)$. As $K_0$ is strictly included in~$K$,
we have $\dim(K_0)=\dim(K)-1$, so we need to prove that $K$
is a vector subspace if and only if $K_0$ is.
If~$K$ is a vector subspace of $V$ then so is~$K_0$.
Suppose that $K_0$ is a vector subspace of $V$. We want to prove that
$K$ also is a vector subspace of $V$.
Without loss of generality we may assume that $\spam{K}=V$.
As $K_+ \neq K_0$
and $K_- \neq K_0$, the hyperplane $H_\mu$ meets the relative interior
of $K$, and so $\spam{K_0}=H_\mu$, hence $H_\mu=K_0\subseteq K$.
As $K$ contains
points in both open half-spaces cut out by $H_\mu$, this implies that
$K=V$.

We now treat the case of $\Delta$.
The faces $F$ of the cone $K$ come in four disjoint categories.
For each category, we consider the contribution to the sum defining~$\Delta(K)$.
\begin{itemize}
\item[(i)]
$F$ is a face of $K_{+}$, but not of $K_{0}$,
that is, $\interior{F} \subseteq H_\mu^{+}$.
Then by Lemma~\ref{lemma_cut_face}(a)
we have $\spam{F}+K=\spam{F}+K_+$.
Hence the contribution is
\begin{align*}
S_{(i)}
& =
\sum_{\substack{F \in \mathcal{F}(K) \cap \mathcal{F}(K_{+}) \\ F\not\in
\mathcal{F}(K_{0})}}
(-1)^{\dim(F)}
\cdot
[F]
\otimes
[\spam{F}+K]\\
& =
\sum_{F \in (\mathcal{F}(K_{+})\cap\mathcal{F}(K))-\mathcal{F}(K_{0})}
(-1)^{\dim(F)}
\cdot
[F]
\otimes
[\spam{F}+K_+]
\end{align*}

\item[(ii)]
$F$ is a face of $K_{-}$, but not of $K_{0}$,
that is, $\interior{F} \subseteq H_\mu^{-}$.
As in case (i), we have $\spam{F}+K=\spam{F}+K_-$,
and the contribution is
\begin{align*}
S_{(ii)}
& =
\sum_{\substack{F \in \mathcal{F}(K) \cap \mathcal{F}(K_{-}) \\ F\not\in
\mathcal{F}(K_{0})}}
(-1)^{\dim(F)}
\cdot
[F]
\otimes
[\spam{F}+K]\\
& =
\sum_{F \in (\mathcal{F}(K_{-})\cap\mathcal{F}(K))-\mathcal{F}(K_{0})}
(-1)^{\dim(F)}
\cdot
[F]
\otimes
[\spam{F}+K_-]
\end{align*}

\item[(iii)]
$F$ is a face of all three cones $K_{+}$, $K_{-}$ and $K_{0}$,
that is, we have $F \subseteq H_{\mu}$.
Here the contribution is
\begin{align*}
S_{(iii)}
& =
\sum_{\substack{F \in \mathcal{F}(K) \\ F \subseteq H_{\mu}}}
(-1)^{\dim(F)}
\cdot
[F]
\otimes
[\spam{F}+K]\\
& = 
\sum_{F \in \mathcal{F}(K_{+})\cap\mathcal{F}(K_{-})\cap\mathcal{F}(K_{0})}
(-1)^{\dim(F)}
\cdot
[F]
\otimes
\bigl([\spam{F}+K_+]
+
[\spam{F}+K_-]
-
[\spam{F}+K_0])\bigr) ,
\end{align*}
since
$\spam{F}+K=(\spam{F}+K_+)\cup(\spam{F}+K_-)$ and
$\spam{F}+K_0=(\spam{F}+K_+)\cap(\spam{F}+K_-)$
by Lemma~\ref{lemma_cut_face}(d).

\item[(iv)]
The face $F$ gets cut into three faces:
$F_{+} = F\cap K_{+}$ in $K_{+}$,
$F_{-} = F\cap K_{-}$ in $K_{-}$ and
$F_{0} = F\cap K_{0}$ in $K_{0}$.
Then we have $\spam{F}=\spam{F_+}=\spam{F_-}$.
By Lemma~\ref{lemma_cut_face}(b), we have 
$\spam{F}+K=\spam{F}+K_0$, and so
\[\spam{F}+K=\spam{F_+}+K_+=\spam{F_-}+K_-=\spam{F}+K_0.\]
By Lemma~\ref{lemma_cut_face}(c), we also have
$\spam{F}+K=\spam{F_0}+K$, and by
Lemma~\ref{lemma_cut_face}(d) we have
$[\spam{F_0}+K]=[\spam{F_0}+K_+]+
[\spam{F_0}+K_-]-[\spam{F_0}+K_0]$.
So the contribution is
\begin{align*}
S_{(iv)}
& =
\sum_{\substack{F \in \mathcal{F}(K) \\ F \text{ being cut}}}
(-1)^{\dim(F)}
\cdot
[F]
\otimes
[\spam{F}+K]\\
& =
\sum_{\substack{F \in \mathcal{F}(K) \\ F \text{ being cut}}}
(-1)^{\dim(F)}
\cdot
([F_+]+[F_-]-[F_0])
\otimes
[\spam{F}+K]\\
& = 
\sum_{\substack{F \in \mathcal{F}(K) \\ F \text{ being cut}}}
(-1)^{\dim(F)}
\cdot
([F_+]\otimes[\spam{F_+}+K_+]+[F_-]\otimes[\spam{F_-}+K_-]-[F_0]\otimes[\spam{F_0}+K])\\
&=
\sum_{\substack{F \in \mathcal{F}(K) \\ F \text{ being cut}}}
(-1)^{\dim(F)}
\cdot 
\biggl(
[F_+]\otimes[\spam{F_+}+K_+]+[F_-]\otimes[\spam{F_-}+K_-]\\
&
\hspace*{20mm}
-
[F_0]\otimes[\spam{F_0}+K_+]
-[F_0]\otimes[\spam{F_0}+K_-]
+[F_0]\otimes[\spam{F_0}+K_0]\biggr).
\end{align*}
\end{itemize}
Now expand $\Delta(K)$ as $S_{(i)}+S_{(ii)}+S_{(iii)}+S_{(iv)}$.
We use use the fact that $(-1)^{\dim(F)} = - (-1)^{\dim(F_{0})}$
in the third, fourth and fifth terms of $S_{(iv)}$.
The contributions to $\Delta(K_{+})$, respectively~$\Delta(K_{-})$,
are given by the sum $S_{(i)}$, respectively~$S_{(ii)}$,
the first term in the sum $S_{(iii)}$, respectively the second term,
and
the first and third terms in the sum $S_{(iv)}$, respectively the second and fourth terms.
See Figure~\ref{figure_one}~(b).
Finally,
the third term of the sum $S_{(iii)}$
and
the fifth term of the sum $S_{(iv)}$
yield the sum for $-\Delta(K_{0})$,
which proves
that 
$\Delta(K)=\Delta(K_{+})+\Delta(K_{-})-\Delta(K_{0})$.

We now prove that $\Delta$ is coassociative. Let $K\in\Cf(V)$. Then
\begin{align*}
(\Delta\otimes\id_{K(V)})(\Delta(K)) & =
(\Delta\otimes\id_{K(V)})\left(
\sum_{F\in\F(K)}(-1)^{\dim F}[F]\otimes[\spam{F}+K]
\right)\\
& = \sum_{F\in\F(K)}\sum_{G\in\F(F)}
(-1)^{\dim F+\dim G}[G]\otimes[\spam{G}+F]\otimes[\spam{F}+K].
\end{align*}
We want to compare this expression with $(\id_{K(V)} \otimes \Delta)(\Delta([K]))$.
To calculate this last expression, we need a description of
the faces of the cone $\spam{G}+K$, where $G$ is a face of $K$. 
Let $\Hf$ be the collection
of hyperplanes containing a facet of $K$. Then $\Hf$ is a finite
central hyperplane arrangement on $V$ and, as in
Subsection~\ref{background}, we write $\Lf=\Lf(\Hf)$ and $\Tc=
\Tc(\Hf)$. Let $C$ and $T$ be the relative interiors of $G$ and $K$
respectively. We have $C\in\Lf$ and $T\in\Tc\cap\Lf_{\geq C}$, and there
is a bijection $\{D\in\Lf_{\geq C} : D\leq T\}\iso\{F\in\F(K) : G\subseteq F\}$
sending $D$ to $\overline{D}$. Let $\Hf(C)$ be the subarrangement
of $\Hf$ whose hyperplanes are the ones containing $C$ (or equivalently
$G$). By Lemma~\ref{lemma_star}, the cone $\spam{G}+K$ is the closure
of the unique chamber of $\Hf(C)$ containing $T$, and there is a bijection
from the set $\{D\in\Lf_{\geq C} : D\leq T\}$ to the set of faces
of $\spam{G}+K$ sending $D$ to $\spam{G}+\overline{D}$. We deduce that
there is a bijection from the set  $\{F\in\F(K) : G\subseteq F\}$
to $\F(\spam{G}+K)$ sending $F$ to $\spam{G}+F$. Moreover,
Lemma~\ref{lemma_star}(iv) states that this bijection preserves dimensions.
Thus we obtain
\begin{align*}
(\id_{K(V)}\otimes\Delta)(\Delta(K)) & =
(\id_{K(V)}\otimes\Delta)\left(
\sum_{G\in\F(K)}(-1)^{\dim G}[G]\otimes[\spam{G}+K]
\right)\\
& = \sum_{G\in\F(K)}\sum_{F'\in\F(\spam{G}+K)}
(-1)^{\dim F'+\dim G}[G]\otimes[F']\otimes[\spam{F'}+K]\\
& = \sum_{G\in\F(K)}\sum_{F\in\F(K) : G\subseteq F}
(-1)^{\dim F+\dim G}[G]\otimes[\spam{G}+F]\otimes[\spam{F}+K]\\
& = (\Delta\otimes\id_{K(V)})(\Delta(K)).
\end{align*}
This completes the proof of the coassociativity of $\Delta$.

We finally prove that $\varepsilon$ is a counit of $\Delta$.
Let $K\in\Cf(V)$. Suppose first that $K$ is not a vector subspace of $V$.
Then the only face of $K$ that is a vector subspace is $\{0\}$,
and the only face $F$ such that $\spam{F}+K$ is a vector subspace is
$K$. Hence
\begin{align*}
(\id_{K(V)}\otimes\varepsilon)(\Delta(K)) & 
=\sum_{F\in\F(K)}(-1)^{\dim F}[F]\otimes\varepsilon([\spam{F}+K])\\
& = (-1)^{\dim K}[K]\otimes\varepsilon([\spam{K}])\\
& = (-1)^{\dim K}(-1)^{\dim\spam{K}}[K]\otimes 1=[K]
\end{align*}
and
\begin{align*}
(\varepsilon\otimes\id_{K(V)})(\Delta(K)) & =
\sum_{F\in\F(K)}(-1)^{\dim F}\varepsilon([F])\otimes[\spam{F}+K]\\
& = (-1)^0\varepsilon([\{0\}]\otimes[K]=[K].
\end{align*}
If $K$ is a vector subspace of $V$ then the only face of $K$ is $K$ itself,
so $\Delta(K)=(-1)^{\dim K}[K]\otimes[K]$, and we clearly have
\[(\id_{K(V)}\otimes\varepsilon)(\Delta(K))=
(\id_{K(V)}\otimes\varepsilon)(\Delta(K))=[K].
\qedhere \]
\end{proof}

\section{Review of 2-structures}
\label{appendix_2_structures}

The concept of $2$-structure for a root system
was introduced by Herb to calculate discrete series characters
on real reductive groups.
See for example Section~5 of~\cite{Herb-DSC} or
Section~4 of the review article~\cite{Herb-2S}.
In this section we review Herb's constructions and adapt them
so that they work for an arbitrary Coxeter system having finite Coxeter group.
We also adapt some of heer results to this setting
and give detailed elementary proofs of these results.
Although this is not strictly necessary, we think that it might
be valiable, as the proofs of these
results in the literature can be very hard to follow for people
not already immersed in the representation theory of real groups.

We fix a finite-dimensional $\R$-vector space $V$ and an inner
product $(\cdot,\cdot)$ on $V$.
For every $v\in V-\{0\}$, we denote
by $s_v$ the (orthogonal) reflection across the hyperplane $v^\perp$.

Whenever we need to describe the irreducible
root systems,
we use the description given in the tables at the end
of~\cite{Bourbaki}, except that we write $(e_1,\ldots,e_n)$
for the canonical basis of~$\R^n$. When we need a system of positive
roots in these root systems, we also use the ones given in these tables.

This appendix is organized as follows.
Subsections~\ref{subsection_B_1}
and~\ref{subsection_2-structures}
contain the definitions and results respectively.
Subsections~\ref{subsection_B_3}
and~\ref{section_2_structures_irreducibles}
contain the technical proofs.
The verification that the Coxeter group $W$ acts transitively
on the set of $2$-structures takes place in the fourth subsection.

\subsection{Pseudo-root systems}
\label{subsection_B_1}

\begin{subdef}
A finite subset
$\Phi$ of $V-\{0\}$ is called a \emph{pseudo-root system} if it satisfies the
following conditions:
\begin{itemize}
\item[(a)] for every $\alpha\in\Phi$, we have $\Phi\cap\R\alpha=\{\pm\alpha\}$;
\item[(b)] for every $\alpha,\beta\in\Phi$, the reflection $s_\alpha$ sends
$\beta$ to a vector of the form $c\gamma$, with $c\in\R_{>0}$ and
$\gamma\in\Phi$.
\end{itemize}
If all the elements of $\Phi$ are unit vectors, we call $\Phi$ a
\emph{normalized pseudo-root system}. In that case, condition (b)
become ``$s_\alpha(\beta)\in\Phi$''.
\label{def_pseudo_root_system}
\end{subdef}

\begin{subrmk}
We use this definition because it is convenient in the context of Coxeter
systems. A root system (in the usual sense) is a pseudo-root system,
which is not
normalized in general. The converse is not true, even if we allow ourselves
to replace the elements of $\Phi$ by scalar multiples, because of the
existence of non-crystallographic Coxeter systems
(see Proposition~\ref{prop_pseudo_root_system_vs_Coxeter_system}).

Pseudo-root systems are called ``root systems'' in~\cite[Section~1.2]{Hu-Cox}
and~\cite[Section~4.4]{BB}. 
We avoid this terminology because it is not compatible with
the established definition of root systems in representation theory.

\end{subrmk}

\begin{subrmk}
If $\Phi$ is normalized or an actual root system then the group
$W$ preserves $\Phi$, so the action of $W$ on $V$ restricts
to an action of $W$ on $\Phi$.
In general, we can still make $W$ act on $\Phi$ by declaring that if
$w\in W$ and $\alpha\in\Phi$ then $w\cdot\alpha$ is the unique
element $\beta$ of $\Phi$ such that $w(\alpha)\in\R_{>0}\beta$.
This reduces to the previous action if $\Phi$ is normalized or an
actual root system. Whenever we write an element
of $W$ acting on an element of $\Phi$, this is the action that
we mean.
\label{rmk_action_W_Phi}
\end{subrmk}

\begin{subdef}
Let $\Phi\subset V$ be a pseudo-root system. 
A subset $\Delta$ of $\Phi$ is called a \emph{system of simple pseudo-roots}
if
\begin{itemize}
\item[(a)]
The set $\Delta$ is a vector space basis for the linear span of $\Phi$.
\item[(b)]
For every $\alpha\in\Phi$, we can write $\alpha=\sum_{\beta\in\Delta}
n_\beta\beta$, where the coefficients~$n_\beta$ are in $\Rrr$ and they are
either all nonnegative or all nonpositive.
\end{itemize}
The corresponding \emph{system of positive pseudo-roots} is then
$$\Phi^+
=
\Phi
\cap
\biggl\{\sum_{\beta\in\Delta} n_\beta \beta\ 
: n_\beta\in\R_{\geq 0}\ \forall\beta\in\Delta\biggr\}. $$
We also write $\Phi^- = - \Phi^+$.
\label{def_positive_pseudo_roots}
\end{subdef}

\begin{subdef}
Let $\Phi\subset V$ be a pseudo-root system. We say that $\Phi$ is
\emph{irreducible} if there is no partition $\Phi=\Phi_1\sqcup\Phi_2$, with
$\Phi_1$ and $\Phi_2$ nonempty pseudo-root systems such that
$(\alpha_1,\alpha_2)=0$ for every $\alpha_1\in\Phi_1$ and 
every $\alpha_2\in\Phi_2$.
\end{subdef}

\begin{subprop}
The following two statements hold:
\begin{itemize}
\item[(i)] (\cite[Section~1.9]{Hu-Cox} and~\cite[Section~1.4]{Hu-Cox}.)
Let $\Phi\subset V$ be a pseudo-root system and $\Delta\subset\Phi$ be
a system of simple pseudo-roots. Let $W=W(\Phi)$ 
be the subgroup of $\GL(V)$ generated
by the reflections $s_\alpha$ for $\alpha\in\Phi$, and let
$S=\{s_\alpha : \alpha\in\Delta\}$.
Then $(W,S)$ is a Coxeter system where $W$ is finite, 
and the Coxeter graph of $(W,S)$ is connected if and only if
$\Phi$ is irreducible.

Moreover, $W$ acts transitively on the
set of systems of positive pseudo-roots if we use the action of
Remark~\ref{rmk_action_W_Phi}.

\item[(ii)] (\cite[Section~5.4]{Hu-Cox}.)
Conversely, let $(W,S)$ be a Coxeter system with $W$ finite, 
and let
$\rho:W \longrightarrow \GL(V)$ be its canonical representation on $V=\bigoplus_{s\in S}
\R e_s$
(see the beginning of Subsection~\ref{section_Coxeter_arrangements}). 
Then $\Phi=\{\rho(w)(e_s) : w\in W,\ s\in S\}$ is a 
normalized 
pseudo-root
system
and $\Delta=\{e_s : s\in S\}$ is a system of simple pseudo-roots in $\Phi$.
\end{itemize}
\label{prop_pseudo_root_system_vs_Coxeter_system}
\end{subprop}

\begin{subdef}
Let $\Phi\subset V$ be an irreducible pseudo-root system. We say that
$\Phi$ is of type~$A_n$, respectively $B_n$, $D_n$, $E_6$, $E_7$, $E_8$, $F_4$,
$H_3$, $H_4$, $I_2(m)$ with $m \geq 3$, if the corresponding Coxeter system
is of that type. Here we use the classification of simple finite
Coxeter systems proved 
in~\cite[Chapter~5]{GB}.
See Table~1 in~\cite[Appendix~A]{BB}.
\end{subdef}

\begin{subrmk}
The Coxeter group of type~$I_2(m)$ is the dihedral group of order $2m$.
Note that types~$I_2(3)$ and~$A_2$ are isomorphic, types
$I_2(4)$ and $B_2$ are isomorphic, and types $I_2(6)$ and $G_2$ are
isomorphic.
We did not include $I_2(2)$ in the list of irreducible types,
because the corresponding Coxeter system is not irreducible,
as it is isomorphic to $A_1\times A_1$.
\end{subrmk}

We will use the following lemma when introducing the sign associated to
a $2$-structure in Proposition~\ref{prop_sign_2_structure}.
Recall that, if $r\geq 1$, then the \emph{lexicographic order}
on $\R^r$ is defined by $(x_1,\ldots,x_r) < (y_1,\ldots,y_r)$
if there exists $1 \leq i \leq r$ such that $x_i < y_i$ and that
$x_j=y_j$ for $1\leq j\leq i-1$.
It is a total order.
Furthermore we say that a vector $x$ is {\em positive} if 
$x > (0,0, \ldots, 0)$.

\begin{sublemma}
Let $\Phi\subset V$ be a pseudo-root system.
Let $v_1, v_2, \ldots, v_r$ be 
linearly independent 
elements of $V$ such that no element of $\Phi$ is
orthogonal to every $v_i$. Define $\Phi^+$ to be the set
of $\alpha\in\Phi$ such that the element $((\alpha,v_1),(\alpha,v_2),\ldots,
(\alpha,v_r))$ of $\R^r$ is positive with respect to the lexicographic
order on~$\R^r$. 
Then $\Phi^+$ is a system of positive pseudo-roots. 

\label{lemma_defining_positive_roots}
\end{sublemma}
\begin{proof}
We complete $(v_1,\ldots,v_r)$ to a basis $(v_1,\ldots,v_n)$ of $V$,
where $n$ is the dimension of~$V$. If $v,w\in V$, we say that $v<w$ if 
$((v,v_1),\ldots,(v,v_n))<((w,v_1),\ldots,(w,v_n))$ in the lexicographic
order on $\R^n$. This defines a total order on $V$ in the sense
of~\cite[Section~1.3]{Hu-Cox}, and $\Phi^+$ is the corresponding
positive system in $\Phi$. By the theorem in~\cite[Section~1.3]{Hu-Cox},
$\Phi^+$ is a system of positive pseudo-roots in the sense of
Definition~\ref{def_positive_pseudo_roots}.
\end{proof}

\begin{subdef}
If $\theta=(v_1,\ldots,v_r)$ is a sequence of 
linearly independent elements of
$V$ such that $\theta^\perp\cap\Phi=\varnothing$, we denote the
system of positive pseudo-roots of Lemma~\ref{lemma_defining_positive_roots} by
$\Phi^+_\theta$.
\label{def_Phi_theta}
\end{subdef}

\subsection{2-structures}
\label{subsection_2-structures}

We define $2$-structures, generalizing a notion introduced by
Herb for root systems; see for example the beginning of
\cite[Section~4]{Herb-2S}. We also generalize some of the results
of~\cite[Section~5]{Herb-DSC} to Coxeter systems with finite
Coxeter groups.

We fix a pseudo-root system $\Phi$ in $V$ and a system of
positive pseudo-roots $\Phi^+\subset\Phi$. We denote by
$(W,S)$ the corresponding Coxeter system
(see Proposition~\ref{prop_pseudo_root_system_vs_Coxeter_system}).

\begin{subdef}
A \emph{$2$-structure} for $\Phi$ is a subset $\varphi$ of $\Phi$,
that is, a pseudo-root system in $V$
satisfying the following properties:
\begin{itemize}
\item[(a)] 
The subset $\varphi$ is a disjoint union
$\varphi=\varphi_1\sqcup\varphi_2\sqcup\cdots\sqcup\varphi_r$,
where the $\varphi_i$ are pairwise orthogonal subsets
of $\varphi$ and each of them is an irreducible pseudo-root system in $V$ of type
$A_1$, $B_2$ or $I_2(2^n)$, for $n\geq 3$.
\item[(b)] Let $\varphi^+=\varphi\cap\Phi^+$. If $w\in W$ is such that
$w(\varphi^+)=\varphi^+$ then $\det(w) = 1$.
\end{itemize}
\label{def_2_structure}
\end{subdef}

\begin{subrmk}
Although condition (b) involves the set of positive pseudo-roots
$\varphi^+$ in $\varphi$, 
it does not actually depend on the choice
of $\varphi^+$, because 
the Coxeter
group of $\varphi$ acts transitively on sets of positive pseudo-roots
in $\varphi$.
\end{subrmk}

\begin{subrmk}
If $\varphi \subseteq \Phi$ is a $2$-structure then there is no
$\alpha\in\Phi$ that is orthogonal to every element of $\varphi$.
Indeed, if such an $\alpha$ existed then the associated reflection
$s_\alpha$ would fix every element of $\varphi$, and in particular
send $\varphi^+$ to itself, which would contradict condition (b)
of Definition~\ref{def_2_structure}.
\label{rmk_no_orthogonal_root}
\end{subrmk}

Let $\Tt(\Phi) \subseteq 2^{\Phi}$ be the set of all 
$2$-structures for the pseudo-root system $\Phi$.
The following proposition is proved in
Subsection~\ref{section_2_structures_irreducibles}, where we
also show that each irreducible pseudo-root system
contains a $2$-structure and give
the type of this $2$-structure.
This introduces no
circularity in the arguments: the only results in this appendix
that depend on Proposition~\ref{prop_W_acts_transitively_on_T_Phi} are
Lemmas~\ref{lemma_T_Phi_induction_step}
and~\ref{lemma_r_Phi}, and these lemmas are not used in
Subsections~\ref{subsection_B_3}
and~\ref{section_2_structures_irreducibles}.

\begin{subprop}
The group $W$ acts transitively on the collection of $2$-structures $\Tt(\Phi)$.
\label{prop_W_acts_transitively_on_T_Phi}
\end{subprop}

Let $\varphi\in\Tt(\Phi)$.
We write $\varphi^+=\varphi\cap\Phi^+$
and $\varphi^-=\varphi\cap\Phi^-$,
and we define
\begin{align*}
W(\varphi,\Phi^+) & = \{w\in W:w(\varphi^+)\subset\Phi^+\} , \\
W_1(\varphi,\Phi^+) & = \{w\in W:w(\varphi^+)\subset\varphi^+\} 
= \{w\in W:w(\varphi^+) = \varphi^+\} .
\end{align*}
Note that $W_1(\varphi,\Phi^+)$ is a subgroup of $W$, and that
the subset $W(\varphi,\Phi^+)$ of $W$ is stable by right translations
by elements of $W_1(\varphi,\Phi^+)$.

\begin{subcor}
Let $\varphi\in\Tt(\Phi)$. Then the
map $W \longrightarrow \Tt(\Phi)$,
$w \longmapsto w(\varphi)$ induces a bijection
\[W(\varphi,\Phi^+)/W_1(\varphi,\Phi^+) \iso \Tt(\Phi).\]
\label{cor_T_Phi_as_quotient}
\end{subcor}

\begin{proof}
We denote by $f:W\fl\Tt(\Phi)$ the map defined by
$f(w)=w(\varphi)$. 

If $u\in W_1(\varphi,\Phi^+)$, then $u(\varphi)=\varphi$, so
$f(wu)=f(w)$ for every $w\in W$. So the map $f$ does induce a
map from $W(\varphi,\Phi^+)/W_1(\varphi,\Phi^+)$ to $\Tt(\Phi)$,
that we denote by $\overline{f}$.

We show that $\overline{f}$ is surjective. Let $\varphi'\in
\Tt(\Phi)$. By Proposition~\ref{prop_W_acts_transitively_on_T_Phi},
there exists $w\in W$ such that $w(\varphi)=\varphi'$. By
the theorem in~\cite[Section~1.3]{Hu-Cox}, the set
$w^{-1}(\Phi^+)\cap\varphi$ is a system of positive pseudo-roots
in $\varphi$, so, by 
Proposition~\ref{prop_pseudo_root_system_vs_Coxeter_system}, there
exists $v\in W(\varphi)$, where
$W(\varphi)$ is the Coxeter group of $\varphi$,
such that $v(\varphi^+)=w^{-1}(\Phi^+)\cap\varphi$.
Then $wv(\varphi^+)=\Phi^+\cap w(\varphi)\subset\Phi^+$, so
$wv\in W(\varphi,\Phi^+)$, and $wv(\varphi)=w(\varphi)=\varphi'$,
that is, $f(wv)=\varphi'$.

We show that $\overline{f}$ is injective. Let $w,w'\in W(\varphi,\Phi^+)$
such that $w(\varphi)=w'(\varphi)$. Then we have $w^{-1}w'(\varphi)=\varphi$,
and, again by the theorem in~\cite[Section~1.3]{Hu-Cox}, the set
$w^{-1}w'(\varphi^+)$ is a system of positive pseudo-roots in
$\varphi$, so there exists $v\in W(\varphi)$ such that
$v^{-1}w^{-1}w'(\varphi^+)=\varphi^+$. This means that we have
$w'=wvu$ with $u\in W_1(\varphi,\Phi^+)$. So we will be done if we
show that $v=1$. Note that $wv(\varphi^+)=wvu(\varphi^+)=
w'(\varphi^+)\subset\Phi^+$. Suppose that $v \neq 1$; then there
exists $\alpha\in\varphi^+$ such that $v(\alpha)\in \varphi^-$,
and then $wv(\alpha)=-w(-v(\alpha))\in \Phi^-$ (because
$w\in W(\varphi,\Phi^+)$), contradicting the fact that $wv(\varphi^+)
\subset\Phi^+$. So $v=1$.
\end{proof}

The following proposition, which 
follows immediately from Lemma~5.6 of \cite{Herb-DSC} for root systems,
can be proved
via a direct calculation for the remaining irreducible types.
We will not need this result, so we do not go into details.

\begin{subprop}
Let $\Tt_{(a)}(\Phi)$ be the set of $\varphi \subseteq \Phi$ that satisfy
condition (a) of Definition~\ref{def_2_structure}. Then $\Tt(\Phi)$
is exactly the set of elements of $\Tt_{(a)}(\Phi)$ that are
maximal with respect to inclusion.
\label{prop_maximality_of_2_structures}
\end{subprop}

\begin{subprop}
Let $\varphi \subseteq \Phi$ be a $2$-structure.
Define an ordered subset $\theta$ of $\varphi$ as
follows.
Select a linear order of the irreducible components
$\varphi_1, \varphi_2, \ldots, \varphi_r$ of $\varphi$.
If $\varphi_{i}$ is a pseudo-root system of type~$A_1$,
let $\theta_{i}$ be the singleton $\varphi_{i} \cap \varphi^{+}$.
If $\varphi_{i}$ is a pseudo-root system of type~$B_{2}$
or~$I_{2}(2^{k})$ for $k \geq 3$, pick two orthogonal elements
$\alpha$ and $\alpha'$ from $\varphi_{i} \cap \varphi^{+}$
such that $\varphi_{i} \cap \varphi^{+}=
\varphi^+_{i,(\alpha,\alpha')}$, that is, such that an element
$\beta$ of $\varphi_i$ is in $\varphi^+$ if and only
if either $(\beta,\alpha)>0$, or $(\beta,\alpha)=0$ and
$(\beta,\alpha')>0$.
Let~$\theta_{i}$ be the sequence $(\alpha, \alpha')$.
Finally let $\theta$ be the concatenation of
the sequences $\theta_{1}, \theta_{2}, \ldots, \theta_{r}$.

Let $\Phi_\theta^+$ be the system of positive pseudo-roots defined by
the sequence  $\theta$ as in Lemma~\ref{lemma_defining_positive_roots},
and let $w_\theta$ be the unique element of $W$ such that
$w_\theta\cdot \Phi^+=\Phi_\theta^+$.
Then the sign $\det(w_\theta)$ depends only on~$\varphi$ and not on the choices
made to form $\theta$.
\label{prop_sign_2_structure}
\end{subprop}

\begin{figure}[t]
\centering
\begin{tikzpicture}
\foreach \tttt in {0, 22.5, ..., 337.5}
{\draw[-] ({3*cos(\tttt)},{3*sin(\tttt)}) --
({-3*cos(\tttt)},{-3*sin(\tttt)});};
\draw[dashed](-2.5,2.5) node{$\Phi^+$}; 
\draw[dashed] ({4*cos(-11.25)},{4*sin(-11.25)}) --
({-4*cos(-11.25)},{-4*sin(-11.25)});
\draw[->, very thick] (0,0) -- (0,2) node[right]{$\alpha$};
\draw[->, very thick] (0,0) -- (2,0) node[above]{$\alpha'$};
\end{tikzpicture}
\hspace*{1 mm}
\begin{tikzpicture}
\foreach \tttt in {0, 22.5, ..., 337.5}
{\draw[-] ({3*cos(\tttt)},{3*sin(\tttt)}) --
({-3*cos(\tttt)},{-3*sin(\tttt)});};
\draw[dashed](-2.5,2.5) node{$\Phi^+$}; 
\draw[dashed] ({4*cos(-11.25)},{4*sin(-11.25)}) --
({-4*cos(-11.25)},{-4*sin(-11.25)});
\draw[->, very thick] (0,0) -- ({2*cos(3*22.5)},{2*sin(3*22.5})
node[left]{$\alpha$};
\draw[->, very thick] (0,0) -- ({2*cos(7*22.5)},{2*sin(7*22.5})
node[above]{$\alpha'$};
\end{tikzpicture}
\caption{The dihedral pseudo-root system $I_{2}(8)$ with the two choices
of $\theta = \bigl(\alpha, \alpha'\bigr)$.}
\label{figure_choice_theta}
\end{figure}
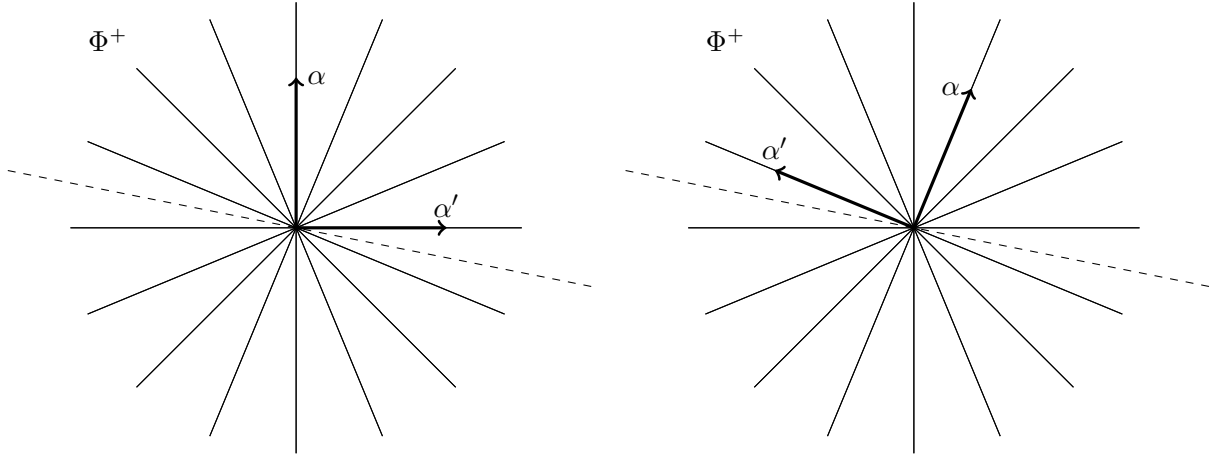

Note that there are several choices when producing the ordered set $\theta$.
First we have to select an order
$\varphi_1, \varphi_2, \ldots, \varphi_r$. There are $r!$ ways to do this.
Second, if $\varphi_i$ is of type~$B_{2}$ or of type~$I_{2}(2^{k})$,
there are two possible choices for the pseudo-roots $\alpha$ and $\alpha'$;
see Figure~\ref{figure_choice_theta}.
These selections do not influence the sign of $w_\theta$, although they do
of course affect the set $\Phi^+_\theta$.

\begin{proof}[Proof of Proposition~\ref{prop_sign_2_structure}.]
Let $\theta$ and $\theta'$ be the results of two possible sequences of choices. 
For an element $w$ in $W$,
recall that its length $\ell(w)$
\footnote{By definition, this is the minimal
number of factors in an expression of $w$ as a product of
reflections corresponding to simple pseudo-roots.}
is also given by the cardinality of
the intersection $w \cdot \Phi^{+} \cap \Phi^{-}$;
see~\cite[Proposition~4.4.4]{BB}.
Note that
$\Phi^+_{\theta} \cap \Phi^-_{\theta'}
=
w_{\theta} \cdot \Phi^+ \cap w_{\theta'} \cdot \Phi^-
=
w_{\theta'} \cdot
(w_{\theta'}^{-1} w_{\theta} \cdot \Phi^+ \cap \Phi^-)$
which has cardinality $\ell(w_{\theta'}^{-1} w_{\theta})$.
Hence to prove that the signs agree, that is, that
$\det(w_\theta)=\det(w_{\theta'})$, it suffices
to show that the set $\Phi^+_{\theta} \cap \Phi^-_{\theta'}$
has an even number of elements.

We can reduce to the following two cases:
\begin{itemize}
\item[(a)] there exists $1 \leq i \leq r$ such that
$\theta$ and $\theta'$ differ only by the choice of the
two pseudo-roots in the factor $\varphi_i$;
\item[(b)] there exists $1 \leq i \leq r-1$ such that
$\theta_i=\theta'_{i+1}$, $\theta_{i+1}=\theta'_i$ and
$\theta_j=\theta'_j$ if $j \neq i,i+1$.
\end{itemize}

We begin by treating case (a). We write $\theta_i=(\alpha, \alpha')$ and
$\theta'_i=(\beta, \beta')$.
Let $\gamma \in \Phi^+_\theta \cap \Phi^-_{\theta'}$.
Then $\gamma$ is orthogonal to $\varphi_1,\ldots,\varphi_{i-1}$,
and it is not orthogonal to $\varphi_i$. Also, as
the sets of positive pseudo-roots in $\varphi_i$ defined by
$\theta_i$ and $\theta'_i$ are equal by assumption, we cannot have
$\gamma\in\varphi_i$. Write $\gamma=c\alpha+c'\alpha'+\lambda$, with
$\lambda\in\varphi_1^\perp\cap\cdots\cap\varphi_i^\perp$.
By the previous sentence, we have $\lambda \neq 0$. The vector
$\iota(\gamma)=-(s_{\alpha}s_{\alpha'})(\gamma)=c\alpha+c'\alpha'-\lambda$
is also in $\Phi$. It is not equal to $\gamma$ because $\lambda \neq 0$,
and it is in $\Phi^+_\theta \cap \Phi^-_{\theta'}$ because
$\gamma$ and $\iota(\gamma)$ have the same inner product with any element of
the set $\{\alpha,\alpha',\beta,\beta'\}$. Note that we clearly
have $\iota(\iota(\gamma))=\gamma$.
We have constructed a fixed-point free
involution~$\iota$ on the set
$\Phi^+_\theta \cap \Phi^-_{\theta'}$, which proves that this set has
even cardinality.

We treat case (b). 
Suppose first that $\varphi_i$ and $\varphi_{i+1}$ are both of type~$A_1$,
so we can write $\theta_i=(\alpha_i)$ and $\theta_{i+1}=(\alpha_{i+1})$.
Let $\Phi'$ be the pseudo-root system $\Phi\cap(\R\alpha_i+\R\alpha_{i+1})$.
If $\Phi'$ is of type~$I_2(m)$ with $m\geq 3$, then $m$ must be even because
$\Phi'$ contains two orthogonal pseudo-roots. 
But then $\Phi'$ contains a multiple
$\beta$ of $\alpha_i-\alpha_{i+1}$, and the reflection $s_\beta$
sends $\varphi^+$ to $\varphi^+$ because it fixes every element
of~$\varphi_j$ for $j=i,i+1$ and exchanges $\alpha_i$ and $\alpha_{i+1}$,
contradicting the definition of a $2$-structure. Hence $\Phi'$ is of type
$A_1\times A_1$, and then the fact that
$|\Phi^+_{\theta} \cap \Phi^-_{\theta'}|$ is even
follows from Lemma~\ref{lemma_change_order_theta_sign}.

Suppose that $\varphi_i$ is of type~$A_1$ and $\varphi_{i+1}$ is of
type~$I_2(2^m)$ with $m\geq 2$. Then we can write
$\theta_i=(\alpha_i)$ and $\theta_{i+1}=(\alpha_{i+1},\alpha'_{i+1})$.
Let $\Phi'$, respectively $\Phi''$, be the pseudo-root
system $\Phi\cap(\R\alpha_i+\R\alpha_{i+1})$,
respectively $\Phi\cap(\R\alpha_i+\R\alpha'_{i+1})$, and let
$\theta''$ be the sequence that we obtain from $\theta$ by switching
$\alpha_i$ and~$\alpha_{i+1}$.
As $\varphi_{i+1}$ is of type
$I_2(2^m)$, it (and hence $\Phi$) contains a pseudo-root $\beta$ proportional
to $\alpha_{i+1}-\alpha'_{i+1}$, and then $s_\beta(\Phi')=\Phi''$, so
$\Phi'$ and $\Phi''$ are of the same type. By 
Lemma~\ref{lemma_change_order_theta_sign}, the cardinalities of the
sets $\Phi^+_\theta \cap \Phi^-_{\theta''}$ and
$\Phi^+_{\theta''} \cap \Phi^-_{\theta'}$ have the same parity, and so
$|\Phi^+_{\theta} \cap \Phi^-_{\theta'}|$ is even.
The case where $\varphi_i$ is of rank $2$ and $\varphi_{i+1}$ 
of rank $1$ follows
from the previous case by switching the roles of $\varphi_i$ and~$\varphi_{i+1}$.

Finally, suppose that both $\varphi_i$ and
$\varphi_{i+1}$ are of rank $2$. Then we can write
$\theta_i=(\alpha_i,\alpha_i')$
and $\theta_{i+1}=(\alpha_{i+1},\alpha'_{i+1})$.
We move from $\theta$ to $\theta'$ by the following sequence of
operations:
\begin{itemize}
\item[(1)]
We switch $\alpha_i'$ and $\alpha_{i+1}$. By
Lemma~\ref{lemma_change_order_theta_sign}
and
Remark~\ref{remark_change_order_theta_sign},
this changes the sign of
$w_\theta$ by $(-1)^{m_1/2 - 1}$, where the pseudo-root system
$\Phi_1=\Phi\cap(\R\alpha_i'+\R\alpha_{i+1})$ is of type~$I_2(m_1)$.

\item[(2)]
We switch $\alpha_i'$ and $\alpha_{i+1}'$. 
By the same lemma and remark,
this changes the sign by $(-1)^{m_2/2 - 1}$,
where the pseudo-root system
$\Phi_2 = \Phi\cap(\R\alpha_i'+\R\alpha'_{i+1})$
is of type~$I_2(m_2)$.

\item[(3)]
We switch $\alpha_{i}$ and $\alpha_{i+1}$.
By the same lemma and remark,
this changes the sign by $(-1)^{m_3/2 - 1}$,
where the pseudo-root system
$\Phi_3=\Phi\cap(\R\alpha_{i}+\R\alpha_{i+1})$
is of type~$I_2(m_3)$.

\item[(4)]
We switch $\alpha_{i}$ and $\alpha_{i+1}'$.
By the same lemma and remark,
this changes the sign by $(-1)^{m_4/2 - 1}$,
where the pseudo-root system
$\Phi_4=\Phi\cap(\R\alpha_{i}+\R\alpha'_{i+1})$
is of type~$I_2(m_4)$.
\end{itemize}
The reflections
$s_i=s_{\alpha_i-\alpha_i'}$ and 
$s_{i+1}=s_{\alpha_{i+1}-\alpha_{i+1}'}$ are both in $W$
because $\varphi_i$ contains a multiple of $\alpha_i-\alpha_i'$ and
$\varphi_{i+1}$ contains a multiple of $\alpha_{i+1}-\alpha_{i+1}'$.
Observe now that
$s_i(\Phi_1)=\Phi_3$,
$s_i(\Phi_2)=\Phi_4$,
$s_{i+1}(\Phi_1)=\Phi_2$
and
$s_{i+1}(\Phi_3)=\Phi_4$.
Thus the four pseudo-root systems
$\Phi_1$, $\Phi_2$, $\Phi_3$ and $\Phi_4$
are isomorphic and hence $m_{1} = m_{2} = m_{3} = m_{4}$.
Hence performing operations~(1) to~(4) changes the sign
by $((-1)^{m_1/2 - 1})^{4} = 1$,
that is, $\det(w_\theta)=\det(w_{\theta'})$.
\end{proof}

\begin{subdef}
Let $\varphi \subseteq \Phi$ be a $2$-structure, and let $w_\theta$ be as
in Proposition~\ref{prop_sign_2_structure}.
Then the sign $(-1)^{r+r'}\det(w_\theta)$, where $r$ is the number of
irreducible factors of $\Phi$ of type~$A_{2n}$ with $n$ odd and
$r'$ is the number of irreducible factors of $\Phi$ of
type~$I_2(2n'+1)$ with $n'\geq 3$ odd,
is called the
\emph{sign} of~$\varphi$ and denoted by $\epsilon(\varphi,\Phi^+)$,
or by $\epsilon(\varphi)$ if the system of positive pseudo-roots
$\Phi^+$ is understood.
\label{def_sign_2_structure}
\end{subdef}

\begin{subrmk}
For a root system,
this coincides with the definition of the sign of $\varphi$ from
Herb's paper \cite{Herb-Pl}, and it differs from the definition
in Section~5 of Herb's paper \cite{Herb-DSC}; see Remark~5.1
of \cite{Herb-DSC} and Corollary~\ref{cor_comparison_signs}.

\end{subrmk}

\begin{sublemma}
Let $\varphi \subseteq \Phi$ be a $2$-structure,
that is, $\varphi\in\Tt(\Phi)$.
\begin{itemize}
\item[(i)]
For every $w\in W$, the identity
$\epsilon(w(\varphi), w(\Phi^+))=\epsilon(\varphi,\Phi^+)$ holds.
\item[(ii)]
Let $w\in W$ be such that $w(\varphi^+) \subseteq \Phi^+$.
Then the identity
$\epsilon(w(\varphi),\Phi^+) = \det(w) \cdot \epsilon(\varphi,\Phi^+)$
holds.
\end{itemize}
\label{lemma_sign_of_w(varphi)}
\end{sublemma}
\begin{proof}
Both identities follow easily from the definition of
$\epsilon(\varphi,\Phi^+)$. Indeed, let $\theta$ be a subset of
$\varphi$ chosen as in Proposition~\ref{prop_sign_2_structure}.
For every $w\in W$, $w(\varphi)$ is a $2$-structure for $\Phi$ and
its subset~$w(\theta)$ satisfies the same conditions
for the system of positive pseudo-roots $w(\Phi^+)$,
and also for the system of positive pseudo-roots $\Phi^+$
if $w(\varphi^+) \subset \Phi^+$.
Also, we have $\Phi^+_{w\cdot \theta}=w\cdot\Phi^+_\theta$.
This immediately yields~(i) and~(ii).
\end{proof}

\begin{sublemma}
Let $\alpha_0\in\Phi$ be a simple pseudo-root, let 
$s_0$ be the simple reflection defined by~$\alpha_0$,
let $\Phi_0 = \alpha_0^\perp\cap\Phi$ and $\Phi_0^+=\Phi_0\cap\Phi^+$.
Let $\Tt''$ be the set of $\varphi\in\Tt(\Phi)$
such that $s_0(\varphi)=\varphi$; we also consider the subsets
$\Tt''_1=\{\varphi\in\Tt'':\varphi\cap
\Phi_0\in\Tt(\Phi_0)\}$ and $\Tt''_2=
\Tt''-\Tt''_1$.
Then the following statements hold: 
\begin{itemize}
\item[(0)] Let $\varphi$ be a $2$-structure for $\Phi$.
Then $s_0(\varphi)=\varphi$, that is, the $2$-structure
$\varphi$ is in
$\Tt''$, if and only if $\alpha_0\in\varphi$.

\item[(1)] The map $\Tt''_1\fl\Tt(\Phi_0)$,
$\varphi\longmapsto\varphi\cap\Phi_0$ is bijective.

\item[(2)] For every $\varphi\in\Tt''_1$, we have
$\epsilon(\varphi,\Phi^+)=\epsilon(\varphi\cap\Phi_0,\Phi_0^+)$.

\item[(3)] There exists an involution $\iota$ of $\Tt''_2$
such that, for every $\varphi\in\Tt''_2$, we have
$\varphi\cap\Phi_0=\iota(\varphi)\cap\Phi_0$ and $\epsilon(\iota(\varphi),
\Phi^+)=-\epsilon(\varphi,\Phi^+)$.

\end{itemize}
\label{lemma_T_Phi_induction_step}
\end{sublemma}

\begin{proof}
We prove (0).
If $\alpha_0\in\varphi$, then $s_0$ is in the Coxeter group of
$\varphi$, so $s_0(\varphi)=\varphi$. Conversely,
we have $s_{\alpha_0}(\Phi^+-\{\alpha_0\})\subset\Phi^+$
by~\cite[Lemma~4.4.3]{BB}, so, if $\varphi\in\Tt''$ and
$\alpha_0\not\in\varphi$, then $s_0(\varphi^+) \subseteq \Phi^+\cap\varphi=
\varphi^+$, contradicting condition~(b) in the definition of a
$2$-structure.
Note also that the subset $\varphi\cap\Phi_0$ 
of $\Phi_0$ always satisfies
condition~(a)
in the definition of a $2$-structure, but it does not always satisfy
condition~(b).

We prove (1). 
We may assume that $\Phi$ is irreducible, and we
will freely use the explicit description of $2$-structures
given in Subsection~\ref{section_2_structures_irreducibles}.
If $2$-structures for $\Phi$ are all of type~$A_1^s$ for some
$s$, which happens in types $A_n$, $D_n$, $E_6$, $E_7$, $E_8$, $H_3$,
$H_4$ and
$I_2(m)$ for $m$ odd,
then $\varphi\cap\Phi_0\in\Tt(\Phi_0)$ for every
$\varphi\in\Tt(\Phi)$, that is, $\Tt_1''=\Tt''$,
and we see in the explicit description of $2$-structures
that the map of statement~(1) is a bijection.
It is easy to check that the same statement holds in
type~$I_2(m)$ for $m$ even.

We now suppose that $\Phi$ is of type~$B_n$ or $F_4$. (Recall that
from the point of view of Coxeter systems types $B_n$ and $C_n$ are
isomorphic.) For convenience, in this case,
we take $\Phi$ to be the actual root system, with possibly
non-normalized roots; this does not affect any of the definitions
that we made before.
To study the map of (1),
we may assume that $\alpha_0 = e_n$ or $\alpha_0 = e_1-e_2$. 
Suppose first that $\alpha_0 = e_1-e_2$.
Then $\Phi_0$ is reducible.
Furthermore,
it is
of type $A_1\times B_{n-2}$ if $\Phi$ is of type~$B_n$, and of type $A_1\times B_2$
if $\Phi$ is of type~$F_4$, where the $A_1$ factor is $\{\pm(e_1+e_2)\}$.
In both cases, 
it is easy to see that $\Tt''_1=\Tt''$ and that
(1) holds.
Suppose that $\alpha_0=e_n$. Then $\Phi_0$ is irreducible.
Furthermore,
it is of type~$B_{n-1}$ if $\Phi$ is of type~$B_n$, and of type~$B_3$ if
$\Phi$ is of type~$F_4$. 
If $\Phi$ is of type~$F_4$ or~$B_n$ with $n$ even then again
it is easy to see that $\Tt''_1=\Tt''$ and that~(1) holds.

Finally, suppose that $\Phi$ is of type~$B_n$ with
$n$ odd and that $\alpha_0=e_n$.
If $\varphi\in\Tt''$ then we have $\varphi\in\Tt''_1$
if and only if $\{\pm e_n\}$ is
an irreducible component of $\varphi$. The map sending
$\varphi_0\in\Tt(\Phi_0)$ to $\varphi_0\sqcup\{\pm e_n\}$
is thus an inverse to the map of~(1), so statement~(1) holds.

We now prove (3).
We have seen in the proof of (1) that $\Tt''_2=\varnothing$ unless
$\Phi$ is of type~$B_n$ with $n$ odd and $\alpha_0$ is the short simple root.
Assume
that we are in this case, which means that $\alpha_0=e_n$.
Let $\varphi\in\Tt''_2$. Then there exists
$2 \leq i \leq n$ such that $\varphi_1=\{\pm e_n,\pm e_i,\pm e_n\pm e_i\}$
is an irreducible component of $\varphi$.
Write $\varphi=\varphi_1\sqcup\varphi_2\sqcup\cdots\sqcup\varphi_r$,
where the $\varphi_k$ are irreducible and
$\varphi_2=\{\pm e_j\}$ is the unique rank $1$ component of $\varphi$.
Set $\iota(\varphi)=\{\pm e_n,\pm e_j,\pm e_n\pm e_j\}\sqcup\{\pm e_i\}\sqcup\varphi_3
\sqcup\cdots\sqcup\varphi_r$.
This map switches the roles of $e_i$ and $e_j$.
Then $\iota(\varphi)$ is also in $\Tt''_2$, it is not
equal to $\varphi$, we have $\iota(\varphi)\cap\Phi_0=\varphi\cap\Phi_0$
and $\iota(\iota(\varphi))=\varphi$. 
To finish the proof of (3),
it suffices to show that $\epsilon(\iota(\varphi),\Phi^+)=
-\epsilon(\varphi,\Phi^+)$ for every
$\varphi\in\Tt''_2$. But this follows immediately from the definition
of~$\iota(\varphi)$ and from Lemma~\ref{lemma_change_order_theta_sign}.

We finally prove (2).
Let $\varphi\in\Tt''_1$.
Choose an ordered subset
$\theta=\{\alpha_1,\ldots,\alpha_r\}$ of $\varphi$ as in 
Proposition~\ref{prop_sign_2_structure}. We may assume that $\alpha_0\in
\theta$.
If 
$\alpha_0$ is in an
irreducible component of $\varphi$ of type~$A_1$, we may assume that
$\alpha_0=\alpha_1$. If $\alpha_0$ is in an
irreducible component of $\varphi$ of rank $2$, then, as it is
a simple pseudo-root, it cannot be the first
element of $\theta$ coming from this rank $2$ factor of $\varphi$
(see Figure~\ref{figure_choice_theta} for an illustration in the
case of $I_2(8)$, the general case is similar), so
we may assume that $\alpha_0=\alpha_r$.

Suppose first that $\alpha_0$ is in an irreducible component of
$\varphi$ of rank $2$ and that
$\alpha_0=\alpha_r$.
By the description of
$2$-structures in Subsection~\ref{section_2_structures_irreducibles},
this can only happen
if $\Phi$ is of type~$B_n$, $F_4$ or $I_2(m)$
with $m$ even. 
The set $\{\alpha_1,
\ldots,\alpha_{r-1}\}$ is an ordered subset of $\varphi_0$ satisfying the
conditions of Proposition~\ref{prop_sign_2_structure}, and
$\Phi^+_{0,\theta_0}=\Phi^+_\theta\cap\Phi_0$.
So the statement of (2) will follow if we can show that
$X=(\Phi^+_{\theta}-\Phi^+_{0,\theta_0}) \cap \Phi^-$ has even cardinality.
Let $s=s_{\alpha_r}$.
We claim that $s(X)=X$ and that $s$ has no fixed points in $X$, which
implies that $X$ has even cardinality because $s^2=1$. The fact
that $s$ has no fixed point in $X$ follows from the facts
that the fixed points of $s$ are the elements of $\alpha_r^\perp$,
that $\Phi\cap\alpha_r^\perp=\Phi_0$ and that $X\cap\Phi_0=
\varnothing$. 
As $\alpha_r\not\in\Phi^-$ and $-\alpha_r\not\in\Phi^+_\theta$, we have
$X=(\Phi^--\{-\alpha_r\})\cap(\Phi^+_\theta-(\Phi^+_{0,\theta_0}
\cup\{\alpha_r\}))$.
As $\alpha_r$ is a simple pseudo-root, we have
$s(\Phi^--\{-\alpha_r\})\subset\Phi^--\{-\alpha_r\}$
by~\cite[Lemma~4.4.3]{BB}. So
it suffices to prove that
$s$ preserves $\Phi^+_\theta-(\Phi^+_{0,\theta_0}\cup\{\alpha_r\})$.
If $\beta\in\Phi^+_\theta-\Phi^+_{0,\theta_0}$ is 
such that $\beta \neq \alpha_r$, then
we cannot have $(\beta,\alpha_i)=0$ for every $i\in\{1,\ldots,r-1\}$;
indeed, as $\Phi$ is of type~$B_n$, $F_4$ or $I_2(m)$ with $m$
even, the family $(\alpha_1,\ldots,\alpha_r)$ is an orthonormal
basis of $V$, so the only element of $\Phi^+_\theta$ that is
orthogonal to $\alpha_1,\ldots,\alpha_{r-1}$ is $\alpha_r$.
So $\Phi^+_\theta-(\Phi^+_{0,\theta_0}\cup\{\alpha_r\})$ is the set
pseudo-roots $\beta\in\Phi$ such that $((\beta,\alpha_1),\ldots,
(\beta,\alpha_{r-1}))>0$ (for the lexicographic order on
$\R^{r-1}$) and that $(\beta,\alpha_r) \neq 0$. This set is stable by
$s$, because, for every $\beta\in V$, we have
$(s(\beta),\alpha_i)=(\beta,s(\alpha_i))=(\beta,\alpha_i)$ if
$1\leq i\leq r-1$ and $(s(\beta),\alpha_r)=(\beta,s(\alpha_r))=
-(\beta,\alpha_r)$.

Now we suppose that $\alpha_0$
is in an irreducible component of $\varphi$ of rank $1$ and that
$\alpha_0=\alpha_1$. Then $\theta_0=\{\alpha_2,
\ldots,\alpha_r\}$ is an ordered subset of $\varphi_0$ satisfying the
conditions of Proposition~\ref{prop_sign_2_structure}, and
$\Phi^+_{0,\theta_0}=\Phi_0\cap\Phi^+_\theta$, so
$\Phi^+_{\theta}-\Phi^+_{0,\theta_0}=\{\beta\in\Phi : (\beta,\alpha_1)>0\}$.
Statement~(2) will follow if we can show that
\[X=(\Phi^+_{\theta}-\Phi^+_{0,\theta_0}) \cap \Phi^-=\{\beta\in\Phi^-:
(\beta,\alpha_1)>0\}\] 
has even cardinality
if $\Phi$ is not of type~$A_{2n}$ or $I_2(2n'+1)$ with $n'$ odd,
and odd cardinality otherwise. 
As $\varphi$ has an irreducible component of rank $1$,
we cannot be in type~$F_4$. We can check that $X$ has even
cardinality
by a computer calculation in the exceptional
types $E$, $G$ and $H$.  

We now go through the remaining types
one by one (in cases $A$, $B$ and $D$, we use the description
of the roots from the tables at the end
of~\cite{Bourbaki}, and not the normalized pseudo-root system):
\begin{itemize}
\item[$\bullet$] \underline{Type~$I_2(m)$}: 
If $m$ is even, then $\varphi$ is a rank $2$ pseudo-root system;
so $m$ must be odd, and then $\varphi_0$ is empty
and $\epsilon(\varphi_0)=1$.
There are exactly $m$ pseudo-roots
$\beta$ such that $(\beta,\alpha_1)>0$, and $(m-1)/2$ of these
are in $\Phi^-$. So $\epsilon(\varphi)=(-1)^{(m-1)/2}(-1)^{(m-1)/2}=1$, 
which is what we wanted.

\item[$\bullet$] \underline{Type~$A_n$}: 
We write $\alpha_0=e_i-e_{i+1}$, with $1\leq i\leq n$. Then
\[X=\{e_j-e_k: 1\leq k<j=i\mbox{ or }i+1=k<j\leq n+1\}\]
has cardinality $n-1$, that is, even if and only if $n$ is odd.

\item[$\bullet$] \underline{Type~$B_n$}: 
As $\varphi$ has an irreducible component of rank $1$,
the integer $n$ must be odd and
$\alpha_0$ is the short
simple root, that is, $\alpha_0=e_n$. Then
\[X=\{-e_i+e_n:1\leq i\leq n\}\]
has cardinality $n-1$, which is even.

\item[$\bullet$] \underline{Type~$D_n$}: 
If $\alpha_0=e_i-e_{i+1}$ with $1\leq i\leq n-1$, then
\begin{align*}
X = &
\{e_j-e_k: 1\leq k<j=i\mbox{ or }i+1=k<j\leq n\} \\
& \cup \{-(e_j+e_k):i+1=j<k\leq n\mbox{ or }i \neq j<k=i+1\}
\end{align*}
has cardinality $2n-4$.
If $\alpha_0=e_{n-1}+e_n$, then
\[X=\{e_j-e_k:n=j>k \neq n-1\mbox{ or }j=n-1>k\}\]
also has cardinality $2n-4$.
\end{itemize}
\end{proof}

\begin{sublemma}
Let $\varphi \subseteq \Phi$ be a $2$-structure. Then $|\Phi^+-\varphi^+|$
is an even integer. More precisely, if $\Phi$ is irreducible, we have
\[|\Phi^+-\varphi^+|=
\begin{cases} 
2n\bmod 4 & \text{ if }\Phi\text{ is of type } A_{2n}, \\
0\bmod 4 & \text{ if }\Phi\text{ is of type } A_{2n+1}, B, D, E, F_4, G_2, \text{or } H, \\
2^r(m-1) & \text{ if } \Phi \text{ is of type } I_2(2^r m) \text{ with } m \text{ odd}.
\end{cases}\]
\label{lemma_r_Phi}
\end{sublemma}

\begin{proof}
This follows from the explicit description of $2$-structures for the
irreducible types in Subsection~\ref{section_2_structures_irreducibles}.
\end{proof}

\subsection{Orthogonal sets of pseudo-roots and 2-structures}
\label{subsection_B_3}

For $\Phi$ a root system (not just a pseudo-root system),
let $\Of(\Phi)$ be the set of all finite sequences
$(\alpha_1,\alpha_2,\ldots,\alpha_r)$
of elements of~$\Phi$ which are
pairwise orthogonal and such that their entries all have the same length,
that is, the following two conditions hold:
\begin{itemize}
\item[(a)] $(\alpha_i,\alpha_j)=0$ for all $1 \leq i < j \leq r$;
\item[(b)] $\|\alpha_1\|=\|\alpha_2\|=\cdots=\|\alpha_r\|$.
\end{itemize}

\begin{sublemma}
Suppose that $\Phi$ is a root system.
Let $\theta=(\alpha_1,\ldots,\alpha_r)$
and $\theta'=(\beta_1,\ldots,\beta_s)$ be elements of $\Of(\Phi)$,
and suppose that
$\theta^\perp\cap\Phi=(\theta')^\perp\cap\Phi=\varnothing$ and that
the elements of $\theta$ and $\theta'$ have the same length.
Then
there exists $w\in W$ such that $\{\alpha_1,\ldots,\alpha_r\}=
\{w(\beta_1),\ldots,w(\beta_s)\}$.
In particular, $r=s$ holds.

\label{lemma_one_structures}
\end{sublemma}
\begin{proof}
Let $\Phi^+_\theta$, respectively $\Phi^+_{\theta'}$, be the system
of positive roots defined by $\theta$, respectively $\theta'$, as in
Definition~\ref{def_Phi_theta}.
As $W$ acts transitively on
the set of systems of positive roots, there exists $w\in W$ such that
$w(\Phi^+_{\theta'})=\Phi^+_{\theta}$. As $w(\Phi^+_{\theta'})=
\Phi^+_{w(\beta_1),\ldots,w(\beta_s)}$, we may assume that
$\Phi^+_\theta=\Phi^+_{\theta'}$. We then wish to prove that
$\theta$ and $\theta'$ are equal up to reordering their entries.
We proceed by induction on the length of $\theta$.
If $\theta$ is empty then $\Phi$ is also empty because of
the condition $\Phi\cap\theta^\perp=\varnothing$,
so $\theta'$ is empty and we are done.
Suppose that $r\geq 1$. 
Let~$j_0$ be the smallest index $j$ such that
$(\alpha_1,\beta_j) \neq 0$.
Since $\Phi\cap(\theta')^\perp=\varnothing$,
this minimum exists.
As $\beta_{j_0}\in\Phi^+_{\theta'}=
\Phi^+_\theta$, we cannot have $(\alpha_1,\beta_{j_0})<0$, so
$(\alpha_1,\beta_{j_0})>0$. 
As $\Phi$ is a root system and not just a pseudo-root system,
the corollary after
\cite[Chapitre~VI, \S~1, \textnumero~3, Th\'eor\`eme~1]{Bourbaki}  
implies that the difference $\gamma=\alpha_1-\beta_{j_0}$
is an element of $\Phi\cup\{0\}$.
Suppose that $\gamma\in\Phi^+_{\theta'}$. 
As $(\gamma,\beta_j)=0$ for $1\leq j<j_0$, we must then have
$0 \leq (\gamma,\beta_{j_0}) = (\alpha_1,\beta_{j_0}) - (\beta_{j_0},\beta_{j_0})$.
The hypothesis states that $\|\alpha_1\|=\|\beta_{j_0}\|$
and hence we deduce that
$(\alpha_1,\beta_{j_0}) \geq \|\beta_{j_0}\|^2 = \|\alpha_1\| \cdot \|\beta_{j_0}\|$.
This inequality implies that $\alpha_1=\beta_{j_0}$,
contradicting the fact that $\gamma$ is nonzero.
Suppose that $\gamma\in \Phi^-_\theta$.
Then
$0 \leq (\alpha_1,-\gamma) =(\alpha_1,\beta_{j_0}) - (\alpha_1,\alpha_1)$,
so
$(\alpha_1,\beta_{j_0})\geq\|\alpha_1\|^2$, and again this implies that
$\alpha_1=\beta_{j_0}$ and contradicts the assumption.
Hence we conclude that $\gamma=0$, that is, $\alpha_1=\beta_{j_0}$.
Let $\Phi_0=\alpha_1^\perp\cap
\Phi=\beta_{j_0}^\perp\cap\Phi$, $\theta_0=(\alpha_2,\ldots,\alpha_r)$ and
$\theta'_0=(\beta_1,\ldots,\widehat{\beta_{j_0}},\ldots,\beta_s)$.
Then $\Phi_0$ is a root system, $\theta_0$ and $\theta_0'$ are in
$\Of(\Phi_0)$, $\theta_0^\perp\cap\Phi_0=(\theta_0')^\perp\cap\Phi_0=
\emptyset$, and $\Phi^+_{0,\theta_0}=\Phi^+_\theta\cap\Phi_0=
\Phi^+_{\theta'}\cap\Phi_0=\Phi^+_{0,\theta'_0}$. We can apply
the induction hypothesis to conclude that $\{\alpha_2,\ldots,\alpha_r\}=
\{\beta_1,\ldots,\widehat{\beta_{j_0}},\ldots,\beta_s\}$, and this
immediately implies that $\{\alpha_1,\ldots,\alpha_r\}=
\{\beta_1,\ldots,\beta_s\}$.
\end{proof}

\begin{table}
\caption{The number of orthogonal sets of roots or pseudo-roots
of size $k$
where the elements all have the same length
in the exceptional/sporadic
reflection arrangements.
Note the double occurrence of $10!$ in the $E_{8}$ column.
The equality of the columns in type~$F_4$ comes from the fact that there
is an automorphism of the underlying vector space that preserves angles,
sends short roots to long roots, and sends long roots to doubles of short
roots (for instance, the automorphism given by
$e_1 \longmapsto e_1+e_2$,
$e_2 \longmapsto e_1-e_2$,
$e_3 \longmapsto e_3+e_4$ and
$e_4 \longmapsto e_3-e_4$).}
$$
\begin{array}{r | r r r r r r r}
k &  E_{6} & E_{7} &   E_{8} & F_{4} & F_{4} & H_{3} & H_{4} \\
  &       &       &         & \text{\small short} & \text{\small long} & & \\ \hline
1 &    72 &   126 &     240 &    24 &   24  &    30 &   120 \\
2 &  1080 &  3780 &   15120 &    72 &   72  &    60 &  1800 \\
3 &  4320 & 32760 &  302400 &    96 &   96  &    40 &  2400 \\
4 &  2160 & 75600 & 1965600 &    48 &   48  &       &  1200 \\
5 &       & 90720 & 3628800 &       &       &       &       \\
6 &       & 60480 & 3628800 &       &       &       &       \\
7 &       & 17280 & 2073600 &       &       &       &       \\
8 &       &       &  518400 &       &       &       &       \\
\end{array}
$$
\label{table_the_number_of_orthogonal_sets}
\end{table}

\begin{sublemma}
Let $\Phi$ be a normalized pseudo-root system, let $\theta=(\alpha_1,
\ldots,\alpha_r)$ be a sequence of pairwise orthogonal
elements of $\Phi$ such that
$\theta^\perp\cap\Phi=\varnothing$, and let
$\theta'$ be the sequence obtained from~$\theta$ by exchanging
$\alpha_i$ and $\alpha_{i+1}$. 
Consider the subroot system $\Phi'=\Phi\cap(\R\alpha_i+\R\alpha_{i+1})$.
Then $\Phi'$ is of type $A_1\times A_1$ or $I_2(m)$ with $m\geq 4$ even, and
the parity of the cardinality of
$\Phi^+_\theta \cap \Phi^-_{\theta'}$
is given by
\[\begin{array}{rll}
|\Phi^+_\theta \cap \Phi^-_{\theta'}|
& \equiv 0 \ \bmod 2 & \text{ if }\Phi'=A_1\times A_1, \\
|\Phi^+_\theta \cap \Phi^-_{\theta'}|
& \equiv m/2-1 \ \bmod 2 & \text{ if }\Phi'=I_2(m).
\end{array}\]
\label{lemma_change_order_theta_sign}
\end{sublemma}
\begin{proof}
As $\Phi'$ is a pseudo-root system of rank $2$ (because it is contained
in a $2$-dimensional vector space and contains the two linearly
independent pseudo-roots $\alpha_i$ and $\alpha_{i+1}$), it is of type
$A_1 \times A_1$ or $I_2(m)$ with $m\geq 3$.
Moreover, $\Phi'$ contains two orthogonal
pseudo-roots, so it cannot be of type~$I_2(m)$ with $m$ odd.

We now set $C=\Phi^+_\theta \cap \Phi^-_{\theta'}$
and calculate the parity of $|C|$.
Let $\gamma\in C$. Then $\gamma$ is
orthogonal to $\alpha_1,\ldots,\alpha_{i-1}$, so we can write
$\gamma=c\alpha_i+d\alpha_{i+1}+\lambda$
with $\lambda\in\Span( \alpha_1,\ldots,\alpha_{i+1})^\perp$
and $c\alpha_i+d\alpha_{i+1} \neq 0$.
Set $\iota(\gamma)=-
s_{\alpha_i}s_{\alpha_{i+1}}(\gamma)$. Then $\iota(\gamma)\in\Phi$ and
$\iota(\gamma)=c\alpha_i+d \alpha_{i+1}-\lambda$, so $\iota(\gamma)\in
C$. Also, we clearly have $\iota(\iota(\gamma))=\gamma$, and $\iota(\gamma)$
is equal to $\gamma$ if and only if $\lambda=0$, that is, if and only
if $\gamma\in\Phi'$. We have
defined an involution $\iota$ of $C$, and we conclude that
$|C| \equiv |C_0| \bmod 2$, where $C_0 = \Phi'\cap C$
is the set of fixed points of $\iota$ in $C$. If $\Phi'$ is of type
$A_1\times A_1$ then we easily see that $C_0$ is empty, so we
are done. Suppose that $\Phi'$ is of type~$I_2(m)$ with $m$ even.
Let $\gamma=c\alpha_i+d\alpha_{i+1}\in\Phi'$, with $c,d\in\R$.
Then $\gamma\in C$ if and only if $c>0$ and $d<0$. The set $C_0$ contains
exactly one quarter of the elements
of~$\Phi'-\{\pm\alpha_i,\pm\alpha_{i+1}\}$, that is,
$|C_0| = (2m-4)/4 = m/2-1$.
\end{proof}

\begin{subrmk}
If we view the root system $A_1 \times A_1$ as the
dihedral pseudo-root system $I_2(2)$ then the conclusion
of Lemma~\ref{lemma_change_order_theta_sign}
is that
$|\Phi^+_\theta \cap \Phi^-_{\theta'}|
\equiv m/2-1  \bmod 2$ if $\Phi'=I_2(m)$
with $m$ even and $m \geq 2$.
\label{remark_change_order_theta_sign}
\end{subrmk}

\begin{sublemma}
Suppose that $\Phi$ is an irreducible root system
(not just a pseudo-root system)
and not of type~$G_2$.
Let $\Phi^+$ be a system of positive roots of $\Phi$
and
let $\varphi \subseteq \Phi$ be a $2$-structure. 
Define a subset $\theta$ of $\varphi$ as in
Proposition~\ref{prop_sign_2_structure}.
Then there is a choice of the sequences $\theta_i$ for which
$\theta$ is an element of~$\Of(\Phi)$.
Moreover,
if~$\Phi$ is of type~$B_n$ or~$F_4$ we can choose $\theta$ to consist of short roots.
Similarly, 
if~$\Phi$ is of type~$C_n$ or~$F_4$ we can choose $\theta$ to consist of long roots.
\label{lemma_choice_of_theta}
\end{sublemma}
\begin{proof}
By Remark~\ref{rmk_no_orthogonal_root} we have
$\theta^\perp\cap\Phi=\varnothing$.
We use the notation of Proposition~\ref{prop_sign_2_structure}.
If all the roots of $\Phi$ have the same length (which is
the case for $A_n$, $D_n$, $E_6$, $E_7$ and $E_8$), then there is nothing
to prove. 
Note also that if $\varphi_i$ is an actual root system of type~$B_2$
(that is, with the correct root lengths), then the two possible choices for
$\theta_i$ are the set of short positive roots and the set of
long positive roots.

Suppose that $\Phi$ is of type~$B_n$.
If $\varphi$ has no irreducible component of type~$A_1$, then
we choose the two short positive
roots in each $\varphi_i$.
Suppose that $\varphi$ has a factor of type~$A_1$.
We show that this factor cannot contain long roots.
Suppose on the contrary that this occurs.
Without loss of generality,
we may assume that $\varphi_1=\{\pm (e_1+e_2)\}$.
The rank $2$ factors of $\varphi$ cannot contain
$e_1-e_2$, so they are all in $e_1^\perp\cap e_2^\perp$.
All the rank $1$ factors that do not contain $e_1-e_2$ must also be
in $e_1^\perp\cap e_2^\perp$.
If $e_1-e_2$ were not in $\varphi$ then the reflection
$s_{e_1-e_2}$ would act as the identity on all the elements on~$\varphi$,
which contradicts the definition of a $2$-structure.
Hence
$\{\pm(e_1-e_2)\}$ is another rank $1$ factor of~$\varphi$. But then
the reflection $s_{e_1}$ preserves $\varphi^+$, which is impossible.
Hence all the $A_1$ factors of $\varphi$ contain only short roots,
and we choose the $\theta_i$ in the $B_2$ factors to contain the two
short positive roots.

The case of $C_n$ is similar, with the roles of short and long
roots uniformly exchanged.

Finally suppose that $\Phi=F_4$.
In this case we can similarly show that the $2$-structure $\varphi$
has type~$B_{2}^{2}$, allowing us to pick
either short or long roots in each factor.
\end{proof}

\subsection{2-structures in the irreducible types}
\label{section_2_structures_irreducibles}

In this subsection we prove
Proposition~\ref{prop_W_acts_transitively_on_T_Phi},
that is, the fact that
the group $W$ acts transitively on the collection of $2$-structures $\Tt(\Phi)$.
It is enough to prove this result for irreducible pseudo-root systems.
We proceed by a case by case analysis.

\subsubsection*{Types $A_n$, $D_n$, $E_6$, $E_7$ and $E_8$}
Suppose that $\Phi$ is a root system of type~$A_n$, $D_n$ or
$E_m$ with $m\in\{6,7,8\}$.
As all the roots of $\Phi$ have the same length and as $\Phi$
contains no $B_2$ root system, the $2$-structures for $\Phi$ are exactly
the maximal sets $\varphi=\{\pm\alpha_1,\ldots,\pm\alpha_r\}$ such that
$(\alpha_1,\ldots,\alpha_r)\in\Of(\Phi)$. By
Lemma~\ref{lemma_one_structures}, for any
$(\alpha_1,\ldots,\alpha_r)$ and $(\beta_1,\ldots,\beta_s)$ on
$\Of(\Phi)$, there exists $w\in W$ such that
$\{\alpha_1,\ldots,\alpha_r\}=\{w(\beta_1),\ldots,w(\beta_s)\}$.
Hence the group $W$ acts transitively on $\Tt(\Phi)$.
In particular, all
the $2$-structures for $\Phi$ are isomorphic, so we can determine
their type.
See Table~\ref{table_2-structure_in_A_D_E}.

\begin{table}
\caption{The $2$-structures in types $A$,~$D$ and $E$
where $m=\lfloor (n+1)/2 \rfloor$ in type~$A$ and $m=\lfloor n/2 \rfloor$
in type~$D$.}
$$
\begin{array}{c | l l}
\text{Type of root} & & \text{Type of} \\
\text{system } \Phi & \text{$2$-structures are isomorphic to} & \text{$2$-structure} \\ \hline
A_n &
\{\pm(e_1-e_2), \:\:\:
\pm(e_3-e_4), \:\:\: \ldots, \:\:\:
\pm(e_{2m-1}-e_{2m})\} &
A_1^m \\ 
D_n &
\{\pm e_1 \pm e_2, \:\:\:
\pm e_3 \pm e_4, \:\:\: \ldots, \:\:\:
\pm e_{2m-1} \pm  e_{2m}\} &
A_1^{2m} \\
E_6 & 
\{\pm e_1 \pm e_2, \:\:\:
\pm e_3 \pm e_4\} &
A_1^{4} \\
E_7 &
\{\pm e_1 \pm e_2, \:\:\:
\pm e_3 \pm e_4, \:\:\:
\pm e_5 \pm e_6, \:\:\:
\pm(e_7-e_8)\} &
A_1^{7} \\
E_8 & 
\{\pm e_1 \pm e_2, \:\:\:
\pm e_3 \pm e_4, \:\:\:
\pm e_5 \pm e_6, \:\:\:
\pm e_7 \pm e_8\} &
A_1^{8}
\end{array}
$$
\label{table_2-structure_in_A_D_E}
\end{table}

\subsubsection*{Types $B_n$ and $C_n$}
Suppose that $\Phi$ is a root system of type~$B_n$. This will also give
the type~$C_n$ case, since $B_n$ and $C_n$ correspond to
the same Coxeter system.
We claim that $W$ acts transitively on $\Tt(\Phi)$.
In particular, all the
$2$-structures for $\Phi$ are isomorphic to
\[\varphi_0=\begin{cases}
\{\pm e_1, \pm e_2, \pm e_1\pm e_2\} \sqcup 
\cdots \sqcup
\{\pm e_{2m-1},\pm e_{2m},\pm e_{2m-1}\pm e_{2m}\}
\!\!\!\!& \text{if } n=2m, \\
\{\pm e_1, \pm e_2, \pm e_1\pm e_2\} \sqcup 
\cdots \sqcup
\{\pm e_{2m-1},\pm e_{2m},\pm e_{2m-1}\pm e_{2m}\}
\sqcup \{\pm e_{2m+1}\}
\!\!\!\!& \text{if } n=2m+1,
\end{cases}\]
so they are of type~$B_2^{m}$ if $n=2m$ is even, and of type
$B_2^{m}\times A_1$ if $n=2m+1$ is odd.

We prove the claim by induction on $n$. The case $n=1$ is clear.
Suppose that $n\geq 2$.
Let $\varphi,\varphi'\in\Phi$. By
Lemma~\ref{lemma_choice_of_theta}, we can choose sequences
$\theta$ of $\varphi$ and $\theta'$ of $\varphi'$ as
in Proposition~\ref{prop_sign_2_structure} such that
$\theta,\theta'\in\Of(\Phi)$ and that these subsets contain only short roots.
By Lemma~\ref{lemma_one_structures},
we may assume that $\theta$ and $\theta'$ coincide up to the order
of their elements.
Denote by $\varphi=\varphi_1\sqcup\cdots\sqcup\varphi_s$ and
$\varphi'=\varphi'_1\sqcup\cdots\sqcup\varphi'_t$ the decomposition
into irreducible systems that gave rise to $\theta$ and $\theta'$.
We can always change the order on the $\varphi_i$ and the
$\varphi'_j$.

Suppose that $\varphi_1$ is of rank $1$, so that
$\varphi_1=\{\pm\alpha_1\}$. We may assume that $\alpha_1\in\varphi'_1$.
If $\varphi'_1$ is of rank $1$ then
$\varphi'_1=\varphi_1$. 
As $\Phi\cap\varphi_1^\perp$ is an irreducible root system of
type~$B_{n-1}$,
the conclusion follows by the induction hypothesis.

If $\varphi'_1$ is of rank $2$ then $\varphi_1'$ is a $B_2$ root
system whose short positive roots are $\alpha_1$ and some~$\alpha_2$,
and we may assume that $\alpha_2\in\varphi_2$.
In particular, $\beta=\alpha_1-\alpha_2\in\Phi$. If $\varphi_2=\{\pm\alpha_2\}$
then the reflection~$s_\beta$ preserves $\varphi^+$, which is not possible.
So $\varphi_2$ is of rank $2$ (in particular, $n\geq 3$),
which means that it is a $B_2$ root
system whose short roots are $\alpha_2$ and some $\alpha_3$.
We may assume that $\alpha_3\in\varphi'_2$.
In particular, $\alpha_2-\alpha_3\in\Phi$, so
$\gamma=s_\beta(\alpha_2-\alpha_3)=\alpha_1-\alpha_3$ is also a root.
The irreducible components of $s_\gamma(\varphi)$ are
$\varphi'_1,\{\pm\alpha_3\},\varphi_3,\ldots,\varphi_s$. As
$\Phi\cap(\varphi_1')^\perp$ is a root system of type~$B_{n-2}$, the
induction hypothesis implies that there is a $w\in W$ such that
$w(\varphi')=s_\gamma(\varphi)$, which finishes the proof
in this case.

Suppose that $\varphi_1$ is of rank $2$, and call its
other short positive root $\alpha_2$. We may assume that
$\alpha_1\in\varphi'_1$. If $\varphi'_1$ is of rank $1$ then
$\varphi'_1=\{\pm\alpha_1\}$, and we can repeat the reasoning of the
previous paragraph with the roles of $\varphi$ and $\varphi'$ exchanged.
If $\varphi'_1=\varphi_1$ then the conclusion follows from the
induction hypothesis applied to the $B_{n-2}$ root system
$\varphi_1^\perp\cap\Phi$. Finally, suppose that $\varphi'_1$ is of
rank~$2$ and $\varphi'_1 \neq \varphi_1$. Let $\alpha_3$ be the
other short positive root of $\varphi'_1$. As $\alpha_2$ and $\alpha_3$
are both short roots, $\beta=\alpha_2-\alpha_3\in\Phi$. Note
that the irreducible components of $s_\beta(\varphi)$
are $\varphi'_1,s_\beta(\varphi_2),\ldots,s_\beta(\varphi_s)$, so
again the induction hypothesis implies that there exists
$w\in W$ such that $s_\beta(\varphi)=w(\varphi')$, and we are done.

\subsubsection*{Type~$F_4$}
Suppose that $\Phi$ is a root system of type~$F_4$.
Then we can show that $W$ acts transitively on $\Tt(\Phi)$ exactly as in
type~$B_n$. In particular, any $2$-structure is isomorphic to
$\varphi_0
=
\{\pm e_1,\pm e_2,\pm e_1\pm e_2\} \sqcup \{\pm e_3,\pm e_4,
\pm e_{3}\pm e_{4}\}$, so it is of type~$B_2^2$.

\subsubsection*{Dihedral types}
Suppose that $\Phi$ is a pseudo-root system of type~$I_2(m)$ with
$m\geq 5$ (this includes the type~$G_2$ root system).
It is straightforward to see
that $W$ acts transitively on $\Tt(\Phi)$.
If $m$ is odd then all the $2$-structures for $\Phi$ are
isomorphic to $\varphi_0=\{\pm e_1\}$,
and in particular of type~$A_1$.
If $m$ is even then all the
$2$-structures for $\Phi$ are of type~$I_2(2^r)$, where $2^r$ is the largest
power of $2$ dividing $m$.

\subsubsection*{Types $H_3$ and $H_4$}

Suppose that $\Phi$ is of type~$H_3$ or $H_4$.
We use the description of the pseudo-root systems $H_3$ and~$H_4$
given in~\cite[Table~5.2]{GB} where they are called $I_3$ and~$I_4$.
In particular, we choose $\Phi$ to be normalized. 
We claim that $W$ acts transitively on $\Tt(\Phi)$, and
so every $2$-structure for $\Phi$ is isomorphic to
\begin{align}
\varphi_0=\begin{cases}
\{\pm e_1,\pm e_2,\pm e_3\} & \text{ if } \Phi=H_3,\\
\{\pm e_1,\pm e_2,\pm e_3,\pm e_4\} & \text{ if }\Phi=H_4,
\end{cases}
\label{equation_type_H_2-structure}
\end{align}
and in particular it is of type~$A_1^3$ if $\Phi=H_3$
and of type~$A_1^4$ if $\Phi=H_4$.

It is clear by the chosen description of $\Phi$
that all of the inner products of elements of $\Phi$ are
in~$\Q[\sqrt{5}]$, and in particular $1/\sqrt{2}$ never appears.
So there are no pseudo-roots in $\Phi$ with an angle of ${\pi}/{4}$ between
them, which implies that $\Phi$ does not contain any pseudo-root system
of type~$I_2(m)$ with $m$ a multiple of $4$, and so $2$-structures for $\Phi$
(if they exist) can only have irreducible components of type~$A_1$.

We check easily that the set $\varphi_0$ given in 
equation~\eqref{equation_type_H_2-structure}
is a $2$-structure, so it remains to show that all the maximal
sets of pairwise orthogonal pseudo-roots are conjugate under~$W$
to $\zeta_0$, where $\zeta_0=\{e_1,e_2,e_3\}$ if $\Phi=H_3$ and 
$\zeta_0=\{e_1,e_2,e_3,e_4\}$ if $\Phi=H_4$.
Any element of the stabilizer $W_0$ of $\zeta_0$ in $W$ must act on
$\Span(\Phi)$ by a permutation of the coordinates, and it must be
an even permutation to be in $W$. This implies that the cardinality of $W_0$
is $3$ for $\Phi=H_3$ and $12$ for $\Phi=H_4$. Using a computer,
it is not hard to count all the maximal sets of pairwise orthogonal
pseudo-roots in $H_3$ and $H_4$ 
(see Table~\ref{table_the_number_of_orthogonal_sets}).
We find that
there are $40$ such sets for $H_3$ and $1200$ such sets for $H_4$.
In both cases, this number is equal to $|W|/|W_0|$, so
$W$ does act transitively on the set of maximal sets of
pairwise orthogonal pseudo-roots, and hence also on $\Tt(\Phi)$.

\section{Relationship with locally symmetric spaces}
\label{appendix_more_detail}

In this appendix, aimed at specialists of Shimura varieties,
we give more details about the connection
between some of the
objects introduced in this article and the calculation of the
weighted cohomology of locally symmetric spaces. 

This is a continuation of the discussion in the first part of
the introduction, and we return
to the notation of this discussion.
We do not suppose yet that the group $G(\Rrr)$ has a discrete series.
In the introduction, we only considered cohomology of
$X_K$ with constant coefficients, but now we need to introduce
a coefficient system. Let F be an irreducible algebraic representation
of $G$. Then, via the $G(\Qqq)$-covering 
$(G(\Rrr)\times G(\mathbb{A}^\infty))/(K_\infty\times K)\rightarrow X_K$
and the action of $G(\Qqq)$ on F, we get a locally constant
sheaf $L_F$ on $X_K$, and we will write $H^*(X_K,F)$ instead of
$H^*(X,L_F)$ for every reasonable cohomology theory $H^*$.
\footnote{We need a different construction of $L_F$ if $X_K$ is the
set of complex points of a Shimura variety and $H^*$ is \'etale
(intersection) cohomology, but this is not the point of this appendix.}
If $T\subset B$ are a maximal torus and a Borel subgroup
of $G_\Ccc$ respectively,
then the representation $F$ has a highest
weight $\lambda_B$ in the Lie algebra of $T$ that is dominant with respect
to $B$.
\footnote{The representation $F$ might not stay irreducible when seen
as a representation of $G_\Ccc$, but we ignore this technical complication.}
The space $V$ of the article will typically be this Lie algebra
with the inner product coming from the Killing form of the Lie algebra
of $G$ in the usual way.
The pseudo-root system $\Phi$ that defines the hyperplane arrangement
will be the root system of $T$ in the Lie algebra of $G$,
with the positive system determined by $B$; sometimes $T$ will be defined
over $\Rrr$, and $\Phi$ will be the real root system of $T$.

The first cohomology theory that we consider is
\emph{weighted cohomology}, from which the weighted complex and
the weighted sum get their names. Weighted cohomology was
introduced by Goresky--Harder--MacPherson in the paper~\cite{GHM}.
It depends on an auxiliary parameter called a ``weight profile'' and
is the cohomology of a sheaf of truncated differential forms
on the reductive Borel-Serre compactification of $X_K$, where the
truncation depends on the weight profile. The Hecke algebra acts on
the weighted cohomology groups, and they are explicit enough to make
the calculation of the traces of Hecke operators possible; see the
paper~\cite{GM} of Goresky and MacPherson. Also, there are two ``middle''
weight profiles for any group $G$ and, if $X_K$ is a Shimura variety,
then the two middle weighted cohomology groups are both isomorphic to the
intersection cohomology of
the Baily-Borel compactification of $X_K$.
What we call the ``weighted sum'' in this article appears in the
calculation of the trace of a Hecke operator on the
weighted cohomology groups, hence the name. This calculation
is carried out in \cite{GM} and summarized in Section~7 of
\cite{GKM}. Very roughly, the trace of a Hecke operator on
a weighted cohomology group is a sum over conjugacy classes
of rational Levi subgroups $M$ of $G$ and certain conjugacy
classes of $\gamma\in M(\Qqq)$ of the product of:
\begin{itemize}
\item[--] a normalizing factor;
\item[--] an orbital integral on the conjugacy class of $\gamma$
in $M(\mathbb{A}^\infty)$ that depends on the Hecke operator
but not on the weight profile or on the coefficient system;
\item[--] a ``term at infinity'' $L_M(\gamma)$ that depends on the weight
profile and the coefficient system but not on the Hecke operator.
\end{itemize}
See formula~(7.14.7) of~\cite{GKM}. 
Goresky, Kottwitz and MacPherson then introduce a stable virtual character
$\Theta$ on $G(\Rrr)$ (this notion is defined on page~495 of~\cite{GKM})
that depends on the coefficient system via the highest weight
of $F$ and on the weight profile. We can
recover the function $L_M$ as the restriction of $\Theta$ to $M$
up to chasing some denominators 
depending on~$M$; this last statement is Theorem~5.1
of~\cite{GKM}, and it works for any weight profile. 
While the expression for the function $L_M$ involves ``relative''
weighted sums $\psi_{\Hf/C}$ where we are 
in the situation of
Example~\ref{ex_parabolic_arrangement} (see pages~504--505 of~\cite{GKM}), 
the virtual character $\Theta$ only involves the simpler
weighted sums $\psi_\Hf$, where we are in the situation
of Subsection~\ref{section_Coxeter_case}. In both cases,
the space $V$ is the real Lie algebra of a maximal torus $T$ of $G$,
the pseudo-root system is the set of real roots of $T$ in $G$, and the
element $\lambda$ of $V$ is, up to a shift depending on the weight
profile, of the form $w(\lambda_B+\rho_B)-\rho_B$, where $B\supset T$
is a Borel subgroup (defined over $\Ccc$), 
$\lambda_B$ is the highest weight of~$F$ corresponding
to $B$, $\rho_B$ is half the sum of the positive roots and $w$ is an
element of the Weyl group.

We now assume that the weight profile is one of the middle profiles and
that $X_K$ is a Shimura variety.
Then, as explained in the introduction, we know that weighted cohomology
is isomorphic to $L^2$ cohomology, for which we have a spectral
description known as Matsushima's formula (even though it was proved
by Borel and Casselman in this generality). This implies in particular
that the virtual character $\Theta$ is equal to the stable discrete
series character corresponding to the dual of the representation $F$,
hence that the weighted sum $\psi_\Hf$ is equal to what are
known as \emph{stable discrete series constants}; see for example pages~493
and 498--500 of~\cite{GKM} for a quick review of these constants.
The first statement of the previous sentence is proved directly
in Theorem~5.2 of~\cite{GKM}, and the second statement is proved
directly in Theorem~3.1 of the same paper. 
The stable discrete series constants can be expressed in terms
of $2$-structures by the work of Herb (see for 
example~\cite[Theorem~4.2]{Herb-2S}), and this is the expression
on the right-hand side of the identity of Corollary~\ref{cor_root_systems}.

We can go further and relate the traces of Hecke operators on
$L^2$ cohomology to the Arthur-Selberg trace formula for a particular
test function. This is
done in Arthur's paper \cite{A-L2}. The resulting trace formula
can then be stabilized. Although this is a very complicated process
in the general case, it is slightly less involved for our test
function, by the work of Kottwitz (unpublished) and Zhifeng
Peng (\cite{Peng}). Thus we get character
formulas relating the virtual character $\Theta$ and stable discrete
series characters on endoscopic groups of $G$. However, when we
express everything in terms of $2$-structures, the distinction between
$G$ and its endoscopic groups disappears. Indeed, endoscopic groups of $G$
have root systems that are subsystems of the root system of $G$,
and $2$-structures, being very small root systems, can be shared between
$G$ and its endoscopic groups. (We are summarily ignoring many complications,
due in particular to the appearance of transfer factors in the character
identities.)

We finally come to the case where $X_K$ is a Shimura variety defined
over some number field $E$ and
we are interested, not just in the action of the Hecke algebra
on the intersection cohomology $IH^*(\overline{X}_K,F)$
of its Baily-Borel compactification
$\overline{X}_K$, but also in the action of the absolute Galois group
of $E$. There is a calculation of the trace of a Hecke operator
times a power of the Frobenius morphism (at at unramified place $p$) that
parallels the calculation of~\cite{GM}: see \cite{Morel1} for
the algebraic version of weighted cohomology, the papers
\cite{Morel2} and \cite{Morel} for the trace calculation in
the cases of unitary and symplectic groups (over $\Qqq$), and
\cite{Zhu} for the trace calculation in the case of orthogonal groups.
We obtain an expression for this trace that is reminiscent of
formula~(7.14.7) of~\cite{GKM}, that we quickly described above,
except that the orbital integral at $p$ is twisted and that the
terms $L_M(\gamma)$ are slightly different. Nevertheless, by
using techniques similar to those of the proof of Theorem~5.1
of~\cite{GKM}, in particular the Weyl character formula and Kostant's theorem, 
we can
still relate $L_M(\gamma)$ to the relative weighted sum
$\psi_{\Hf/C}$ in the situation of
Example~\ref{ex_parabolic_arrangement}. For symplectic groups
over $\Qqq$, this calculation is done
in the proof of Proposition~3.3.1 of~\cite{Morel}. The difference
with the situation of~\cite{GKM} is that $L^2$ cohomology does not
have an action of the absolute Galois of $E$, so we do not have a nice spectral
expression for our trace, and in particular we do not know if there
is a stable virtual character ``interpolating'' the function
$L_M$ as in Theorem~5.1 of~\cite{GKM}. Fortunately, we are still
able to relate our trace expression directly to a sum of stable trace
formulas for well-chosen test functions on endoscopic groups of $G$,
and this is where Theorem~\ref{thm_second_main} comes into play:
We must express the function $L_M$ in terms
of stable discrete series constants for endoscopic groups of $G$. Via
Herb's formula, this reduces to giving a formula for $L_M$ involving
$2$-structures for the root systems of these endoscopic groups, but, as
we explained above, these $2$-structures can also be seen as
$2$-structures for the root system of $G$. Again, we are sweeping
many technical complications under the rug, and the story is by no means
finished once we have Theorem~\ref{thm_second_main}.

\section*{Acknowledgements}

We thank the anonymous referees for their helpful comments.
The greatest part of this work was written
while the second author was a professor at Princeton
University.
It was also partially 
supported by the LABEX MILYON (ANR-10-LABX-0070) of Universit\'e 
de Lyon, within the program ``Investissements d'Avenir'' (ANR-11-IDEX-0007)
operated by the French National Research Agency (ANR). More precisely,
the authors would like to thank 
the \'Ecole Normale Sup\'erieure de Lyon (\'ENS de Lyon)
for its hospitality and support to the second author during
the academic year 2017--2018, and to the first and third author
during one week visits.
The first and third authors also thank 
Princeton University
for hosting four one-week visits during the academic year 2018--2019,
and the Institute for Advanced Study for hosting
a research visit Summer 2019,
and the second author thanks the University of
Kentucky for its hospitality during a one-week visit in the Fall of
2019.
This work was also partially supported by grants from the
Simons Foundation
(\#429370 to Richard~Ehrenborg
and \#422467 to Margaret~Readdy). 

Finally, we used SageMath and Maple for innumerable root system
computations.

\printbibliography

\end{document}